\documentclass[11pt]{amsart}
\usepackage{fullpage,verbatim,amssymb, amsmath, mathdots}
\usepackage[active]{srcltx}
\usepackage{hyperref}
\usepackage{enumerate}
\usepackage[usenames,dvipsnames]{color}
\usepackage{bm}
\usepackage{aliascnt}

\definecolor{blue}{rgb}{0,0,1}
\definecolor{red}{rgb}{1,0,0}
\definecolor{green}{rgb}{0,.6,.2}
\definecolor{purple}{rgb}{1,0,1}
\long\def\red#1\endred{\textcolor{red}{#1}}
\long\def\blue#1\endblue{\textcolor{blue}{#1}}
\long\def\purple#1\endpurple{\textcolor{purple}{ #1}}
\long\def\green#1\endgreen{\textcolor{green}{#1}}

\def\scrA{{\mathcal A}}
\def\scrB{{\mathcal B}}
\def\scrC{{\mathcal C}}
\def\scrD{{\mathcal D}}
\def\scrF{{\mathcal F}}
\def\scrR{{\mathcal R}}

\def\fH{{\mathfrak H}}
\def\fS{{\mathfrak S}}

\newcommand{\sm}{\left(\begin{smallmatrix}}
\newcommand{\esm}{\end{smallmatrix}\right)}
\newcommand{\bpm}{\begin{pmatrix}}
\newcommand{\ebpm}{\end{pmatrix}}
\def\M{\operatorname{M}}
\newcommand{\matr}[4]{\left( \begin{matrix} #1 & #2 \\ #3 & #4 \end{matrix} \right) }
\newcommand{\cmatr}[2]{\left( \begin{matrix} #1 \\ #2 \end{matrix} \right) }
\newcommand{\scmatr}[2]{\left( \begin{smallmatrix} #1 \\ #2 \end{smallmatrix} \right) }

\newcommand{\Z}{{\mathbb Z}}

\newcommand{\R}{{\mathbb R}}
\newcommand{\C}{{\mathbb C}}
\newcommand{\F}{{\mathbb F}}

\newcommand{\bs}{\backslash}
\newcommand{\tfS}{\widetilde{\fS}}
\newcommand{\dd}{\mathsf{d}}
\newcommand{\GaG}{\Gamma\bs{\rm G}}
\newcommand{\og}{\overline{g}}
\newcommand{\hF}{{\widehat{F}}}

\newcommand{\tJ}{{\widetilde{J}}}
\newcommand{\tmu}{{\widetilde{\mu}}}
\newcommand{\tOmega}{{\widetilde{\Omega}}}
\newcommand{\tN}{{\widetilde{N}}}
\newcommand{\tn}{{\widetilde{n}}}

\newcommand{\ve}{\varepsilon}
\newcommand{\col}{\: : \:}
\renewcommand{\mod}{\:\operatorname{mod}\:}
\newcommand{\bn}{{\text{\boldmath$0$}}}
\newcommand{\oF}{\overline{\mathbb{F}}}
\newcommand{\oJ}{\overline{J}}

\def\diag{\operatorname{diag}}

\def\SL{\operatorname{SL}}
\def\ASL{\operatorname{ASL}}
\def\GL{\operatorname{GL}}
\def\new{\operatorname{new}}
\def\rank{\operatorname{rank}}
\def\Span{\operatorname{Span}}
\def\ord{\operatorname{ord}}

\def\tre{\operatorname{Re}}
\def\vol{\operatorname{vol}}

\newtheorem{theorem}{Theorem}[section]
\newaliascnt{lemma}{theorem} 
\newtheorem{lemma}[lemma]{Lemma} 
\aliascntresetthe{lemma}

\newaliascnt{proposition}{theorem} 
\newtheorem{proposition}[proposition]{Proposition} 
\aliascntresetthe{proposition}

\newaliascnt{corollary}{theorem} 
\newtheorem{corollary}[corollary]{Corollary} 
\aliascntresetthe{corollary}

\newaliascnt{problem}{theorem} 
\newtheorem{problem}[problem]{Problem} 
\aliascntresetthe{problem}

\theoremstyle{remark}

\newtheorem{remark}[theorem]{Remark}

\numberwithin{theorem}{section}
\numberwithin{equation}{section}

\def\vecnull{{\text{\boldmath$0$}}}
\def\veca{{\text{\boldmath$a$}}}
\def\vecb{{\text{\boldmath$b$}}}
\def\vecc{{\text{\boldmath$c$}}}
\def\vece{{\text{\boldmath$e$}}}
\def\vecn{{\text{\boldmath$n$}}}

\def\vecv{{\text{\boldmath$v$}}}
\def\vecw{{\text{\boldmath$w$}}}
\def\vecx{{\text{\boldmath$x$}}}

\def\trans{\,^\mathrm{t}\!}

\DeclareMathOperator{\tr}{tr}
\DeclareMathOperator{\Gr}{Gr}

\title{Effective equidistribution of primitive rational points on expanding horospheres}

\author{Daniel El-Baz}
\address{Institute of Analysis and Number Theory, TU Graz, Steyrergasse 30, 8010 Graz, Austria}
\email{\tt daniel.elbaz.88@gmail.com}
\author{Min Lee}
\address{School of Mathematics, University of Bristol, Bristol BS8 1TW, U.K.}
\email{\tt min.lee@bristol.ac.uk}
\author{Andreas Str\"ombergsson}
\address{Department of Mathematics, Uppsala University, Box 480, SE-75106, Uppsala, Sweden}
\email{\tt astrombe@math.uu.se}

\date{\today}

\begin{document}
\thanks{We are grateful to \'Arp\'ad T\'oth and M\'arton Erd\'elyi for sharing their preprint on the matrix Kloosterman sum early with us and several conversations. 
We are also grateful to Igor Shparlinski for making us aware of his paper with Ahmadi, \cite{AhShS2007}.
We would further like to express our thanks for a Heilbronn Focused Research Grant and the hospitality of the Heilbronn institute in Bristol. 
D.E.\ is supported by the Austrian Science Fund (FWF), Projects P-34763 and Y-901.
M.L.\ is supported by a Royal Society University Research Fellowship. 
A.S.\ is supported by the Knut and Alice Wallenberg Foundation}
\maketitle

\begin{abstract}
We prove an effective version of a result
due to Einsiedler, Mozes, Shah and Shapira
on the asymptotic distribution of primitive rational points on
expanding closed horospheres in the space of lattices. 
Key ingredients of our proof include recent bounds on matrix Kloosterman sums
due to Erd\'elyi and T\'oth,
results by Clozel, Oh and Ullmo on the effective equidistribution of Hecke points,
and Rogers' integration formula in the geometry of numbers.
As an application of the main theorem, we also obtain a 
result on the limit distribution of the number of 
small solutions of a random system of linear congruences to a large modulus.
Furthermore, as a by-product of our proofs, %
we obtain a sharp bound on the number of nonsquare matrices over a finite field $\F_p$
with small entries 
and of a given size and rank.
\end{abstract}

\tableofcontents

\section{Introduction}

\subsection{Setup}
Let $1\leq n \leq d$, ${\rm G} = {\rm SL}_{d+n}(\mathbb{R})$ and $\Gamma={\rm SL}_{d+n}(\mathbb{Z})$.
Our discussion will take place in %
the homogeneous space $\Gamma\bs{\rm G}$.
We will often view an element $g\in {\rm G}$ as a block matrix,
$g = \bpm A & B\\ C& D\ebpm$,
where $A,B,C,D$ are real matrices of dimensions $d\times d$,
$d\times n$, $n\times d$ and $n\times n$, respectively.
In particular, for $V\in{\rm M}_{n\times d}(\mathbb{R})$
(that is, $V$ being a real matrix of dimension $n\times d$), let us write
\begin{align}
& %
n_+(V):= \bpm I_d & V\\ \vecnull & I_n\ebpm\in {\rm G}.
\end{align}
Also, for $y>0$, let
\begin{equation}\label{Dydef}
D(y) := \bpm y^{-\frac{n}{d}} I_d & \vecnull \\ \vecnull & y I_n\ebpm \in {\rm G}.
\end{equation}
For each $V\in\M_{d\times n}(\R)$, the point $\Gamma\, n_+(V)$ in $\GaG$
depends only on $V\mod\M_{d\times n}(\Z)$;
hence the map $V\mapsto\Gamma\, n_+(V)$ factors through a map %
\begin{align}\label{tnpDEF}
\widetilde{n}_+:\M_{d\times n}(\R/\Z) %
\to\Gamma\bs {\rm G}.
\end{align}
In fact $\tn_+$ is a smooth embedding of the 
$dn$-dimensional torus 
$\M_{d\times n}(\R/\Z)$;
its image is a closed horosphere in 
$\Gamma\backslash {\rm G}$,
which we call $\fH_1$.
More generally, let $\fH_y$ be $\fH_1$ translated by $D(y)$:
\begin{align*}
\fH_y=\fH_1D(y)=\bigl\{\tn_+(V)D(y)\col V\in\M_{d\times n}(\R/\Z)\bigr\}.
\end{align*}
These $\fH_y$ form a family of closed horospheres in $\GaG$,
which expand  %
as $y$ increases. %
It is well-known that as $y\to\infty$,
the $\fH_y$ become equidistributed in $\GaG$
with respect to the ${\rm G}$-invariant probability measure.

Our main object of study is a very special finite subset of 
the closed horosphere $\fH_y$, appearing when $y$ is an integer.
To describe this set,
let ${\rm H}$ be the following subgroup of ${\rm G}$:
\begin{equation}
{\rm H} = \left\{\bpm A & \bn \\ U & I_n\ebpm : A\in {\rm SL}_d(\mathbb{R}),\:
 U\in {\rm M}_{n\times d}(\mathbb{R}),\:
 A=I_n \text{ if } n=d\right\}.
\end{equation}
Then $\Gamma\bs\Gamma{\rm H}$ is a closed embedded submanifold of $\GaG$, and $\Gamma\bs\Gamma{\rm H}$ has the structure of a torus fiber bundle over $\SL_d(\Z)\bs\SL_d(\R)$.
Let $\fS_y$ be the \textit{intersection} of $\Gamma\bs\Gamma{\rm H}$ and $\fH_y$.
\begin{lemma}\label{fScharLEM}
The set $\fS_y$ is empty unless $y$ is an integer.
For $y=q$ a positive integer,
the set $\fS_q$ consists exactly of the points
$\tn_+(q^{-1}R)D(q)$ where
$R$ runs through all matrices in
$\M_{d\times n}(\Z/q\Z)$ with the property that
the rows of $R$ generate $(\Z/q\Z)^n$.
\end{lemma}
(Here, naturally, $\M_{d\times n}(\Z/q\Z)$ denotes the group of $d\times n$ matrices with
entries in $\Z/q\Z$; note also that for any 
$R\in\M_{d\times n}(\Z/q\Z)$, %
$q^{-1}R$ is a well-defined point in the torus
$\M_{d\times n}(\R/\Z)$.)

\vspace{3pt}

We prove \autoref{fScharLEM} %
in Section \ref{RqSEC}
(see also \cite[Sec.\ 2]{EMSS2016}).
As in \cite[Definition~1.1]{EMSS2016},
for a positive integer $q$,
let us call a matrix 
$R\in {\rm M}_{d\times n}(\mathbb{Z}/q\Z)$ \textit{($q$-)primitive} if the rows of $R$
generate $(\mathbb{Z}/q\mathbb{Z})^n$.
We will also say that a matrix $R\in {\rm M}_{d\times n}(\mathbb{Z})$ is $q$-primitive
if its reduction mod $q$ is $q$-primitive.
Let $\scrR_q$ be the set of all primitive matrices in $\M_{d\times n}(\Z/q\Z)$.
Then \autoref{fScharLEM} says that
\begin{align}\label{ourrationalpoints}
\fS_q=\bigl\{\tn_+(q^{-1}R)D(q)\col R\in\scrR_q\bigr\}.
\end{align}
We call $\fS_q$ the set of primitive rational points on $\fH_q$.

We are interested in the behavior of this %
point set $\fS_q$ for $q$ large.
It was proved by Einsiedler, Mozes, Shah and Shapira \cite{EMSS2016}
that $\fS_q$ becomes equidistributed in
$\Gamma\bs\Gamma{\rm H}$
with respect to the ${\rm H}$-invariant probability measure,
as $q\to\infty$.
In fact, confirming a conjecture by Marklof, they proved the much stronger fact that the point set
\begin{equation}
\tfS_q:=\bigl\{\bigl(q^{-1}R, \tn_+(q^{-1}R)D(q)\bigr)\col R\in\scrR_q\bigr\}
\end{equation}
becomes (jointly) equidistributed in the product space
$(\R/\Z)^{dn}\times\Gamma\bs\Gamma{\rm H}$,
where we have identified the torus
$\M_{d\times n}(\R/\Z)$
with $(\R/\Z)^{dn}$.

In the present paper we give a new proof of this equidistribution result
which relies on harmonic analysis and number theory,
spectral theory of automorphic forms, the newly studied object of matrix Kloosterman sums,
and Rogers' integration formula in the geometry of numbers.
Our proof leads to an \textit{effective} version of the equidistribution result,
that is, we obtain explicit information on how quickly the equidistribution takes place as $q\to\infty$.

\subsection{Informal statement of the main result}
Given a function $f:(\R/\Z)^{dn} \times \Gamma \backslash \Gamma {\rm H}\to\R$, set
\begin{equation}\label{SqfDEF}
\scrA_q(f) = 
\frac{1}{\#\mathcal{R}_q} \sum_{R\in \mathcal{R}_q} f(q^{-1}R, \tn_+(q^{-1} R) D(q)).
\end{equation}
Then the statement of Einsiedler--Mozes--Shah--Shapira's theorem
is precisely that whenever $f$ is bounded and continuous, 
$\scrA_q(f)$ converges to the integral of $f$ %
as $q\to\infty$.
By standard approximation arguments,
it is equivalent to 
state %
that this convergence holds 
whenever %
$f$ is smooth and compactly supported.

Our main result is an effective version of that %
result, with a power-saving error term, meaning that we prove, 
for every $1 \leq n \leq d$, and for any sufficiently smooth $f$,
\begin{equation}
\scrA_q(f) = \int_{(\R/\Z)^{dn} \times \Gamma \backslash \Gamma {\rm H}} f\, dT \, d \mu_{\rm H} 
+ O_d\bigl(S(f)\, q^{-\delta}\bigr)
\end{equation}
as $q \to \infty$, 
where $dT$ is the usual Lebesgue measure on $(\R/\Z)^{dn}$,
$\mu_{\rm H}$ is the $\rm H$-invariant  probability measure on $\Gamma\bs\Gamma{\rm H}$,
and $S(f)$ is a certain Sobolev norm 
of $f$, defined in terms of the ${\rm L}^2$ and  ${\rm L}^\infty$ 
norms of $f$ and its first several derivatives 
(see \S\ref{ss:mainresult}),  %
while $\delta > 0$ is a fixed constant. 

In the special case $n=1$ such an effective equidistribution result was obtained in \cite{LeeMarklof2017} (for $d=2$) and \cite{EHL18} (for general $d$).
Our main theorem, stated more precisely in the next section, finally provides an effective version of the general case of the Einsiedler--Mozes--Shah--Shapira theorem.

\subsection{Formal statement of the main result}\label{ss:mainresult}

In order to state our result we need to introduce certain Sobolev norms of functions on homogeneous spaces
(compare \cite[Sec.\ 2.9.2]{aV2010a}).
Suppose $\Lambda$ is a lattice in a connected Lie group $L$,
and let $\mu$ be the $L$-invariant probability measure on $\Lambda\bs L$.
Fix, once and for all, a linear basis $\scrB$ for the Lie algebra of $L$.
Let $k\geq0$ be an integer.
For $f\in{\rm C}^k(\Lambda\bs L)$ and $1\leq p\leq\infty$ (in fact we will only consider $p=2$ and $p=\infty$), we 
define the Sobolev norm of $f$
\begin{align}\label{SpkDEF}
S_{p,k}(f)=\sum_{\ord(\scrD)\leq k}\|\scrD f\|_{{\rm L}^p(\Lambda\bs L,\mu)},
\end{align}
where $\scrD$ runs through all monomials in $\scrB$ of order $\leq k$. 
Here $\scrD$ acts on $f$ by right differentiation:
\begin{equation} Xf(g)=\frac d{dt}f(g\exp(tX))\big|_{t=0} 
\text{ for any }X\in \scrB.  
\end{equation}
It should be noted that changing the basis $\scrB$ only distorts $S_{p,k}$ by a bounded factor.

We write ${\rm C}_b^k(\Lambda\bs L)$ for the space of functions in ${\rm C}^k(\Lambda\bs L)$ which have all derivatives of order $\leq k$ bounded,
i.e.,
\begin{align*}
{\rm C}_b^k(\Lambda\bs L)=\bigl\{f\in {\rm C}^k(\Lambda\bs L)\col S_{\infty,k}(f)<\infty\bigr\}.
\end{align*}
It will be convenient to also introduce,
in a non-standard but elementary way,
\textit{fractional} Sobolev norms (cf.\ \cite[Lemma 2]{StrombergssonVenkatesh2005}):
For any %
real number $k<\kappa<k+1$
and $f\in{\rm C}^{k+1}(\Lambda\bs L)$,
we set
\begin{align}\label{fracSobnorm}
S_{p,\kappa}(f)=S_{p,k}(f)^{k+1-\kappa}S_{p,k+1}(f)^{\kappa-k}.
\end{align}
In the statement of the following theorem,
the above formalism is applied
for the homogeneous space
$(\R/\Z)^{dn}\times \Gamma\bs\Gamma{\rm H}$,
that is, with 
$\Lambda=\Z^{dn}\times(\Gamma\cap {\rm H})$
and $L=\R^{dn}\times {\rm H}$.

Let $\theta$ be the constant towards the Ramanujan conjecture
for Maass wave forms on $\SL_2(\Z)\bs\SL_2(\R)$, which asserts $\theta=0$.
The current best bound is $\theta \leq 7/64$, due to Kim and Sarnak
\cite[Appendix~2]{KimSarnak2003}.

\begin{theorem}\label{MAINTHM}
For the given positive integers $1\leq n\leq d$, let
\begin{align*}
\kappa=2dn;\hspace{10pt}
\vartheta=\begin{cases} n-1 & (\text{if } n>1), \\ \frac{1}{2} & (\text{if } n=1);\end{cases}
\quad\text{and}\quad
\begin{cases}
\kappa'=\tfrac12(d^2-1);\hspace{10pt}
\vartheta'=\tfrac12\min(n,d-n)
\\[1pt]
\rule{0pt}{0pt} \hspace{125pt}(\text{if }\: n<d\text{ and }d\geq3);
\\[3pt]
\kappa'=\tfrac32;\hspace{10pt}
\vartheta'=\tfrac12-\theta 
 \hspace{30pt}(\text{if }\: n=1\text{ and }d=2);
\\[3pt]
\kappa'=\kappa;\hspace{10pt}
\vartheta'=\vartheta
\hspace{50pt}(\text{if }\: n=d).
\end{cases}
\end{align*}
Also let $k$ be the smallest integer greater than both $\kappa$ and $\kappa'$.

\vspace{3pt}

Then for any $0<\ve<\frac12$,
$f\in {\rm C}_b^k((\R/\Z)^{dn} \times \Gamma \backslash \Gamma {\rm H})$, 
and any positive integer $q$, %
\begin{align}\label{MAINTHMres}
\scrA_q(f)=\int_{(\R/\Z)^{dn}} 
\int_{\Gamma\backslash \Gamma {\rm H}} f(T, g) \,d\mu_{\rm H}(g) \,dT
+
O\Bigl(S_{\infty,\kappa+\ve}(f)\, q^{-\vartheta+\ve}+%
S_{2,\kappa'+\ve}(f)\, q^{-\vartheta'+\ve} \Bigr),
\end{align}
where the implied constant only depends on $d$ and $\ve$.
\end{theorem}

\begin{remark}
In the statement of \autoref{MAINTHM},
it should be noted that we always have
$\kappa+\ve<k$ and $\kappa'+\ve<k$,
and thus both the Sobolev norms
$S_{\infty,\kappa+\ve}(f)$
and $S_{2,\kappa'+\ve}(f)$ are defined and finite.
It should also be noted that 
the introduction of $\kappa'$ and $\vartheta'$ in the case $n=d$ is only a
notational convenience,
allowing a simple comprehensive statement of \eqref{MAINTHMres}. 
Indeed, 
in that case %
the error term in 
\eqref{MAINTHMres} reduces to 
$O\bigl(S_{\infty,\kappa+\ve}(f)\, q^{-\vartheta+\ve}\bigr)$,
since $S_{2,\kappa+\ve}(f)\leq S_{\infty,\kappa+\ve}(f)$.
\end{remark}

\subsection{Discussion of the result and layout of the proof}
\label{layoutSEC}
As we have already mentioned, the problem of studying the limiting distribution of the primitive rational points \eqref{ourrationalpoints} 
on the expanding closed horosphes $\fH_q$, was raised by Marklof, specifically in \cite{Marklof2010} when $n=1$. 
Marklof proved an averaged version of the equidistribution of primitive rational points on expanding horospheres
and used it to obtain a limiting distribution result for Frobenius numbers. 
His work was made effective, using estimates on the decay of matrix coefficients, by Li \cite{Li2015}. 

The proof of Marklof's conjecture by Einsiedler, Mozes, Shah and Shapira \cite{EMSS2016}
uses techniques from homogeneous dynamics and relies in particular on measure-classification theorems due to Ratner \cite{rat}, extended by Shah \cite{sha}, which are inherently ineffective. 

For $n\geq2$, the result of \autoref{MAINTHM}, with \textit{any} effective rate of equidistribution,
is new. 
It is also worth noticing that in the special case $n=1$,
our error bound is stronger than those in \cite{LeeMarklof2017} (for $n=1$ and $d=2$) 
and in \cite{EHL18} (for $n=1$ and $d \geq 2$). 
More precisely, for $n=1$ and $d\geq3$,
the error bound in \autoref{MAINTHM} is $\bigl(S_{\infty,\kappa+\ve}(f)+S_{2,\kappa'+\ve}(f)\bigr)\cdot q^{-\frac12+\ve}$ with $\kappa=2d$ and $\kappa'=\frac12(d^2-1)$;
this is stronger than the error term in \cite[Theorem 1.1]{EHL18}, both in terms of the Sobolev norm and the power of $q$.\footnote{One 
may note that the $q$-exponent in \cite[Theorem 1.1]{EHL18} tends to our exponent $-\frac12+\ve$ if one lets the order of the 
Sobolev norm tend to $+\infty$.}
For $n=1$ and $d=2$,
the error bound in \autoref{MAINTHM} is $S_{\infty,4+\ve}(f)\, q^{-\frac12+\ve}+S_{2,\frac32+\ve}(f)\, q^{-\frac12+\theta+\ve}$, which is stronger than the bound in both \cite[Theorem 1.3]{LeeMarklof2017} and \cite[Remark 1.2]{EHL18}. 
Finally for $n=d=1$ the error bound in \autoref{MAINTHM} is $S_{\infty,2+\ve}(f)\, q^{-\frac12+\ve}$.
That case is quite easy; see \cite{Marklof2010horospheres} and \cite[Sec.\ 2.1]{EMSS2016}
(neither of those include the precise error term, but that is not at all difficult).

The basic set-up of the proof of \autoref{MAINTHM} is similar to the one in both
\cite{LeeMarklof2017} and \cite{EHL18}:
In \autoref{lem:scrRq_parameter} we 
give a parametrization of the set $\scrR_q$
of primitive matrices in terms of %
$\Gamma^0(q)\bs\SL_d(\Z)$ and $\GL_n(\Z/q\Z)$,
where $\Gamma^0(q)$ is a certain congruence subgroup of $\SL_d(\Z)$
(for $n=1$ this was done in \cite[Lemma 2.2]{EHL18}).
Furthermore, our first step is to Fourier expand the given test function on
$(\R/\Z)^{dn} \times \Gamma \backslash \Gamma {\rm H}$,
both with respect to the variable in the 
torus $(\R/\Z)^{dn}$ and with respect to the
torus fiber variable in
$\Gamma \backslash \Gamma {\rm H}$; see Section \ref{ss:Fourier}.
Then the main term in \eqref{MAINTHMres} is obtained by using
the asymptotic equidistribution of certain Hecke orbits in
$\SL_d(\Z)\bs\SL_d(\R)$, and for this an optimal error term is provided by
the results of Clozel, Oh and Ullmo \cite{ClozelOhUllmo2001};
see Section \ref{ss:Heckepts}. 

However, the task of bounding the contribution from %
the remaining sums is significantly more challenging
in the present paper where we deal with general $n\geq1$.
Here our first step is to apply bounds on the newly studied ``matrix Kloosterman sums''.
For prime moduli, key bounds on these matrix Kloosterman sums have been proved by Erd\'elyi and T\'oth \cite{mEaT2021};
for the case of %
higher prime power moduli
we prove non-trivial bounds in Section \ref{primepowerSEC},
by elementary but somewhat complicated computations.
Similar bounds have also, 
independently, been obtained by 
Erd\'elyi, T\'oth and Z\'abr\'adi in the recent paper \cite{ETZ}. 
The majorizing sum which arises from the application of the bounds on matrix Kloosterman sums is still non-trivial to control.
At this point we make use of a Hecke operator interpretation 
followed by an application of an integration formula by Rogers \cite{cR55} in the geometry of numbers, to arrive at a satisfactory final bound. This is carried out in Section~\ref{E2qfSEC}. 
The usage of Rogers' integration formula
in the present method
is also the reason for %
our improvement of the error bounds in 
\cite{LeeMarklof2017} and \cite{EHL18}
in the case $n=1$.

\subsection{Consequences of our main theorem and its proof} 
The case $n=1$ of the equidistribution result of
Einsiedler, Mozes, Shah and Shapira 
is known to have applications to the 
distribution of Frobenius numbers \cite{Marklof2010}, 
the distribution of shapes of lattices \cite{EMSS2016},
and to the distribution of metric parameters of random Cayley graphs of cyclic groups
\cite{MarklofStrombergsson2013}. 
Naturally, an effective version of this equidistribution result
can be expected to lead to information on the rate of convergence 
in these applications;
in \cite[Cor.\ 5.1]{EHL18} this was carried out for the case of
the diameter of random Cayley graphs of cyclic groups.
(Our improved error bound in \autoref{MAINTHM} should lead to an improved
exponent $\eta_d$ in \cite[Cor.\ 5.1]{EHL18}.)

In the present article, 
in Section \ref{ss:smallsolutions},
we give a new application of 
the equidistribution result,
this time for arbitrary $1\leq n\leq d$:
We obtain %
the limit distribution of the number of 
small solutions of a random system of linear congruences to a large modulus.
This can be seen as a variation, and in a sense a refinement,
of results by Str\"ombergsson and Venkatesh \cite{StrombergssonVenkatesh2005}
(see Remark \ref{SVspecREM}).

Furthermore, while first attempting to follow the strategy deployed in \cite{EHL18}, we came across an elementary counting problem in linear algebra, for which we were however unable to find an elementary solution.
That resulted in the technique we instead follow in Section~\ref{E2qfSEC}, using a Hecke operator interpretation followed by an application of Rogers' integration formula. 
As a by-product of our proof, we are able to satisfactorily solve the linear algebra problem, whose statement is as follows:
\begin{problem}\label{rankcountprobl}
For integers $1\leq r < n < d$, 
a prime $p \geq 3$ and an integer $1 \leq b \leq \frac {p-1}2$,
estimate the growth rate of
\begin{equation}\label{rankcount}
    \# \{ A \in \mathrm{M}_{d \times n}(\mathbb{Z}) \, : \, \|A\|_\infty \leq b \text{ and } \mathrm{rank}(A \bmod p) = r \}
\end{equation}
as $p$ gets large. Here $\|A\|_\infty$ denotes %
the maximum of the absolute values of the entries of $A$.
\end{problem}

\autoref{rankcountprobl} %
has been studied previously, for all values of $n$ and $d$, by Ahmadi and Shparlinski \cite{AhShS2007},
who obtained an asymptotic formula for \eqref{rankcount}
valid as $p\to\infty$ with $b$ in a restricted range.
In Section \ref{ss:byprod} we prove a sharp estimate on \eqref{rankcount},
valid for arbitrary $b$ and $p$.

\section{The primitive rational points on $\fH_q$}
\label{RqSEC}

As in the introduction, we keep $1\leq n\leq d$ fixed.
\begin{proof}[Proof of \autoref{fScharLEM}]
Let $q$ be a positive real number, %
and assume that $\fS_q$ is non-empty.
This means that there is some $V\in\M_{d\times n}(\R)$
such that $n_+(V)D(q)\in\Gamma{\rm H}$,
that is
\begin{align}\label{fScharLEMpf1}
\matr{q^{-n/d}I_d}{qV}{\bn}{qI_n}=\gamma\matr A{\bn}U{I_n}
\end{align}
for some $\gamma\in\Gamma$, $A\in\SL_d(\R)$
(if $n=d$: $A=I_n$)
and $U\in\M_{n\times d}(\R)$.
All the entries in the last $n$ columns of the matrix in the right hand side
are integers;
hence $q$ must be an integer, and $V=q^{-1}R$ for some $R\in\M_{d\times n}(\Z)$.
Also, left-multiplying the relation in \eqref{fScharLEMpf1} %
by $\gamma^{-1}$
and inspecting the bottom right $n\times n$ submatrix ($=I_n$), %
it follows that each of the standard basis vectors $\vece_1,\ldots,\vece_n$ of $\R^n$
is %
an integer linear
combination of the row vectors of $R$
and $q\vece_1,\ldots,q\vece_n$.
Hence $R$ is $q$-primitive.

Conversely, assume that $q$ is a positive integer and %
$R\in\M_{d\times n}(\Z)$ is $q$-primitive.
Then the homomorphism 
$\veca\mapsto\veca R\mod q$ from $\Z^d$ to $\Z^n/q\Z^n$ 
is surjective; hence its kernel $K$ is a subgroup of $\Z^d$ of index $q^n$.
Let $\veca_1,\ldots,\veca_d$ be a positively oriented $\Z$-basis of $K$,
where if $n=d$ we require $\veca_j:=q\vece_j$ for $j=1,\ldots,d$
(this is ok since $K=q\Z^n$ if $n=d$).
Let $A'$ be the $d\times d$ matrix with row vectors $\veca_1,\ldots,\veca_d$.
Then $\det(A')=q^n$,
and for each $\veca_j$ there is a unique $\vecb_j\in\Z^n$ 
such that $\veca_jR+q\vecb_j=\bn$.
Also, since $R$ is $q$-primitive,
there exist $\vecc_1,\ldots,\vecc_n\in\Z^{d+n}$
such that $\vecc_j\begin{pmatrix} R\\ qI_n\end{pmatrix}=\vece_j$ ($j=1,\ldots,n$).
Now let $\eta$ %
be the square matrix with row vectors 
$(\veca_1,\vecb_1),\ldots,(\veca_d,\vecb_d),\vecc_1,\ldots,\vecc_n$;
then
\begin{align*}
\eta\matr{q^{-n/d}I_d}{R}{\bn}{qI_n}=\matr {q^{-n/d}A'}{\bn}U{I_n}
\end{align*}
for some %
$U\in\M_{n\times d}(\R)$.
Here $\det(q^{-n/d}A')=1$,
and if $n=d$ then $q^{-1}A'=I_n$.
Hence the above matrix lies in ${\rm H}$, and 
$\det(\eta)=1$, i.e.\ $\eta\in\Gamma$.
Therefore $\tn_+(q^{-1}R)D(q)\in\fS_q$.
\end{proof}

It will be useful to know the cardinality of $\fS_q$, i.e.\ the cardinality of $\scrR_q$.
We write $\mathbb{Z}^+$ for the set of positive integers.
\begin{lemma}\label{RqcardLEM}
$\forall q \in \Z^+, \#\scrR_q=q^{dn}\prod_{p\mid q}\prod_{j=d+1-n}^d(1-p^{-j})$.
\end{lemma}
\begin{proof}
It follows from the Chinese Remainder Theorem that
the function $q\mapsto\#\scrR_q$ is multiplicative;
hence %
it suffices to prove the lemma 
when $q$ is a prime power, say $q=p^r$ ($r\geq1$).
Now, 
for any $R\in\M_{d\times n}(\Z/q\Z)$
such that $R\mod p$ is $p$-primitive,
$R\mod p$ has some $n\times n$ submatrix which belongs to $\GL_n(\Z/p\Z)$;
therefore %
the determinant of the corresponding submatrix of $R$ itself
is a unit in $\Z/q\Z$, viz., that submatrix belongs to $\GL_n(\Z/q\Z)$
and $R$ is $q$-primitive.
Hence $R\in\M_{d\times n}(\Z/q\Z)$ is $q$-primitive 
if and only if $R$ mod $p$ is $p$-primitive,
and so %
$\#\scrR_q=p^{(r-1)dn}\#\scrR_p$.

It remains to prove the lemma in the case $q=p$, a prime.
Let us write $\F_p$ for the field $\Z/p\Z$.
A matrix in $\M_{d\times n}(\F_p)$ is $p$-primitive 
if and only if it has full rank,
that is, if and only if its columns are linearly independent.
Note that there are exactly 
$p^d-1$ full rank matrices in $\M_{d\times 1}(\F_p)$.
Furthermore, for any $1\leq \ell<d$, given any 
matrix $A\in\M_{d\times\ell}(\F_p)$ of full rank,
the column span of $A$ has cardinality $p^\ell$,
and hence there are exactly $p^d-p^\ell$ ways to choose a column to the right of $A$
to form a full rank matrix in $\M_{d\times(\ell+1)}(\F_p)$.
Hence $\#\scrR_p=\prod_{\ell=0}^{n-1}(p^d-p^{\ell})$,
and the lemma is proved.
\end{proof}

In the next lemma we give a parametrization of $\scrR_q$ which will be crucial in our proof of the main theorem.
If $n<d$, then we define $\Gamma^0(q)$ to be the 
following congruence subgroup of ${\rm SL}_d(\mathbb{Z})$: 
\begin{equation}\label{e:Gamma0dnq_def}
\Gamma^0(q) = \left\{\matr AB{\trans C}D\in {\rm SL}_d(\mathbb{Z})\col
A\in\M_{d-n},\: B,C\in\M_{(d-n)\times n},\: D\in\M_n,\: 
B\equiv\bn\mod q\right\},
\end{equation}
and we fix a set $\mathcal{B}_q$ of representatives for $\Gamma^0(q) \backslash {\rm SL}_d(\mathbb{Z})$. 
When $d=n$, we set 
\begin{equation}
\Gamma^0(q) = {\rm SL}_n(\mathbb{Z})
\end{equation}
and $\mathcal{B}_q:=\{I_n\}$. 
The following lemma generalizes \cite[Lemma 2.2]{EHL18}.
\begin{lemma}\label{lem:scrRq_parameter}
The map 
\begin{equation}
\mathcal{B}_q \times {\rm GL}_n(\mathbb{Z}/q\mathbb{Z}) 
\to {\rm M}_{d\times n}(\mathbb{Z}/q\mathbb{Z})
\end{equation}
given by 
\begin{align}\label{e:map_scrRq_def}
&\hspace{90pt}
\left<\gamma, U\right> \mapsto \gamma^{-1} \bpm \vecnull \\ U\ebpm 
\hspace{50pt} (\gamma\in\scrB_q,\: U\in{\rm GL}_n(\mathbb{Z}/q\mathbb{Z})),
\end{align}
is a bijection onto $\mathcal{R}_q\subset {\rm M}_{d\times n}(\mathbb{Z}/q\mathbb{Z})$. 
\end{lemma}
(In the case $n=d$, the matrix ``$\scmatr{\bn}U$'' in \eqref{e:map_scrRq_def}
should be interpreted as ``$U$''.)

\begin{proof}
If $n=d$
then $\scrR_q=\GL_n(\Z/q\Z)$ and $\scrB_q=\{I_n\}$
and the lemma is trivial.

From now on we assume that $1\leq n < d$. 
It is clear that the image of the map in \eqref{e:map_scrRq_def}
is contained in $\scrR_q$.
To prove that the map is surjective,
let $R\in\scrR_q$ be given.
Then by the Smith Normal Form Theorem,
there exist $\delta\in\GL_d(\Z/q\Z)$ and $\eta\in\GL_n(\Z/q\Z)$
and a diagonal matrix $D\in\M_n(\Z/q\Z)$
such that 
\begin{align*}
R=\delta\cmatr{\bn}{D}\eta=\delta\cmatr{\bn}{D\eta}.
\end{align*}
Note that the above identity remains true if we replace $\delta$ by $\delta \diag[u,1,\cdots,1]$
for any $u\in(\Z/q\Z)^\times$;
hence we may arrange 
that 
$\delta\in\SL_n(\Z/q\Z)$.
The reduction map from $\SL_n(\Z)$ to $\SL_n(\Z/q\Z)$ is surjective 
(see, e.g., the proof of \cite[Lemma 1.38]{Shi94});
hence there exists a lift $\delta'\in\SL_n(\Z)$ of $\delta$.
Let $\gamma$ be the unique element in $\scrB_q\cap\Gamma^0(q){\delta'}^{-1}$;
then
$\delta'=\gamma^{-1}\matr ABC{D'}$
for some $\matr ABC{D'}\in\Gamma^0(q)$, and so
\begin{align*}
R=\gamma^{-1}\cmatr{\bn}{D'D\eta}
\qquad\text{in }\: \M_{d\times n}(\Z/q\Z).
\end{align*}
But $R\in\scrR_q$ implies $\gamma R\in\scrR_q$;
hence the rows of $D'D\eta$ generate $(\Z/q\Z)^n$, that is, $D'D\eta\in\GL_n(\Z/q\Z)$,
and we have thus proved that $R$ lies in the image of the map
in \eqref{e:map_scrRq_def}.

It remains to verify that the map is injective.
Thus we assume that the two pairs
$\langle\gamma,U\rangle$ and
$\langle\gamma',U'\rangle$ map to the same element in $\scrR_q$.
Then
\begin{align*}
\gamma'\gamma^{-1}\cmatr{\bn}{U}=\cmatr{\bn}{U'}
\qquad\text{in }\: \M_{d\times n}(\Z/q\Z).
\end{align*}
This forces $\gamma'\gamma^{-1}\in\Gamma^0(q)$,
and since $\gamma,\gamma'\in\scrB_q$
it follows that $\gamma=\gamma'$.
Hence also $U=U'$, and the injectivity is proved.
\end{proof}

In the next lemma we give a formula which will be useful
when applying \autoref{lem:scrRq_parameter}
to re-express the sum in \eqref{SqfDEF}.
Let us introduce, for $U \in  {\rm M}_{n\times d}(\mathbb{R})$, 
\begin{equation}n_-(U) = \bpm I_d & \vecnull \\ U & I_n\ebpm.
\end{equation}
We also introduce the map
\begin{align}\label{tnnDEF}
\widetilde{n}_-:\M_{n\times d}(\R/\Z)\to\Gamma\bs G
\end{align}
by setting
$\tn_-(U):=\Gamma n_-(U')$ where $U'$ is any lift
to $\M_{n\times d}(\R)$ of $U\in\M_{n\times d}(\R/\Z)$
(this is analogous to $\tn_+$ in \eqref{tnpDEF}).

\begin{lemma}\label{lem:mtx_relation}
Assume that $1\leq n \leq d$. 
Let $\gamma\in \mathcal{B}_q$ and $U\in {\rm GL}_n(\mathbb{Z}/q\mathbb{Z})$, 
and set
\begin{align}\label{lem:mtx_relationDEFS1}
&R=\gamma^{-1} \cmatr{\bn}U\in\scrR_q
\qquad\text{and}\qquad
S= \bigl(\bn\hspace{7pt} U^{-1}\bigr) \in\M_{n\times d}(\Z/q\Z) ;
\\\label{Dqdef}
& D_q = \begin{cases}
q^{-\frac nd}\bpm I_{d-n} & \\ & qI_n\ebpm & \text{ when } n< d, \\
I_n & \text{ when } n=d. 
\end{cases} 
\end{align}
Then 
\begin{equation}\label{e:mtx_relation}
\tn_+(q^{-1} R) D(q) = 
\tn_-(q^{-1}S)  \matr{D_q\gamma}{\bn}{\bn}{I_n}.
\end{equation}
\end{lemma}
(In the case $n=d$, the matrices
``$\scmatr{\bn}U$'' and 
``$(\bn\hspace{5pt} U^{-1})$''   %
in \eqref{lem:mtx_relationDEFS1}
should be interpreted as ``$U$'' and ``$U^{-1}$'', respectively.
The matrix $D_q\in\SL_d(\R)$ in \eqref{Dqdef} should not be mixed up with the matrix $D(y)$ in ${\rm G}=\SL_{d+n}(\R)$
defined in \eqref{Dydef}.)
\begin{proof}
Our task is to prove \eqref{e:mtx_relation}, or equivalently
\begin{align}\label{lem:mtx_relationPF1}
n_-(q^{-1}S')  \matr{D_q\gamma}{\bn}{\bn}{I_n} 
\big(n_+(q^{-1} R') D(q)\big)^{-1} \in \Gamma,
\end{align}
where $R'$ and $S'$ are arbitrary lifts
of $R$ to $\M_{d\times n}(\Z)$ and $S$ to $\M_{n\times d}(\Z)$, respectively.
The matrix in \eqref{lem:mtx_relationPF1} clearly has determinant one;
hence it remains to prove that all its entries are integers.
By a quick computation, the matrix 
is seen to equal
\begin{align}\label{lem:mtx_relationPF2}
\begin{pmatrix}q^{\frac nd}D_q\gamma & -q^{\frac nd-1}D_q\gamma R'
\\[5pt]
q^{\frac nd-1} S'D_q\gamma
& -q^{\frac nd-2}\, S'D_q\gamma R'+q^{-1}I_n
\end{pmatrix}.
\end{align}
Here the top left block matrix is clearly in $\M_{d\times d}(\Z)$,
and using $\gamma R'\equiv\cmatr{\bn}U\mod q$,
the top right block matrix is seen to be in $\M_{d\times n}(\Z)$;
similarly the bottom left block matrix is in $\M_{n\times d}(\Z)$.
Finally, one verifies that $q^{\frac nd-1} S' D_q$
is in $\M_{n\times d}(\Z)$ with its rightmost $n\times n$ submatrix being
$\equiv U^{-1}\mod q$;
hence, since also
$\gamma R'\equiv\cmatr{\bn}U\mod q$,
it follows that $q^{\frac nd-1} S' D_q \gamma R'\in I_n+q\cdot \M_{n}(\Z)$.
This implies that the bottom right block matrix in \eqref{lem:mtx_relationPF2}
is in $\M_n(\Z)$, and the lemma is proved.
\end{proof}

\section{Fourier analysis on the space $\Gamma\bs\Gamma {\rm H}$}\label{ss:Fourier}
The material in the present section generalizes %
\cite[Sec.\ 4]{Stromb2015}.
Throughout the section we assume $1\leq n<d$.

We will parametrize the group ${\rm H}$ using the following diffeomorphism:
\begin{align}\label{Hparametrization}
\SL_d(\R)\times\M_{n\times d}(\R)\stackrel{\sim}{\longrightarrow} {\rm H},
\qquad (g,X)\mapsto\matr{I_d}{\bn}X{I_n}\matr g{\bn}{\bn}{I_n}
=\matr g{\bn}{Xg}{I_n}.
\end{align}
Note that then
$\Gamma\cap {\rm H}$ corresponds to $\SL_d(\Z)\times\M_{n\times d}(\Z)$,
and the multiplication law in ${\rm H}$ is given by
\begin{align}
(g,X)(g',X')=(gg',X+X'g^{-1}).
\end{align}
In particular, if $F$ is a left $\Gamma\cap {\rm H}$ invariant function on ${\rm H}$
(or equivalently, a function on $\Gamma\bs\Gamma {\rm H}$),
then in terms of our parametrization we have
$F(g,X+M)\equiv F(g,X)$ for all $M\in\M_{n\times d}(\Z)$ \footnote{We 
also have $F(\gamma g,X\gamma^{-1})\equiv F(g,X)$ for all $\gamma\in\SL_d(\Z)$.},
which means that for any fixed $g\in\SL_d(\R)$,
$X\mapsto F(g,X)$ is a function on the torus
$\M_{n\times d}(\R/\Z)$.
We write 
$\hF(g;M)$ for the Fourier coefficients in the torus variable:
\begin{align}\label{hFgMdef}
\hF(g;M)=\int_{\M_{n\times d}(\R/\Z)} F(g, X)\, e^{-2\pi i \,{\rm tr}(\trans MX)} dX,
\end{align}
where $dX$ denotes the Lebesgue measure on 
$\M_{n\times d}(\R)\cong(\R^d)^n$.
Thus 
for any $k>\frac12nd$ and any $F\in {\rm C}^k(\Gamma\bs\Gamma {\rm H})$, we have
\cite[Theorem 3.2.16]{Grafakos}
\begin{align}\label{e:Fourier_ASL}
F(g, X) = \sum_{M\in {\rm M}_{n\times d}(\mathbb{Z})} \widehat{F}(g; M) e^{2\pi i {\rm tr}(\trans M X)},
\end{align}
with a uniform absolute convergence\footnote{For any fixed ordering of $\M_{n\times d}(\Z)$.}
over $(g,X)$ in any compact subset of ${\rm H}$.

\begin{lemma}\label{lem:Fouriercoeff_semiautomorphic}
Let $F\in{\rm C}(\Gamma\bs\Gamma {\rm H})$.
Then for any 
$\gamma\in {\rm SL}_{d}(\mathbb{Z})$, $g\in {\rm SL_d}(\mathbb{R})$ and $M\in {\rm M}_{n\times d}(\mathbb{Z})$, 
\begin{equation}\label{e:Fouriercoeff_semiautomorphic}
\widehat{F}(\gamma g; M) = \widehat{F}(g; M\trans\gamma^{-1}). 
\end{equation}
\end{lemma}
\begin{proof}
This follows from the formula \eqref{hFgMdef} 
and $F(\gamma g,X)\equiv F(g,X\gamma)$,
and the fact that the map $X\mapsto X\gamma$ is a diffeomorphism of the 
torus $\M_{n\times d}(\R/\Z)$ onto itself preserving the Lebesgue measure.
\end{proof}
For $1\leq i \leq n$ and $1\leq j \leq d$, let $E_{i, j}\in {\rm M}_{n\times d}(\mathbb{R})$ denote the matrix with a $1$ at the $(i, j)$th 
position and zeros elsewhere. 
We define the following differential operator:
\begin{equation}\label{e:Liealgebra_diff_def}
(E_{i, j} F)(g, X) = \left.\frac{\partial}{\partial t} F\bigl((g, X)(I_d, tE_{i,j})\bigr)\right|_{t=0}.
\end{equation}
Using $(g,X)(I_d,tE_{i,j})=(g,X+tE_{i,j}g^{-1})$
and the chain rule,
we get
\begin{equation}\label{e:Liealgebra_diff}
(E_{i, j} F)(g, X)
= \sum_{\ell=1}^d \og_{j,\ell} \frac{\partial F}{\partial x_{i,\ell}}(g, X), 
\end{equation}
where $X=(x_{i,j})_{1\leq i\leq n,1\leq j\leq d}$ and $g^{-1}=(\og_{i,j})_{1\leq i,j\leq d}$.

\begin{lemma}\label{lem:Fouriercoeff_bound}
Let $0\leq\kappa\leq k$ with $k\in\Z$.
Then for any 
$F\in {\rm C}_b^k(\Gamma\bs\Gamma {\rm H})$, 
$g\in\SL_d(\R)$ and $M\in\M_{n\times d}(\Z)$, 
\begin{equation}\label{e:Fouriercoeff_bound} 
\bigl|\widehat{F}(g; M)\bigr|
\ll_k\frac{S_{\infty,\kappa}(F)}{1+\|M \trans g^{-1}\|_\infty^{\kappa}}.
\end{equation}
\end{lemma}
\begin{proof}
By \eqref{hFgMdef} and \eqref{e:Liealgebra_diff}, we have 
for any $1\leq i\leq n$ and $1\leq j\leq d$,
\begin{align}
(\widehat{E_{i, j} F})(g; M)
= \sum_{\ell=1}^d \og_{j,\ell}
\int_{{\rm M}_{n\times d}(\R/\Z)}\frac{\partial F}{\partial x_{i,\ell}}(g, X)e^{-2\pi i\, {\rm tr}(\trans M X)} dX. 
\end{align}
Hence by integration by parts,
\begin{align}
(\widehat{E_{i, j} F})(g; M)
= 2\pi i \biggl(\sum_{\ell=1}^d \og_{j,\ell}
m_{i,\ell}\biggr) \hF(g;M).
\end{align}
Repeated use of this formula gives
\begin{align}
(\widehat{E_{i, j}^k F})(g; M)
= (2\pi i)^k \biggl(\sum_{\ell=1}^d \og_{j,\ell}m_{i,\ell}\biggr)^k \hF(g;M).
\end{align}
Recall the definition of the Sobolev norm $S_{\infty,k}$ on ${\rm C}_b^k(\Gamma\bs\Gamma {\rm H})$;
see \eqref{SpkDEF}.
We may assume that the fixed basis for the Lie algebra of ${\rm H}$ which is used in this definition
contains the vectors $\frac d{dt}\sm I_d  & \bn \\ tE_{i,j} & I_n \esm\big|_{t=0}$ for all $i,j$.
Then, using also \eqref{hFgMdef}, we have
\begin{align*}
\Bigl|(\widehat{E_{i, j}^k F})(g; M)\Bigr|
\leq \bigl\|E_{i, j}^k F\bigr\|_\infty
\leq S_{\infty,k}(F).
\end{align*}
Hence we conclude:
\begin{align*}
(2\pi)^k\biggl|\sum_{\ell=1}^d \og_{j,\ell}m_{i,\ell}\biggr|^k \bigl|\hF(g;M)\bigr|
\leq S_{\infty,k}(F).
\end{align*}
Note also that, trivially,
\begin{equation}
\left|\widehat{F}(g; M)\right| \leq S_{\infty,k}(F). 
\end{equation}
Hence
\begin{equation}
\biggl(1+(2\pi)^k\biggl|\sum_{\ell=1}^d \og_{j,\ell}m_{i,\ell}\biggr|^k\biggr) \bigl|\hF(g;M)\bigr|
\leq 2S_{\infty,k}(F).
\end{equation}
The above inequality holds for any $1\leq i\leq n$ and $1\leq j\leq d$.
Note that $\sum_{\ell=1}^d \og_{j,\ell}m_{i,\ell}$ equals the 
entry of the matrix $M\trans g^{-1}$ 
at position $i,j$;
hence the maximum of $\bigl|\sum_{\ell=1}^d \og_{j,\ell}m_{i,\ell}\bigr|$
over all $i,j$ equals $\|M\trans g^{-1}\|_\infty$.
Hence we obtain \eqref{e:Fouriercoeff_bound}
with $\kappa=k$.

Finally, to extend to general $\kappa$, note that after possibly decreasing $k$ we may 
assume that $k-1<\kappa\leq k$.
If $\kappa=k$ then we are done;
hence we may now assume $k-1<\kappa<k$ (thus $\kappa>0$ and $k\geq1$).
The bound proved above holds both for $k$ and for $k':=k-1$;
and combining these we obtain
\begin{align*}
\bigl|\widehat{F}(g; M)\bigr|
\ll_k
\biggl(\frac{S_{\infty,k'}(F)}{1+\|M \trans g^{-1}\|_\infty^{k'}}\biggr)^{k-\kappa}
\biggl(\frac{S_{\infty,k}(F)}{1+\|M \trans g^{-1}\|_\infty^k}\biggr)^{\kappa-k'}.
\end{align*}
This implies \eqref{e:Fouriercoeff_bound},
by \eqref{fracSobnorm} (applied with $k'$ in place of $k$)
and since $(1+x^{k'})^{k-\kappa}(1+x^{k})^{\kappa-k'}\geq 1+x^{\kappa}$
for all $x\geq0$ (by H\"older's inequality).
\end{proof}

\section{Effective equidistribution of Hecke points}\label{ss:Heckepts}

In this section we collect the results about equidistribution of Hecke points which we will need
in the proof of our main theorem.
Our main reference will be  \cite{ClozelOhUllmo2001};
the proofs in that paper make use of spectral theory of automorphic forms
and the strong uniform bounds on matrix exponents of unitary representations
obtained in \cite{Oh2002}.

In this %
section we again assume $1\leq n<d$. 
Recall from \autoref{lem:mtx_relation} that we then have
\begin{align}\label{Dqdefrep}
D_q = q^{-\frac nd}\bpm I_{d-n} & \\ & qI_n\ebpm \in\SL_d(\R).
\end{align}

\begin{lemma}\label{lem:Hecke_decomp}
We have the disjoint coset decomposition 
\begin{equation}\label{e:Hecke_decomp}
{\rm SL}_d(\mathbb{Z})\, D_q\, {\rm SL}_d(\mathbb{Z}) 
= \bigsqcup_{\delta\in \Gamma^0(q) \backslash {\rm SL}_d(\mathbb{Z})} {\rm SL}_d(\mathbb{Z}) D_q \delta. 
\end{equation}
\end{lemma}
(Recall that $\Gamma^0(q)$ was defined in \eqref{e:Gamma0dnq_def}.)
\begin{proof}
Observe that 
\begin{equation}
D_q^{-1} {\rm SL}_d(\mathbb{Z}) D_q \cap {\rm SL}_d(\mathbb{Z})
= \Gamma^0(q). 
\end{equation}
Hence the group ${\rm SL}_d(\mathbb{Z})$ can be expressed as a disjoint union 
\begin{equation}
{\rm SL}_d(\mathbb{Z}) 
= \bigcup_{\delta\in \Gamma^0(q) \backslash {\rm SL}_d(\mathbb{Z})} 
\big(D_q^{-1} {\rm SL}_d(\mathbb{Z}) D_q \cap {\rm SL}_d(\mathbb{Z})\big) \delta. 
\end{equation}
Following \cite[Proposition~3.1]{Shi94}, 
we get \eqref{e:Hecke_decomp}. 
\end{proof}

We now follow the definition of Hecke operators given in \cite{ClozelOhUllmo2001}. 
For a complex valued function $\Phi$ on ${\rm SL}_d(\mathbb{Z}) \backslash {\rm SL}_d(\mathbb{R})$,
the Hecke operator for $D_q$ is defined as 
\begin{equation}\label{e:Heckeop_def}
(T_{D_q} \Phi) (g)
= \frac{1}{\#\bigl(\Gamma^0(q)\backslash {\rm SL}_d(\mathbb{Z})\bigr)} \sum_{\delta\in \Gamma^0(q)\backslash {\rm SL}_d(\mathbb{Z})} \Phi\left(D_q\delta g\right). 
\end{equation}
This makes sense since $[{\rm SL}_d(\mathbb{Z}): \Gamma^0(q)]<\infty$.

The map $T_{D_q}$ restricts to
a bounded linear operator on ${\rm L}^2(\SL_d(\Z)\bs\SL_d(\R))$.
We will later also encounter the dual operator,
$T_{D_q}^*$, i.e.\ the bounded linear operator on 
${\rm L}^2(\SL_d(\Z)\bs\SL_d(\R))$
which satisfies
\begin{equation}\label{TDqdual}
\big\langle T_{D_q} \Phi_1, \Phi_2\big\rangle
= \int_{{\rm SL}_d(\mathbb{Z})\backslash {\rm SL}_d(\mathbb{R})} 
\bigl[T_{D_q}\Phi_1\bigr](g) \overline{\Phi_2(g)}\, d\mu_0(g)
= \left<\Phi_1, T^*_{D_q} \Phi_2\right>
\end{equation}
for all $\Phi_1, \Phi_2\in {\rm L}^2({\rm SL}_d(\mathbb{Z})\backslash {\rm SL}_d(\mathbb{R}))$, 
where $\mu_0$ is the $\SL_d(\R)$-invariant probability measure on $\SL_d(\Z)\bs\SL_d(\R)$.
By mimicking the proof of
\cite[Proposition~3.39]{Shi94}
one verifies that $T_{D_q}^*$ is in fact the Hecke operator for $D_q^{-1}$.
Using also the fact that the map
$g\mapsto\trans g^{-1}$ is an automorphism of $\SL_d(\R)$ which maps $D_q$ to $D_q^{-1}$,
it follows from \autoref{lem:Hecke_decomp}
that 
\begin{align*}
{\rm SL}_d(\mathbb{Z})\, D_q^{-1}\, {\rm SL}_d(\mathbb{Z}) 
= \bigsqcup_{\delta\in \Gamma^0(q) \backslash {\rm SL}_d(\mathbb{Z})} {\rm SL}_d(\mathbb{Z}) D_q^{-1} \trans\delta^{-1},
\end{align*}
and hence
\begin{align}\label{TDqdualformula}
\bigl(T_{D_q}^*\Phi\bigr)(g)
=\frac{1}{\#\bigl(\Gamma^0(q)\backslash {\rm SL}_d(\mathbb{Z})\bigr)} \sum_{\delta\in \Gamma^0(q)\backslash {\rm SL}_d(\mathbb{Z})} \Phi\left(D_q^{-1}\,\trans\delta^{-1} g\right)
\end{align}
for any $\Phi\in{\rm L}^2(\SL_d(\Z)\bs\SL_d(\R))$.
In fact, we will take \eqref{TDqdualformula}
as a definition of $T_{D_q}^*\Phi$ for \textit{any}
function $\Phi:\SL_d(\Z)\bs\SL_d(\R)\to\C$.

Recall that we denote by $\theta$ the constant towards the Ramanujan conjecture
for Maass wave forms on $\SL_2(\Z)\bs\SL_2(\R)$.

\begin{proposition}\label{HeckeequidistrPROP}
Let $\kappa=\frac{d^2-1}2$, $\ve>0$,
and $k=\lceil\kappa+\ve\rceil$.
Then for every 
$\Phi\in {\rm C}_b^{k}({\rm SL}_d(\mathbb{Z}) \backslash {\rm SL}_d(\mathbb{R}))$, we have
\begin{equation}\label{e:Hecke_eq_pointwise} 
\left|(T_{D_q} \Phi)(I_d) - \int_{{\rm SL}_d(\mathbb{Z}) \backslash {\rm SL}_d(\mathbb{R})} \Phi(g)\, d\mu_0(g)\right| 
\ll_{\ve} S_{2,\kappa+\ve}(\Phi) 
\begin{cases}
q^{-\frac{1}{2}+\theta+\ve} 
& \text{ if } n = 1 \text{ and } d=2 \\[3pt]
q^{-\frac{\min\{n, d-n\}}{2}+ \ve} & \text{ otherwise.}
\end{cases}
\end{equation}
\end{proposition}
\begin{proof}
It is a known result that for every 
$\Phi\in {\rm L}^2({\rm SL}_d(\mathbb{Z})\backslash {\rm SL}_d(\mathbb{R}))$,
\begin{equation}\label{HECKEL2bound}
\left\|T_{D_q}\Phi - \int_{{\rm SL}_d(\mathbb{Z})\backslash {\rm SL}_d(\mathbb{R})} \Phi(g)\, d\mu_0(g)\right\|_2 \ll_{\ve} \|\Phi\|_2
\begin{cases}
q^{-\frac{1}{2}+\theta+\ve} 
& \text{ if } n = 1 \text{ and } d=2 \\[3pt]
q^{-\frac{\min\{n, d-n\}}{2}+ \ve} & \text{ otherwise.}
\end{cases}
\end{equation}
Indeed, if $d\geq3$ then \eqref{HECKEL2bound} follows by applying
\cite[Theorem 1.1 and p.\ 332 (Remark (3)) and Sec.\ 5.1]{ClozelOhUllmo2001}
for the group $G=\GL_d$,
and using the identification between $\SL_d(\Z)\bs\SL_d(\R)$ with $Z\GL_d(\Z)\bs\GL_d(\R)$,
where $Z$ is the center of $\GL_d(\R)$.
In the case $d=2$ one instead starts by noticing that
(see \cite[Ch.\ 3.1--2; in particular Thm.\ 3.24]{Shi94}; alternatively follow
the computation in 
\cite[p.\ 6599(top)]{LeeMarklof2017}):
\begin{align}\label{HeckeequidistrPROPpf1}
T_{D_q}\Phi=\frac1{q\prod_{p\mid q}(1+p^{-1})}\sum_{a^2\mid q}\mu(a)\sigma_1\Bigl(\frac q{a^2}\Bigr)T_{q/a^2}\Phi,
\end{align}
where the sum runs over all positive integers $a$ satisfying $a^2\mid q$,
and $\sigma_1(m):=\sum_{d\mid m}d$,
and where $T_m$ ($m\in\Z^+$) is the Hecke operator on ${\rm L}^2(\SL_2(\Z)\bs\SL_2(\R))$ defined by
\begin{align*}
(T_m\Phi)(g)=\frac1{\sigma_1(m)}\sum_{a\mid m}\sum_{b=0}^{\frac m{a}-1}\Phi\left(m^{-\frac12}\matr ab0{m/a} g\right)
\qquad (g\in\SL_2(\R)).
\end{align*}
Next, by \cite[Sec.\ 3]{dGaM2003} we have
$\bigl\|T_m\Phi-\int_{\SL_2(\Z)\bs\SL_2(\R)}\Phi\,d\mu_0\bigr\|_2\ll_{\ve} m^{-\frac12+\theta+\ve}\|\Phi\|_2$
for all $m\in\Z^+$.
Using this bound in \eqref{HeckeequidistrPROPpf1},
the triangle inequality,
and the fact that
$\sum_{a^2\mid q}\mu(a)\sigma_1\bigl(\frac q{a^2}\bigr)=q\prod_{p\mid q}(1+p^{-1})$,
we obtain \eqref{HECKEL2bound} for $d=2$.
(The last step was also carried out in
\cite[pp.\ 6599--6600]{LeeMarklof2017}).

Finally, 
after recalling the definition \eqref{fracSobnorm},
the bound in \eqref{e:Hecke_eq_pointwise}  is deduced from
\eqref{HECKEL2bound} as in the proof of 
\cite[Lemma 5]{StrombergssonVenkatesh2005}.
\end{proof}

\section{Matrix Kloosterman sums}
\label{s:MATRIXKLOOSTERMANsec}

In this section we use the following notation:
\begin{align*}
e_q(x):=e^{2\pi ix/q}
\qquad (q\in\Z^+,\: x\in\R).
\end{align*}

For $n,q\in\Z^+$ %
and $A, B \in \M_n(\Z/q\Z)$, 
we define 
\begin{equation}\label{Kndefrep}
    K_n(A, B; q) = \sum_{X \in \GL_n(\Z/q\Z)} e_q({\rm tr}(AX + BX^{-1})).
\end{equation}

\subsection{Prime moduli}
For $p$ a prime number, we denote
the field $\Z/p\Z$ by $\F_p$.
In \cite[Corollary  1.11]{mEaT2021},  
Erd\'elyi and T\'oth 
have recently proved that 
for any prime number $p$ and
any $A,B\in \M_n(\F_p)$, not both $\vecnull$, 
\begin{equation}\label{ETres1}
\bigl|K_n(A, B; p)\bigr| \leq 2p^{n^2 - n+1}.
\end{equation}
In fact, the main result of that paper \cite{mEaT2021} 
is that if both $A$ and $B$ belong to $\GL_n(\F_p)$,
then 
the following much sharper bound holds:
\begin{equation}\label{ETres2}
\bigl|K_n(A, B; p)\bigr| \ll p^{(3n^2-\delta_n)/4},
\end{equation}
where $\delta_n=0$ if $n$ is even and $\delta_n=1$ if $n$ is odd
\cite[Theorem 1.8]{mEaT2021}.

\subsection{General moduli}\label{genmoduliSEC}

This case is easily reduced to the case of prime power moduli,
using the standard multiplicativity relation:
\begin{lemma}\label{multiiplicativityLEM}
Let $q=\prod_{j=1}^rq_j$ where $q_1,\ldots,q_r\in\Z^+$ are pairwise relatively prime,
and for each $j$, 
let $c_j\in(\Z/q_j\Z)^\times$ be a multiplicative inverse of
$\prod_{i\neq j}q_i$ modulo $q_j$.
Then for any 
$A,B\in\M_n(\Z/q\Z)$,
\begin{align}\label{multiiplicativityLEMres}
K_n(A,B;q)=\prod_{j=1}^r K_n(c_jA,c_jB;q_j)
\end{align}
\end{lemma}
\begin{proof}
For any integer $a$ we have $a\equiv\sum_{j=1}^r (q/q_j)c_ja\mod q$, and so
$e_q(a)=\prod_{j=1}^r e_{q_j}(c_ja)$.
In particular, for any $X\in\GL_n(\Z/q\Z)$,
\begin{align}\label{multiiplicativityLEMpf1}
e_q\bigl({\rm tr}(AX+BX^{-1})\bigr)
=\prod_{j=1}^r \, e_{q_j}\!\bigl({\rm tr}(c_jAX+c_jBX^{-1})\bigr).
\end{align}
On the right-hand side, the
$j$th factor depends only on $X\mod q_j$,
and when $X$ runs through $\GL_n(\Z/q\Z)$,
the $r$-tuple $\langle X\mod q_j\rangle_{j=1,\ldots r}$
runs through the Cartesian product $\prod_{j=1}^r\GL_n(\Z/q_j\Z)$.
Hence when we sum \eqref{multiiplicativityLEMpf1}
over $X\in\GL_n(\Z/q\Z)$,
we obtain \eqref{multiiplicativityLEMres}.
\end{proof}

\begin{lemma}\label{KvanishingLEM2}
Let $q\in\Z^+$ and $A,B\in\M_n(\Z/q\Z)$,
and let $\ell$ be a divisor of $\frac{q}{\prod_{p\mid q} p}$.
Assume also that $\ell\mid B$.
Then
\begin{align}\label{KvanishingLEM2res}
K_n(A,B;q)=\begin{cases}
0&\text{if }\: \ell\nmid A
\\
\ell^{n^2}\, K_n(\ell^{-1}A,\ell^{-1}B;\ell^{-1}q)
&\text{if }\: \ell\mid A.
\end{cases}
\end{align}
\end{lemma}
\begin{proof}
Because of the assumption on $\ell$,
we can fix a subset $R$ of $\GL_n(\Z/q\Z)$ containing exactly one representative for
each congruence class in
$\GL_n(\Z/\frac q{\ell}\Z)$,
and then the map
$\langle Y,Z\rangle\mapsto\frac q{\ell}Y+Z$ is a bijection from
$\M_n(\Z/\ell\Z)\times R$ onto $\GL_n(\Z/q\Z)$.
Using this parametrization in \eqref{Kndefrep},
writing $B=\ell B'$ with $B'\in\M_n(\Z/\frac q{\ell}\Z)$
and noticing that $(\frac q{\ell}Y+Z)^{-1}\equiv Z^{-1}\mod\frac q{\ell}$,
we get
\begin{align*}
K_n(A,B;q)=\sum_{Y\in\M_n(\Z/\ell\Z)}e_{\ell}\bigl({\rm tr}(AY)\bigr)
\sum_{Z\in R}e_q\bigl({\rm tr}(AZ)\bigr)
e_{q/\ell}\bigl({\rm tr}(B'Z^{-1})\bigr).
\end{align*}
Here the sum over $Y$ equals $\ell^{n^2}$ if $\ell\mid A$,
and otherwise vanishes. 
Hence we obtain \eqref{KvanishingLEM2res}.
\end{proof}

\subsection{Prime power moduli}\label{primepowerSEC}
In the case of higher prime power moduli, we will prove a 
bound on $K_n(A,B;q)$ by direct and elementary computations;
see \autoref{KboundgenqPROP} below for the final result. 
We remark that bounds of a similar nature,
but more precise and in certain respects stronger,
have independently been obtained in the recent paper
\cite{ETZ} by Erd\'elyi, T\'oth and Z\'abr\'adi.
However we choose to include the proofs in this section in order to 
make our paper more self-contained
and because we use a shortcut that leads to an upper bound which is sufficient for our needs.
We further emphasize that, using the bounds from \cite{ETZ} instead, 
would not lead to an improvement of the exponents in our main result, \autoref{MAINTHM}.

For $q\in\Z^+$ 
and
$A, B \in \M_n(\Z/q\Z)$, 
we define 
\begin{equation}\label{scrCdef}
\mathcal{C}_q(A, B) = \{Y \in \GL_n(\Z/q \Z) \, : \, AY \equiv Y^{-1}B \mod{q}\}.
\end{equation}
For any prime $p$ and $C, D \in \M_n(\F_p)$, we also introduce the following
matrix Gauss sum:
\begin{equation}\label{GpCDdef}
    G_p(C, D) = \sum_{Z \in \M_n(\F_p)} e_p({\rm tr}(CZ^2 + DZ)).
\end{equation}

\begin{lemma} \label{thm:Taylor_csq}
Let $q=p^\beta$ where $p$ is a prime and $\beta\geq2$,
and set $\alpha=\lfloor\beta/2\rfloor$.
Let $A,B\in\M_n(\Z/q\Z)$, and assume $A,B\not\equiv 0\mod p$.
If $\beta$ is even, then 
\begin{equation}\label{thm:Taylor_csqres1}
    |K_n(A, B; q)| \leq p^{\alpha n^2}  \# \mathcal{C}_{p^\alpha}(A, B).
\end{equation}
If $\beta$ is odd, then 
\begin{equation}\label{thm:Taylor_csqres2}
    |K_n(A, B; q)| \leq p^{\alpha n^2}  \# \mathcal{C}_{p^\alpha}(A, B) 
\cdot\max\bigl\{|G_p(C,D)|\col C,D\in\M_n(\F_p),\: C\neq 0\bigr\}.
\end{equation}
\end{lemma}
\begin{proof}
Fix a subset $R$ of $\GL_n(\Z/q\Z)$ containing exactly one representative
for each congruence class in $\GL_n(\Z/p^\alpha\Z)$.
Let us first assume that $\beta$ is even; thus $\beta=2\alpha$.
Then the map $\langle Y,Z\rangle\mapsto Y(I+p^\alpha Z)$
is a bijection from $R\times\M_n(\Z/p^\alpha\Z)$ onto $\GL_n(\Z/q\Z)$.
Using this parametrization in \eqref{Kndefrep},
and the fact that $(I+p^\alpha Z)^{-1}\equiv I-p^{\alpha}Z\mod q$,
we obtain
\begin{align*}
K_n(A,B;q)=\sum_{Y\in R}e_q({\rm tr}(AY+BY^{-1}))\sum_{Z\in\M_n(\Z/p^\alpha\Z)}e_{p^\alpha}\bigl({\rm tr}\bigl((AY-Y^{-1}B)Z\bigr)\bigr).
\end{align*}
Here the inner sum vanishes unless $AY-Y^{-1}B\equiv0\mod p^\alpha$.
Hence we obtain the bound in \eqref{thm:Taylor_csqres1}.

Next assume that $\beta$ is odd, i.e.\ $\beta=2\alpha+1$.
Then we also fix a subset $R'$ of $\M_n(\Z/p^{\alpha+1}\Z)$
containing exactly one representative for each congruence class in $\M_n(\F_p)$.
Then the map $\langle Y,Z_1,Z_2\rangle\mapsto Y(I+p^\alpha Z_1+p^{\alpha+1}Z_2)$
is a bijection from $R\times R'\times\M_n(\Z/p^\alpha\Z)$ onto $\GL_n(\Z/q\Z)$.
Using this in \eqref{Kndefrep},
together with the fact that
$(I+p^\alpha Z_1+p^{\alpha+1}Z_2)^{-1}
\equiv I-p^\alpha Z_1-p^{\alpha+1}Z_2+p^{2\alpha}Z_1^2\mod q$,
we obtain
\begin{align*}
K_n(A,B;q)=\sum_{Y\in R}e_q({\rm tr}(AY+Y^{-1}B))
\sum_{Z_1\in R'}e_{p^{\alpha+1}}\bigl({\rm tr}\bigl((AY-Y^{-1}B)Z_1+p^\alpha Y^{-1}BZ_1^2\bigr)\bigr)
\hspace{30pt}
\\
\times\sum_{Z_2\in\M_n(\Z/p^\alpha\Z)}e_{p^\alpha}\bigl({\rm tr}\bigl((AY-Y^{-1}B)Z_2\bigr)\bigr).
\end{align*}
Here the sum over $Z_2$ vanishes unless $AY-Y^{-1}B\equiv0\mod p^\alpha$;
hence we obtain
\begin{align*}
K_n(A,B;q)=p^{\alpha n^2}\sum_{\substack{Y\in R\\ AY-Y^{-1}B\equiv0 \mod p^\alpha}}e_q({\rm tr}(AY+Y^{-1}B))
\cdot G_p\biggl(Y^{-1}B,\frac{AY-Y^{-1}B}{p^\alpha}\biggr),
\end{align*}
and this leads to the bound in \eqref{thm:Taylor_csqres2}.
\end{proof}

In order to make the bound in \autoref{thm:Taylor_csq}
useful, we need to bound $\#\scrC_{p^\alpha}(A,B)$.
We will first treat the case $\alpha=1$,
and for this we will need the following lemma.\footnote{We learnt about this fact 
from MathOverflow, question 41784 (``Roots of permutations'') \cite{MO}.}
\begin{lemma}\label{mathoverflowLEMMA}
Let $G$ be a finite group with the property that 
every irreducible linear representation of $G$ over $\C$
is either realizable over $\R$ or has non-real character.
Let $f:G\to\Z_{\geq0}$ be the function that counts the number of square roots
of each element in $G$,
viz., $f(g):=\#\{x\in G\col x^2=g\}$.
Then $f(g)\leq f(e)$ for all $g\in G$.
\end{lemma}
\begin{proof}
Clearly $f$ is a class function (i.e., invariant under conjugation),
and hence 
$f=\sum_\chi\langle f,\chi\rangle \chi$
where the sum is taken over all irreducible characters of $G$
\cite[Theorem 6]{jS77}.
Here
\begin{align*}
\langle f,\chi\rangle
=\frac1{\#G}\sum_{g\in G} f(g)\chi(g^{-1})
=\frac1{\#G}\sum_{g\in G}\sum_{x\in G} \delta_{x^2=g}\, \chi(g^{-1})
=\frac1{\#G}\sum_{x\in G}\chi(x^{-2})
=\frac1{\#G}\sum_{x\in G}\chi(x^{2}).
\end{align*}
This formula, together with the assumption of the lemma and
\cite[Prop.\ 39]{jS77},
implies that $\langle f,\chi\rangle\in\{0,1\}$ for all $\chi$.
Hence $f=\sum_{\chi\in S}\chi=\sum_{\chi\in S}\tre\chi$,
where $S$ is the set of those $\chi$ for which $\langle f,\chi\rangle=1$.
Hence for any $g\in G$,
$f(g)=\sum_{\chi\in S}\tre\chi(g)
\leq \sum_{\chi\in S}\chi(e)
=f(e)$.
\end{proof}

\begin{proposition}\label{CpABboundPROP}
For every prime $p$ and every $A,B\in\M_n(\F_p)$, not both zero, we have
\begin{align}\label{CpABboundPROPres}
\#\scrC_p(A,B)\ll p^{(n-1)^2+1}.
\end{align}
More precisely,
if $A$ is in $\mathrm{GL}_n(\F_p)$, then
\begin{equation}\label{CpABboundPROPres2}
\# \mathcal{C}_p(A, B) \ll p^{\frac{1}{2}(n^2-\delta_n)}, 
\end{equation}
where $\delta_n=\frac{1-(-1)^{n}}{2}$, 
while if $r=\rank A$ satisfies $1\leq r\leq n-1$ then
\begin{equation}\label{CpABboundPROPres3}
\# \mathcal{C}_p(A, B) \ll p^{n^2-2r(n-r)}.
\end{equation} 
The implied constants in all three bounds are absolute.
\end{proposition}
(A slightly more precise bound is given in \cite[Theorem 1.6]{ETZ}.)
\begin{proof}
It suffices to prove \eqref{CpABboundPROPres2} and \eqref{CpABboundPROPres3},
since these imply \eqref{CpABboundPROPres}.
It is immediate from the definition,
\eqref{scrCdef},
that $\rank A\neq\rank B$ implies
$\scrC_p(A,B)=\emptyset$;
hence we may assume that
$r=\rank A=\rank B$.

Let us first assume $r=n$,
i.e.\ $A$ and $B$ both lie in $\GL_n(\F_p)$.
Substituting $Z=AY$ in the definition of $\scrC_p(A,B)$,
it follows that $\#\scrC_p(A,B)$ equals the number of elements
$Z\in \GL_n(\F_p)$ with $Z^2=AB$.
Also the group $\GL_n(\F_p)$
is known to have the property that
all of its linear representations are
either realizable over $\R$ or have non-real character
\cite[Ch.\ III, 12.6]{aZ81}.
Using these facts
in combination with \autoref{mathoverflowLEMMA},
we conclude that
\begin{align}\label{CpABboundPROPpf1}
\#\scrC_p(A,B)
\leq \#\{Z\in \GL_n(\F_p)\col Z^2=I\}.
\end{align}

The cardinality on the right-hand side of \eqref{CpABboundPROPpf1} is easy
to calculate:
if $Z\in\GL_n(\F_p)$ satisfies $Z^2=I$, then all eigenvalues of $Z$
must equal $\pm1$,
and hence
$Z$ is conjugate over $\F_p$ to a matrix $J$ in Jordan canonical form,
say with Jordan blocks $J_1,\ldots,J_k$ (in this order) where $J_i$ is the $n_i\times n_i$ matrix
\begin{align*}
J_i=\begin{pmatrix}
\ve_i &  1    &    &&&
\\
       & \ve_i & 1 & & 0 &
\\
 & & \cdots &&&
\\
 &&& \cdots &&
\\
& 0 &&& \ve_i & 1
\\
&&&&& \ve_i
\end{pmatrix}
\end{align*}
with $\ve_i\in\{1,-1\}$.
Let us first assume $p\neq2$.
Then $J^2=I$ forces $n_i=1$ for all $i$,
and so we conclude that for every matrix $Z$ belonging
to the set on the right-hand side of \eqref{CpABboundPROPpf1},
there is a unique $0\leq a\leq n$ such that $Z$ 
is conjugate over $\F_p$ to the diagonal matrix $D_a$ having
$a$ 1's and $(n-a)$ $-1$'s along the diagonal, in this order.
Hence the right-hand side of \eqref{CpABboundPROPpf1} equals
\begin{align*}
\sum_{a=0}^n\#\{TD_aT^{-1}\col T\in\GL_n(\F_p)\}
=\sum_{a=0}^n\#(\GL_n(\F_p)/C(D_a)),
\end{align*}
where $C(D_a)$ is the centralizer of $D_a$ in $\GL_n(\F_p)$.
But $C(D_a)$ consists of exactly the matrices in $\GL_n(\F_p)$
which are block diagonal with blocks of sizes $a,n-a$,
and so 
\begin{align*}
\#C(D_a)=\#\GL_{a}(\F_p)\#\GL_{n-a}(\F_p)
=\prod_{j=0}^{a-1}(p^a-p^j)\prod_{j=0}^{n-a-1}(p^{n-a}-p^j).
\end{align*}
If $a\in\{0,n\}$ this should of course be understood to say $\#C(D_a)=\#\GL_n(\F_p)=\prod_{j=0}^{n-1}(p^n-p^j)$.
Hence the right-hand side of \eqref{CpABboundPROPpf1} equals
\begin{align*}
\sum_{a=0}^n\frac{\prod_{j=0}^n(p^n-p^j)}{\prod_{j=0}^{a-1}(p^a-p^j)\prod_{j=0}^{n-a-1}(p^{n-a}-p^j)}.
\end{align*}
Noticing that $\prod_{j=0}^{a-1}(p^a-p^j)\asymp p^{a^2}$ uniformly over all primes $p$ and all $a\geq0$,
the above expression is seen to be
\begin{align*}
\asymp\sum_{a=0}^n p^{n^2-a^2-(n-a)^2}
\asymp p^{\frac12(n^2-\delta_n)},
\end{align*}
with
$\delta_n=\frac{1-(-1)^{n}}{2}$
and we have thus proved 
\eqref{CpABboundPROPres2} in the case $r=n$, $p\neq 2$.

Next we assume $r=n$, $p=2$.
Then $J^2=I$ forces $n_i\in\{1,2\}$ for all $i$,
and thus, since $-1=1$ in $\F_2$,
every Jordan block appearing in $J$ must equal
$\oJ:=\bigl(1\bigr)$
or $\oJ':=\matr1101$.
Hence for every matrix $Z$ belonging
to the set in the right-hand side of \eqref{CpABboundPROPpf1},
there is a unique $0\leq a\leq \lfloor n/2\rfloor$ such that $Z$ 
is conjugate over $\F_2$ to the block diagonal matrix $D_a'$ having
$a$ blocks $\oJ'$ and $n-2a$ blocks $\oJ$ along the diagonal, in this order.
It follows that the right-hand side of \eqref{CpABboundPROPpf1} equals
\begin{align}\label{CpABboundPROPpf3}
\sum_{0\leq a\leq\lfloor n/2\rfloor}\#\{TD_a'T^{-1}\col T\in\GL_n(\F_2)\}
=\sum_{0\leq a\leq\lfloor n/2\rfloor}\#(\GL_n(\F_2)/C(D_a')).
\end{align}
Here we claim that
\begin{align}\label{CpABboundPROPpf2}
\#C(D_a')=2^{a(2n-3a)}\#\GL_a(\F_2)\#\GL_{n-2a}(\F_2).
\end{align}
To prove this, note that $X\in\GL_n(\F_2)$ commutes with $D_a'$ if and only if 
$X$ commutes with
$D_a'-I$,
and $D_a'-I$ can be conjugated, by a permutation matrix,
into the matrix 
\begin{equation} 
U_a:=\begin{pmatrix} \vecnull & I_a & \vecnull \\ \vecnull & \vecnull & \vecnull \\ \vecnull & \vecnull & \vecnull \end{pmatrix}, 
\end{equation}
with block sizes $a,a,n-2a$ in this order.
Hence $\#C(D_a')=\#C(U_a)$,
and writing $X=\bigl(X_{ij}\bigr)_{i,j=1,2,3}$ with block sizes $a,a,n-2a$,
we find that $X$ belongs to $C(U_a)$ if and only if
$X_{11}=X_{22}$ and the three matrices $X_{21},X_{23},X_{31}$ vanish.
Furthermore, by considering the determinant,
such a block matrix $X$ 
is invertible
if and only if both $X_{11}=X_{22}$ and $X_{33}$ 
are invertible.
Hence we obtain the formula in
\eqref{CpABboundPROPpf2}.

It follows from \eqref{CpABboundPROPpf3} and \eqref{CpABboundPROPpf2}
that the right-hand side of 
\eqref{CpABboundPROPpf1} is
\begin{align*}
\asymp\sum_{0\leq a\leq\lfloor n/2\rfloor} 2^{n^2-a(2n-3a)-a^2-(n-2a)^2}
=\sum_{0\leq a\leq\lfloor n/2\rfloor} 2^{\frac12n^2-2(a-\frac12n)^2}
\asymp 2^{\frac12(n^2-\delta_n)}.
\end{align*}
Hence 
\eqref{CpABboundPROPres2} also holds
in the case $r=n$, $p=2$.

\vspace{5pt}

It remains to consider the case when $r=\rank A=\rank B$ satisfies
$1\leq r\leq n-1$;
we then wish to prove the bound \eqref{CpABboundPROPres3}.
We may of course assume that $\scrC_p(A,B)$ is non-empty;
thus let us fix some $Y_0\in\scrC_p(A,B)$,
and set $A_0:=AY_0=Y_0^{-1}B\neq \vecnull$.
Then for any $Y\in\GL_n(\F_p)$,
the condition $Y\in\scrC_p(A,B)$ is equivalent with
$A_0Y_0^{-1}Y=Y^{-1}Y_0A_0$,
viz., $Y_0^{-1}Y\in\scrC_p(A_0,A_0)$.
Hence
\begin{align*}
\#\scrC_p(A,B)=\#\scrC_p(A_0,A_0).
\end{align*}

Let $V=\{A_0\vecx\col\vecx\in\F_p^n\}=\{A\vecx\col\vecx\in\F_p^n\}$;
this is an $r$-dimensional subspace of $\F_p^n$. 
Let us note that
\begin{align}\label{CpABboundRELPROPpf1}
\forall Y\in\scrC_p(A_0,A_0):\qquad
Y|_V\in\GL(V).
\end{align}
Indeed, $Y\in\scrC_p(A_0,A_0)$
implies $A_0Y\vecx=Y^{-1}A_0\vecx$ for all $\vecx\in\F_p^n$;
hence 
$V=Y^{-1}(V)$,
and so $Y|_V\in\GL(V)$.

Let us fix $\vecb_1,\ldots,\vecb_{n-r}$ to be a basis of 
some complementary subspace
of $V$ in $\F_p^n$.
Now let $Y_1\in\GL(V)$ be given.
Then for any $Y\in\scrC_p(A_0,A_0)$ with $Y|_V=Y_1$ (if such a $Y$ exists at all),
we have
$A_0Y(\vecb_j)=Y^{-1}A_0(\vecb_j)=Y_1^{-1}A_0(\vecb_j)$
for every $j$.
This means that $Y(\vecb_j)$ belongs to the preimage of
$Y_1^{-1}A_0(\vecb_j)$ under $A_0$. 
Since $A_0$ is a linear map on $\mathbb{F}_p^n$ of rank $r$, and the preimage of $Y_1^{-1}A_0(\vecb_j)$ under $A_0$ is not empty, the preimage is an affine linear subspace of $\mathbb{F}_p^n$ of dimension $n-r$. 
It follows that the tuple $Y(\vecb_1),\ldots,Y(\vecb_{n-r})$
can be chosen in at most $p^{(n-r)^2}$ ways.
Now since $Y$ is determined by linearity from $Y_1$ and
the elements $Y(\vecb_1),\ldots,Y(\vecb_{n-r})$,
we conclude that:
\begin{align}\label{CpABboundRELPROPpf2}
\forall Y_1\in\GL(V):
\qquad
\#\{Y\in\scrC_p(A_0,A_0)\col Y|_V=Y_1\}\leq p^{(n-r)^2}.
\end{align}

It follows from \eqref{CpABboundRELPROPpf1}, \eqref{CpABboundRELPROPpf2}
that
\begin{align*}
\#\scrC_p(A,B)=
\#\scrC_p(A_0,A_0)\leq p^{(n-r)^2}\#\GL(V)\leq p^{(n-r)^2+r^2},
\end{align*}
i.e.\ \eqref{CpABboundPROPres3} holds.
\end{proof}

We now turn to the problem of bounding 
$\#\scrC_{p^\alpha}(A,B)$
for $\alpha\geq2$.
We will start by proving a bound on the following
quantity, which turns out to be relevant 
for bounding 
both $\#\scrC_{p^\alpha}(A,B)$
and the Gauss sum $G_p(C,D)$.
For any $C \in \M_n(\F_p)$, we set
\begin{equation}\label{dCdef}
d(C)= \dim_{\F_p} \scrA(C),
\qquad\text{where}
\quad
\scrA(C) = \{Z \in \M_n(\F_p) \, : \, CZ + ZC = 0 \}.
\end{equation}
Note that $\scrA(C)$ is a vector subspace of $\M_n(\F_p)\cong \F_p^{n^2}$, 
and $\#\scrA(C)=p^{d(C)}$.  
\begin{lemma}\label{dCBOUNDlem}
For any $C\in\M_n(\F_p)\setminus\{0\}$,
if either $p>2$ or $C\neq I$ then
\begin{align}\label{dCBOUNDlemRES}
d(C)\leq (n-1)^2+1.
\end{align}
\end{lemma}
\begin{proof}
Let us fix an algebraic closure 
$\oF_p$
of $\F_p$.
For any $C'\in\M_n(\oF_p)$ we define
\begin{align*}
d(C')=\dim_{\oF_p}\{Z\in\M_n(\oF_p)\col C'Z+ZC'=0\}.
\end{align*}
Note that this formula is consistent with \eqref{dCdef} if $C'\in\M_n(\F_p)$.
Note also that $d(C)=d(TCT^{-1})$ for any $T\in\GL_n(\oF_p)$.
We may now choose $T\in\GL_n(\oF_p)$
so that $C':=TCT^{-1}$ is in
Jordan canonical form.
Thus let us assume that $C'$ has 
Jordan blocks $J_1,\ldots,J_k$ (in this order)
where 
$J_i$ is the $n_i\times n_i$ matrix
\begin{align}\label{JORDANblock}
J_i=\begin{pmatrix}
\lambda_i &  1    &    &&&
\\
       & \lambda_i & 1 & & 0 &
\\
 & & \cdots &&&
\\
 &&& \cdots &&
\\
& 0 &&& \lambda_i & 1
\\
&&&&& \lambda_i
\end{pmatrix}
\hspace{50pt}
(\lambda_i\in\oF_p).
\end{align}
Writing $Z$ in block decomposed form as 
$Z=(Z_{i,j})$ with $Z_{i,j}\in\M_{n_i,n_j}(\oF_p)$
for $i,j\in\{1,\ldots,k\}$,
one notes that $C'Z+ZC'=0$ holds
if and only if $J_iZ_{i,j}+Z_{i,j}J_j=0$ for all pairs
$i,j$.
Hence
\begin{align}\label{dCBOUNDlempf1}
d(C)=d(C')=\sum_{i=1}^k\sum_{j=1}^k
\dim_{\oF_p}\{Z\in\M_{n_i,n_j}(\oF_p)
\col J_iZ+ZJ_j=0\}.
\end{align}

We now claim that
\begin{align}\label{dCBOUNDlempf2}
\dim_{\oF_p}\{Z\in\M_{n_i,n_j}(\oF_p)
\col J_iZ+ZJ_j=0\}
=\begin{cases}
\min(n_i,n_j)&\text{if }\: \lambda_i=-\lambda_j
\\
0&\text{if }\: \lambda_i\neq -\lambda_j.
\end{cases}
\end{align}
To prove this, note that writing 
$Z=(z_{a,b})\in\M_{n_i,n_j}(\oF_p)$,
the relation $J_iZ+ZJ_j=0$ holds if and only if
\begin{align}\label{dCBOUNDlempf3}
(\lambda_i+\lambda_j)z_{a,b}
+z_{a+1,b}+z_{a,b-1}=0
\qquad(\forall a\in\{1,\ldots,n_i\},\: b\in\{1,\ldots,n_j\}),
\end{align}
where we understand that $z_{n_i+1,b}=0$ for all $b$
and $z_{a,0}=0$ for all $a$.
Let us first assume $\lambda_i=-\lambda_j$,
so that the equation \eqref{dCBOUNDlempf3}
simply reads $z_{a+1,b}+z_{a,b-1}=0$.
This implies that the matrix entries are alternating along each diagonal,
viz., for any fixed $a'\in\{1,\ldots,n_i\}$ and $b'\in\{1,\ldots,n_j\}$
with either $a'=1$ or $b'=1$,
we have $z_{a',b'}=-z_{a'+1,b'+1}=z_{a'+2,b'+2}=\cdots=
(-1)^{\ell}z_{a'+\ell,b'+\ell}$ where $\ell=\min(n_i-a',n_j-b')$.
But we also have $z_{a',1}=0$ for all $a'\geq2$ (by \eqref{dCBOUNDlempf3} applied with $a=a'-1$ and $b=1$)
and $z_{n_i,b'}=0$ for all $b'<n_j$ (by \eqref{dCBOUNDlempf3} applied with $a=n_i$ and $b=b'+1$).
Hence all the diagonals which start at $z_{a',1}$ with $1<a'\leq n_i$ vanish completely;
and if $n_i<n_j$ then also the diagonals which start at $z_{1,b'}$ with 
$1\leq b'\leq n_j-n_i$ vanish completely.
Conversely one verifies that any matrix having vanishing diagonals as just described,
and the remaining diagonals alternating,
satisfies \textit{all} the relations in \eqref{dCBOUNDlempf3}.
Furthermore, there are exactly $\min(n_i,n_j)$ diagonals which are not forced to vanish.
Hence \eqref{dCBOUNDlempf2} holds
in the case $\lambda_i=-\lambda_j$.

Next we assume $\lambda_i\neq-\lambda_j$.
Then, applying \eqref{dCBOUNDlempf3} for
$b=1$ and $a=n_i,n_i-1,\ldots,1$ (in this order)
we get $z_{a,1}=0$ for all $a$.
Next, applying \eqref{dCBOUNDlempf3} for
$b=2$ and $a=n_i,n_i-1,\ldots,1$ 
gives $z_{a,2}=0$ for all $a$.
This may be repeated successively for $b=3,\ldots,n_j$,
finally giving $Z=0$.
Hence \eqref{dCBOUNDlempf2} holds
also in the case $\lambda_i\neq-\lambda_j$.

Using \eqref{dCBOUNDlempf2} in \eqref{dCBOUNDlempf1},
we obtain
\begin{align}\label{dCBOUNDlempf4}
d(C)=\sum_{i=1}^k
\sum_{\substack{j=1\\(\lambda_j=-\lambda_i)}}^k\min(n_i,n_j).
\end{align}
This implies
\begin{align}\label{dcBOUNDlempf100}
d(C)\leq \sum_{i=1}^k\sum_{j=1}^k n_in_j\Bigl(\tfrac12+\tfrac12\delta_{n_i=n_j=1}\Bigr),
\end{align}
and using $\sum_i n_i=n$, 
the right-hand side of \eqref{dcBOUNDlempf100} is seen to equal
$\frac12 n^2+\frac12 m^2$,
where $m:=\#\{i\col n_i=1\}$.
If $n_i\geq2$ for some $i$ 
then $m\leq n-2$,
and so
\begin{align*}
d(C)\leq\tfrac12 n^2+\tfrac12(n-2)^2=(n-1)^2+1,
\end{align*}
i.e.\ \eqref{dCBOUNDlemRES} holds.

Hence from now on we may assume that $n_i=1$ for all $i$,
viz., $C$ is diagonalizable.
Then \eqref{dCBOUNDlempf4} gives
\begin{align*}
d(C)=\sum_{\lambda\in S}d_\lambda d_{-\lambda},
\end{align*}
where $S=\{\lambda_1,\ldots,\lambda_k\}$ 
is the set of eigenvalues of $C$
and $d_\lambda=\sum_{i:\: \lambda_i=\lambda}n_i$
is the dimension of the eigenspace for $\lambda$.
First assume $p\neq2$. Then $\lambda\neq-\lambda$ for all $\lambda\in\oF_p\setminus\{0\}$
and thus we can choose a subset $S'\subset S$ such that
$S\setminus\{0\}\subset S'\cup(-S')$ and $S'\cap(-S')=\emptyset$.
Now
\begin{align*}
\sum_{\lambda\in S}d_\lambda d_{-\lambda}
=d_0^2+2\sum_{\lambda\in S'}d_\lambda d_{-\lambda}
\leq d_0^2+\frac12\sum_{\lambda\in S'}(d_\lambda+d_{-\lambda})^2
\leq d_0^2+\frac12\biggl(\sum_{\lambda\in S'}(d_\lambda+d_{-\lambda})\biggr)^2
\\
=d_0^2+\tfrac12(n-d_0)^2,
\end{align*}
and since $C\neq0$ we have $0\leq d_0\leq n-1$,
so that 
\begin{align*}
d_0^2+\tfrac12(n-d_0)^2
\leq (n-1)^2+1
\end{align*}
(with equality if and only if $n=2$ and $d_0=0$).
Hence \eqref{dCBOUNDlemRES} holds.

Finally assume $p=2$.
Then 
$d(C)=\sum_{\lambda\in S} d_\lambda^2$.
If $S$ is a singleton set, 
say $S=\{\lambda\}$,
then $C'=\lambda I$, and thus also $C=\lambda I$.
This forces $\lambda\in\F_2$,
and so $C\in\{0,I\}$.
Hence if $C\notin\{0,I\}$
then $\#S\geq2$,
and choosing some element $\lambda'\in S$ we get
\begin{align*}
d(C)=d_{\lambda'}^2+\sum_{\lambda\in S\setminus\{d_{\lambda'}\}}d_\lambda^2
\leq d_{\lambda'}^2+(n-d_{\lambda'})^2
\leq (n-1)^2+1,
\end{align*}
since $1\leq d_{\lambda'}\leq n-1$.
Thus \eqref{dCBOUNDlemRES} holds.
\end{proof}

For $p=2$ we will also need the following bound of similar type.
\begin{lemma}\label{vardCBOUNDlem}
For any $C\in\M_n(\F_2)$, 
\begin{align}\label{vardCBOUNDlemres}
\dim_{\F_2}\{Z\in\M_n(\F_2)\col Z+CZ+ZC=0\}\leq \frac12n^2.
\end{align}
\end{lemma}
\begin{proof}
The proof of \autoref{dCBOUNDlem} carries over with some modifications.
Introducing the Jordan decomposition of $C$ exactly as in that proof,
the analogue of \eqref{dCBOUNDlempf1} now says that the dimension in the 
left hand side of \eqref{vardCBOUNDlemres} equals
\begin{align}\label{vaddCBOUNDlempf1}
\sum_{i=1}^k\sum_{j=1}^k
\dim_{\oF_2}\{Z\in\M_{n_i,n_j}(\oF_2)
\col Z+J_iZ+ZJ_j=0\}.
\end{align}
Here we have, just as in \eqref{dCBOUNDlempf2},
\begin{align*}
\dim_{\oF_2}\{Z\in\M_{n_i,n_j}(\oF_2)\col Z+J_iZ+ZJ_j=0\}=\begin{cases}
\min(n_i,n_j)&\text{if }\: 1+\lambda_i+\lambda_j=0
\\
0&\text{if }\: 1+\lambda_i+\lambda_j\neq0.
\end{cases}
\end{align*}
(Indeed, $Z+J_iZ+ZJ_j=0$ is equivalent with
\eqref{dCBOUNDlempf3}
but with $\lambda_i+\lambda_j$ replaced by $1+\lambda_i+\lambda_j$.)
Hence the dimension in the 
left hand side of \eqref{vardCBOUNDlemres} is
\begin{align*}
\sum_{i=1}^k\sum_{\substack{j=1\\(\lambda_j=-\lambda_i-1)}}^k\min(n_i,n_j)
\leq \sum_{i=1}^k\sum_{\substack{j=1\\(\lambda_j=-\lambda_i-1)}}^k n_in_j
=\sum_{\lambda\in S}d_\lambda d_{-1-\lambda},
\end{align*}
where $S=\{\lambda_1,\ldots,\lambda_k\}$ 
is the set of eigenvalues of $C$
and $d_\lambda=\sum_{i:\: \lambda_i=\lambda}n_i$
is the generalized eigenspace dimension for $\lambda$.
Now since $\lambda\neq-1-\lambda$ for all $\lambda\in\oF_2$,
we can choose a subset $S'\subset S$ such that
$S\subset S'\cup(-1-S')$ and $S'\cap(-1-S')=\emptyset$,
and we then get
\begin{align*}
\sum_{\lambda\in S}d_\lambda d_{-1-\lambda}
=2\sum_{\lambda\in S'}d_\lambda d_{-1-\lambda}
\leq \frac12\sum_{\lambda\in S'}(d_\lambda+d_{-1-\lambda})^2
\leq \frac12\biggl(\sum_{\lambda\in S'}(d_\lambda+d_{-1-\lambda})\biggr)^2
=\frac12 n^2.
\end{align*}
\end{proof}

\begin{lemma}\label{NLIFTSboundLEM}
Let $q$ be a prime power.
Then for any $A,B\in\M_n(\Z/q\Z)$ with $\gcd(q,A,B)=1$,
\begin{align*}
\#\scrC_{q}(A,B)\ll q^{(n-1)^2+1},
\end{align*}
where the implied constant is absolute.
\end{lemma}
\begin{proof}
Let us write $q=p^\alpha$ with $p$ a prime and $\alpha\geq1$.
Let $A,B\in\M_n(\Z/q\Z)$ and assume $\gcd(q,A,B)=1$,
viz., either $A\not\equiv0$ or $B\not\equiv0\mod p$. 
Note that $\#\scrC_q(A,B)=\#\scrC_q(B,A)$,
since for any $Y\in\GL_n(\Z/q\Z)$, %
$AY\equiv Y^{-1}B\mod q$ holds if and only if 
$(Y^{-1})^{-1}A\equiv BY^{-1}\mod q$.
Hence without loss of generality we may assume that 
$A\not\equiv0\mod p$.
In view of \eqref{CpABboundPROPres} in \autoref{CpABboundPROP},
it suffices to prove that for every $Y\in\scrC_p(A,B)$
there exist at most $p^{(\alpha-1)((n-1)^2+1)}$
lifts of $Y$ to $\scrC_{p^\alpha}(A,B)$,
that is, at most $p^{(\alpha-1)((n-1)^2+1)}$
matrices $Y'\in\scrC_{p^\alpha}(A,B)$
satisfying $[Y'\mod p]=Y$.
By induction over $\alpha$, it suffices to prove that 
for any $\alpha\geq1$, any $A,B\in\M_n(\Z/p^{\alpha+1}\Z)$
with $A\not\equiv 0\mod p$,
and any $Y\in\scrC_{p^\alpha}(A,B)$,
there exist at most $p^{(n-1)^2+1}$ matrices
$Y'\in\scrC_{p^{\alpha+1}}(A,B)$
satisfying $[Y'\mod p^\alpha]=Y$.
This holds trivially if there is \textit{no} such matrix $Y'$;
hence we may assume that there exists a matrix
$Y'_0\in\scrC_{p^{\alpha+1}}(A,B)$
with $[Y'_0\mod p^\alpha]=Y$.
Now the set of matrices 
$Y'\in\M_n(\Z/p^{\alpha+1}\Z)$
with $Y'\equiv Y'_0\mod p^{\alpha}$ can be parametrized
as $Y'=Y'_0(I+p^\alpha Z)$ with $Z$ running through
$\M_n(\F_p)$,
and we then compute that
\begin{align*}
AY'-{Y'}^{-1}B
\equiv p^\alpha (AY'_0Z+Z{Y'_0}^{-1}B)
\equiv p^\alpha (AY'_0Z+ZAY'_0)\mod p^{\alpha+1}.
\end{align*}
Hence $Y'$ lies in $\scrC_{p^{\alpha+1}}(A,B)$
if and only if $Z\in\scrA(AY_0')$.
Therefore, the number of admissible lifts $Y'$ equals
$\#\scrA(AY_0')=p^{d(AY_0')}$.
Note that $AY_0'\not\equiv0\mod p$;
hence if $p>2$ then 
by \autoref{dCBOUNDlem} we have
$d(AY_0')\leq(n-1)^2+1$,
and the proof is complete.

From now on we assume $p=2$.
In this case we decompose $\scrC_{2^\alpha}(A,B)$
as the disjoint union of the two sets
\begin{align*}
\scrC_{2^\alpha}^{(0)}(A,B):=\{Y\in\scrC_{2^\alpha}(A,B)\col AY\not\equiv I\mod 2\}
\end{align*}
and
\begin{align*}
\scrC_{2^\alpha}^{(1)}(A,B):=\{Y\in\scrC_{2^\alpha}(A,B)\col AY\equiv I\mod 2\}.
\end{align*}
For $\scrC_{2^\alpha}^{(0)}(A,B)$ the argument in the previous paragraph applies
(since we get $AY_0'\not\equiv I\mod 2$ as required in \autoref{dCBOUNDlem}),
and we thus obtain
\begin{align*}
\#\scrC_{2^\alpha}^{(0)}(A,B)\ll 2^{\alpha((n-1)^2+1)}
=q^{(n-1)^2+1}.
\end{align*}
We next consider $\scrC_{2^\alpha}^{(1)}(A,B)$.
Note that from now on we may assume that
both $A,B\in\GL_n(\Z/2^\alpha\Z)$ since otherwise
$\scrC_{2^\alpha}^{(1)}(A,B)=\emptyset$.
Substituting $X=AY$ we have
\begin{align*}
\#\scrC_{2^\alpha}^{(1)}(A,B)
=\#\{X\in\GL_n(\Z/2^\alpha\Z)\col
X^2\equiv AB\mod 2^\alpha\text{ and }X\equiv I\mod 2\}.
\end{align*}
Substituting next $X=I+2U$ with $U\in {\rm M}_n(\mathbb{Z}/2^{\alpha-1}\mathbb{Z})$, we see that $\#\scrC_{2}^{(1)}(A,B)\leq1$
(with equality if and only if $AB\equiv I\mod 2$),
while for $\alpha\geq2$ we obtain 
$\scrC_{2^\alpha}^{(1)}(A,B)=\emptyset$ if $AB\not\equiv I\mod 4$,
while in the case $AB\equiv I\mod 4$ we get,
after choosing
$B'\in\M_n(\Z/2^{\alpha-2}\Z)$ such that $AB\equiv I+4B'\mod 2^\alpha$:
\begin{align}\notag
\#\scrC_{2^\alpha}^{(1)}(A,B)
&=\#\{U\in\M_n(\Z/2^{\alpha-1}\Z)\col U+U^2\equiv B'\mod 2^{\alpha-2}\}
\\\label{NLIFTSboundLEMpf2}
&=2^{n^2}\#\{U\in\M_n(\Z/2^{\alpha-2}\Z)\col U+U^2\equiv B'\mod 2^{\alpha-2}\}.
\end{align}
In particular if $\alpha=2$ then 
$\#\scrC_{2^\alpha}^{(1)}(A,B)=2^{n^2}\leq 2^{2((n-1)^2+1)}$,
as desired.
To handle the case $\alpha\geq3$ we will prove that
for any $\beta\geq1$ and any 
$B'\in\M_n(\Z/2^{\beta}\Z)$,
\begin{align}\label{NLIFTSboundLEMpf1}
\#\{U\in\M_n(\Z/2^{\beta}\Z)\col U+U^2\equiv B'\mod 2^{\beta}\}\leq 2^{\frac12(\beta+1)n^2}.
\end{align}
This bound is trivial for $\beta=1$,
and to prove it for $\beta\geq2$ it suffices, by 
the same inductive lifting argument as in the first paragraph,
to prove that 
for any $\beta\geq2$
and any $U,B'\in\M_n(\Z/2^\beta\Z)$
with $U+U^2\equiv B'\mod 2^\beta$,
the number of $Z\in\M_n(\Z/2\Z)$ satisfying
$(U+2^{\beta-1}Z)+(U+2^{\beta-1}Z)^2\equiv B'\mod 2^\beta$
is at most $2^{\frac12n^2}$.
But the last equation is seen to be equivalent to
\begin{align*}
Z+UZ+ZU\equiv 0\mod2,
\end{align*}
and hence the claim follows from \autoref{vardCBOUNDlem}.
Using \eqref{NLIFTSboundLEMpf2} and \eqref{NLIFTSboundLEMpf1},
it follows that for all $\alpha\geq3$,
\begin{align*}
\#\scrC_{2^\alpha}^{(1)}(A,B)\leq 2^{\frac12(\alpha+1)n^2}
\leq 8\cdot 2^{\alpha((n-1)^2+1)}
=8q^{(n-1)^2+1},
\end{align*}
where the last inequality holds since
$3+\alpha((n-1)^2+1)-\frac12(\alpha+1)n^2
=\frac{\alpha-1}2(n-\frac{2\alpha}{\alpha-1})^2+\frac{\alpha-3}{\alpha-1}\geq0$.
This completes the proof of the lemma.
\end{proof}

\begin{lemma}\label{GpCDboundLEM}
For any prime $p$ and any $C,D\in\M_n(\F_p)$,
\begin{align*}
|G_p(C,D)|\leq p^{\frac12(n^2+d(C))}.
\end{align*}
\end{lemma}
\begin{proof}
It follows from the definition,
\eqref{GpCDdef},
that
\begin{align*}
|G_p(C,D)|^2=\sum_{Z,Y\in\M_n(\F_p)}e_p\bigl({\rm tr}(C(Z^2-Y^2)+D(Z-Y))\bigr).
\end{align*}
Substituting $Z=X+Y$, this becomes
\begin{align*}
\sum_{X\in\M_n(\F_p)}
\sum_{Y\in\M_n(\F_p)}
e_p\bigl({\rm tr}((XC+CX)Y)+{\rm tr}(CX^2+DX)\bigr),
\end{align*}
and here the inner sum vanishes unless $XC+CX=0$.
Hence
\begin{align*}
|G_p(C,D)|^2\leq \biggl|p^{n^2}\sum_{X\in\scrA(C)}e_p(CX^2+DX)\biggr|
\leq p^{n^2}\#\scrA(C)=p^{n^2+d(C)}.
\end{align*}
\end{proof}

\begin{proposition}\label{KboundgenqPROP}
Let $q$ be a prime power.
Then for any $A,B\in\M_n(\Z/q\Z)$ with $\gcd(q,A,B)=1$,
\begin{align}\label{KboundgenqPROPres}
|K_n(A,B;q)|
&\ll_n q^{n^2-n+1}.
\end{align}
\end{proposition}
\begin{proof}
Let us write $q=p^\beta$ with $p$ a prime and $\beta\geq1$.
If $\beta=1$ then \eqref{KboundgenqPROPres}
follows from \eqref{ETres1};
hence from now on we assume $\beta\geq2$.
It follows from $\gcd(q,A,B)=1$ that
$A\not\equiv 0$ or $B\not\equiv 0\mod p$.
If exactly one of $A$ or $B$ is divisible by $p$ 
then $K_n(A,B;q)=0$ by \autoref{KvanishingLEM2};
hence from now on we may assume that both $A,B\not\equiv 0\mod p$.
We will now use the bound in \autoref{thm:Taylor_csq}.
Thus set $\alpha=\lfloor\beta/2\rfloor\geq1$.
By \autoref{NLIFTSboundLEM},
$\#\scrC_{p^\alpha}(A,B)\ll p^{\alpha((n-1)^2+1)}$,
and so
\begin{align*}
p^{\alpha n^2}\#\scrC_{p^\alpha}(A,B)\ll p^{2\alpha(n^2-n+1)}.
\end{align*}
Furthermore,
if $p>2$, then by \autoref{dCBOUNDlem}
we have $d(C)\leq(n-1)^2+1$ for all $C\in\M_n(\F_p)\setminus\{0\}$,
and hence by
\autoref{GpCDboundLEM},
\begin{align*}
\max\bigl\{|G_p(C,D)|\col C,D\in\M_n(\F_p),\: C\neq 0\bigr\}\leq
p^{\frac12(n^2+(n-1)^2+1)}
=p^{n^2-n+1}.
\end{align*}
On the other hand for $p=2$ we have the trivial bound
\begin{align*}
|G_p(C,D)|\leq 2^{n^2}\ll_n 2^{n^2-n+1}.
\end{align*}
Using these bounds in \autoref{thm:Taylor_csq}, we get
\begin{align*}
|K_n(A,B;q)|\ll_n p^{\beta(n^2-n+1)}=q^{n^2-n+1}.
\end{align*}
\end{proof}

By combining \autoref{KboundgenqPROP}
with \autoref{multiiplicativityLEM} and \autoref{KvanishingLEM2},
we now obtain a bound valid for general moduli.
\begin{theorem}\label{KBOUNDmaintheorem}
Let $\ve>0$, $q\geq2$ and $A,B\in\M_n(\Z/q\Z)$.
If $\gcd(q,A,B)=1$ then
\begin{align}\label{KBOUNDmaintheoremRES1}
|K_n(A,B;q)|\ll_{n,\ve} q^{n^2-n+1+\ve}.
\end{align}
If $\ell=\gcd(q,A)$ then
\begin{align}\label{KBOUNDmaintheoremRES2}
\bigl|K_n(A,B;q)\bigr|\ll_{n,\ve} q^{n^2}(q/\ell)^{-n+1+\ve}.
\end{align}
\end{theorem}
(See also \cite[Theorem 1.8]{ETZ} for somewhat stronger and more precise bounds.)
\begin{proof}
If $\gcd(q,A,B)=1$ then it follows from
\autoref{multiiplicativityLEM}
and \autoref{KboundgenqPROP}
that
\begin{align*}
|K_n(A,B;q)|\ll_{n,\ve} q^{n^2-n+1+\ve}.
\end{align*}

Next we assume instead $\ell=\gcd(q,A)$.
Write $q=\prod_{i=1}^s p_i^{\alpha_i}$
and $\ell=\prod_{i=1}^s p_i^{\gamma_i}$
(thus $0\leq\gamma_i\leq\alpha_i$ and $0<\alpha_i$ for all $i$).
Picking $c_i\in(\Z/p^{\alpha_i}\Z)^\times$ as in 
\autoref{multiiplicativityLEM}
we have
\begin{align*}
K_n(A,B;q)=\prod_{i=1}^s K_n(c_iA,c_iB;p^{\alpha_i}).
\end{align*}
For each $i$,
if $\gamma_i<\alpha_i$ then
$p^{\alpha_i}\nmid A$,
and so by \autoref{KvanishingLEM2} 
we have
\begin{align*}
K_n(c_iA,c_iB;p^{\alpha_i})=p^{\gamma_i n^2}
K_n(p^{-\gamma_i}c_iA,p^{-\gamma_i}c_iB;p^{\alpha_i-\gamma_i})
\end{align*}
if $p^{\gamma_i}\mid B$,
and otherwise
$K_n(c_iA,c_iB;p^{\alpha_i})=0$.
Hence if $\gamma_i<\alpha_i$ then by
\autoref{KboundgenqPROP},
\begin{align*}
|K_n(c_iA,c_iB;p^{\alpha_i})|\leq C p^{\gamma_i n^2} p^{(\alpha_i-\gamma_i)(n^2-n+1)}
=C p^{\alpha_i n^2}p^{(\alpha_i-\gamma_i)(-n+1)},
\end{align*}
where $C=C(n)\geq1$ is the implied constant in 
\eqref{KboundgenqPROPres}.
In the remaining case, when $\gamma_i=\alpha_i$,
we use the \textit{trivial} bound
$|K_n(c_iA,c_iB;p^{\alpha_i})|\leq p^{\alpha_i n^2}$.
Multiplying over all $i$, we obtain:
\begin{align*}
|K_n(A,B;q)|\leq q^{n^2}\prod_{\substack{i=1\\(\gamma_i<\alpha_i)}}^s \Bigl(Cp^{(\alpha_i-\gamma_i)(-n+1)}\Bigr)
\ll_{n,\ve} q^{n^2}(q/\ell)^{-n+1+\ve}.
\end{align*}
Hence we have proved \eqref{KBOUNDmaintheoremRES2}.
\end{proof}

Finally we deal with the Ramanujan sum case.
\begin{proposition}\label{prop:mtx_ramanujan}
Let $p$ be a prime and $m\in \mathbb{Z}^+$. 
For $A\in {\rm M}_{n}(\mathbb{Z}/p^m\Z)$,
when $p^{m-1}\nmid A$, 
\begin{equation}\label{prop:mtx_ramanujanRES1}
K_n(\vecnull, A; p^m) = 0.
\end{equation}
Assume that $p^{m-1}\mid A$ and let $r\in \{0, 1, \ldots, n\}$ be the rank of the matrix $p^{-(m-1)}A$ in ${\rm M}_{n}(\mathbb{F}_p)$. 
Then 
\begin{equation}\label{prop:mtx_ramanujanRES2}
K_n(\vecnull, A; p^m) = p^{(m-1)n^2} (-1)^r
p^{-\frac{r(r+1)}{2}+ rn} 
\prod_{i=0}^{n-r-1} (p^{n-r}-p^{i}). 
\end{equation}
\end{proposition}
\begin{proof}
Applying \autoref{KvanishingLEM2}
with $\ell=p^{m-1}$,
the first claim,
\eqref{prop:mtx_ramanujanRES1},
follows immediately,
and the second claim,
\eqref{prop:mtx_ramanujanRES2},
is reduced to the case $m=1$.
Now
\eqref{prop:mtx_ramanujanRES2} follows from
\cite[Thm.\ 1.9]{mEaT2021},
since $\bigl|\GL_{n-r}(\F_p)\bigr|=\prod_{i=0}^{n-r-1}(p^{n-r}-p^i)$.
\end{proof}

\begin{corollary}\label{cor:Kn0A_boundCOR}
Let $q$ be a positive integer. 
We have 
\begin{equation}\label{e:Kn0A_bound}
|K_n(\vecnull, A; q)|
\leq \begin{cases} 
q^{n^2} \left(\frac{q}{\gcd(q, A)}\right)^{-n} & \text{ when } \prod_{p\mid q}p^{{\rm ord}_p(q)-1}\mid A, \\
0 & \text{ otherwise. } 
\end{cases} 
\end{equation}
\end{corollary}
\begin{proof}
Let us write $m_p=\ord_p(q)$.
If $\prod_{p\mid q}p^{m_p-1}\nmid A$ then
$K_n(\vecnull, A; q)=0$, by the 
first part of \autoref{prop:mtx_ramanujan}.
From now on we assume $\prod_{p\mid q}p^{m_p-1}\mid A$.
Letting $r_p$ be the rank of 
the matrix $p^{-(m_p-1)} A$ in ${\rm M}_{n}(\mathbb{F}_p)$,
we have by the second part of 
\autoref{prop:mtx_ramanujan}:
\begin{align}
|K_n(\vecnull, A; q)|
& \leq \prod_{p\mid q}
p^{(m_p-1)n^2} p^{-\frac{r_p(r_p+1)}{2}+r_pn} p^{(n-r_p)^2}
\\ & = \prod_{p\mid q}
p^{m_pn^2+\frac{r_p(r_p-1)}{2}-r_pn}
=q^{n^2} \prod_{p\mid q}p^{r_p\left(\frac{r_p-1}{2}-n\right)}.
\end{align}
One verifies that 
$r_p\bigl(\frac{r_p-1}{2}-n\bigr)\leq -n$ whenever $r_p\geq1$;
furthermore the product of all primes $p\mid q$ satisfying $r_p\geq1$
equals $\frac q{\gcd(q,A)}$;
hence the inequality in \eqref{e:Kn0A_bound}
follows.
\end{proof}

\section{Geometry of numbers} \label{sec:GoN}

Let us fix integers $1\leq n<d$.
For any real numbers 
$\kappa>nd$ and $a,b>0$, and any $g\in {\rm SL}_d(\mathbb{R})$,
we define
\begin{equation}\label{e:Phiell_def}
\Phi^{(\kappa)}_{a,b}(g) = \sum_{\substack{X\in {\rm M}_{n\times d}(\mathbb{Z})\\ X\neq\vecnull}} 
\frac 1{a+b\|X g\|_\infty^{\kappa}},
\end{equation}
Note that the condition $\kappa>nd$ ensures that the series on the right-hand side
converges (see also \autoref{LATTICECOUNTBOUNDlem2} below).
Furthermore, $\Phi^{(\kappa)}_{a,b}$ is (left) $\SL_d(\Z)$-invariant;
indeed, for any $\gamma\in {\rm SL}_d(\mathbb{Z})$
we have
\begin{equation}
\Phi^{(\kappa)}_{a,b}(\gamma g) = \sum_{\substack{X\in {\rm M}_{n\times d}(\mathbb{Z})\\ X\neq\vecnull}} 
\frac1{a+b \|X\gamma g\|_\infty^\kappa}
= \Phi^{(\kappa)}_{a,b}(g).
\end{equation}

Our goal in the present section is to prove a bound on 
the integral
of $\Phi^{(\kappa)}_{a,b}$
over $\SL_d(\Z)\bs\SL_d(\R)$;
see \autoref{L1boundprop} below.
This bound will play an important role in our proof of the main theorem
in Section \ref{MAINTHMPFsec}.
We start by proving, in \autoref{LATTICECOUNTBOUNDlem2} below,
a pointwise bound on $\Phi^{(\kappa)}_{a,b}$.

For a lattice $L$ in $\R^d$ 
we write $\lambda_i=\lambda_i(L)$ ($i=1,\ldots,d$)
for its successive minima with respect to the unit ball,
i.e.,
\begin{align}\label{e:lambdai_def}
\lambda_i(L):=\min\bigl\{\lambda\in\R_{\geq0}\col L\text{ contains $i$ linearly independent vectors of length $\leq\lambda$}\bigr\}.
\end{align}
Thus $0<\lambda_1\leq\lambda_2\leq\cdots\leq\lambda_d$.
Let $\mathcal{B}_R^d \subset \mathbb{R}^d$ be the ball of radius $R$ with centre at the origin in $\mathbb{R}^d$.
\begin{lemma}\label{LATTICECOUNTBOUNDlem1}
For every lattice $L$ in $\R^d$ and every $R>0$,
\begin{align*}
\#(L\cap\scrB_R^d)\asymp_d\prod_{i=1}^d\Bigl(1+\frac R{\lambda_i(L)}\Bigr).
\end{align*}
\end{lemma}
\begin{proof}
See \cite[Proposition 6]{hGcS91}.\footnote{As noted in the erratum of
\cite{hGcS91}, in the statement of \cite[Proposition 6]{hGcS91},
``$\lambda_1\cdots\lambda_k/M(K)$'' should read
``$\lambda_1\cdots\lambda_k M(K)$'', and in the last line of the proof, ``$\sim V(K_0)$'' should be corrected to ``$\sim V(K_0)^{-1}$''.}
\end{proof}

\begin{lemma}\label{LATTICECOUNTBOUNDlem2}
For any $\kappa>nd$, $a,b>0$ and $g\in\SL_d(\R)$,
writing $\lambda_i:=\lambda_i(\Z^dg)$ for $i=1,\ldots,d$, 
we have
\begin{align}\label{LATTICECOUNTBOUNDlem2res}
\Phi^{(\kappa)}_{a,b}(g)\ll\begin{cases}
b^{-1}\lambda_1^{-\kappa}
&\text{if }\:\lambda_1\geq (a/b)^{1/\kappa}
\\[5pt]
{\displaystyle a^{-1}\prod_{i=1}^d\Bigl(1+\frac{(a/b)^{1/\kappa}}{\lambda_i}\Bigr)^n}
&\text{if }\:\lambda_1\leq (a/b)^{1/\kappa},
\end{cases}
\end{align}
where the implied constant depends only on $n,d$ and $\kappa$.
\end{lemma}

\begin{proof}
As we will see,
the lemma follows from the definition \eqref{e:Phiell_def}
and \autoref{LATTICECOUNTBOUNDlem1}
by a simple computation using dyadic decomposition.
Note that for any $r>0$,
both sides in \eqref{LATTICECOUNTBOUNDlem2res}
are scaled by a factor $r^{-1}$ when replacing $\langle a,b\rangle$ by $\langle ra,rb\rangle$;
hence we may without loss of generality assume $b=1$.
Now set $c:=\lambda_1/\sqrt d$ and 
\begin{align*}
N_m:=\#\bigl\{X\in \M_{n\times d}(\Z)\col c\,2^{m-1}\leq \|Xg\|_{\infty}<c\,2^m\bigr\}
\qquad(m\in\Z).
\end{align*}
Then $N_m=0$ for all $m\leq0$, since every non-zero vector $\vecv\in\Z^dg$
satisfies $\|\vecv\|_\infty\geq \|\vecv\|/\sqrt d\geq c$, for $\|\vecv\|=\sqrt{\trans \vecv \vecv}$.  
Hence
\begin{align}\notag
\Phi^{(\kappa)}_{a,1}(g)
= \sum_{\substack{X\in {\rm M}_{d\times n}(\mathbb{Z})\\ X\neq\vecnull}} 
\frac1{a+ \|Xg\|_\infty^{\kappa}}
\ll_{\kappa}
\sum_{m=1}^\infty \frac{N_m}{a+(c\,2^m)^{\kappa}}
\ll_{{\kappa},d}
\sum_{m=1}^\infty \frac{N_m}{a+(\lambda_1\,2^m)^{\kappa}}.
\end{align}
Here 
\begin{align}\label{tNmdef}
N_m 
\leq \left(\#(\mathbb{Z}^d g\cap \mathcal{B}_{c 2^m}^d)\right)^n 
\ll_d \tN_m:=\prod_{i=1}^d\Bigl(1+\frac{\lambda_1 2^m}{\lambda_i}\Bigr)^n,
\end{align}
by \autoref{LATTICECOUNTBOUNDlem1} (and since $c\ll_d\lambda_1$),
and also
\begin{align}
\frac1{a+(\lambda_1\,2^m)^{\kappa}}< A_m:=\min\bigl(a^{-1},(\lambda_1\,2^m)^{-{\kappa}}\bigr),
\end{align}
and so
\begin{align}
\Phi^{(\kappa)}_{a,1}(g)\ll_{{\kappa},d}\sum_{m=1}^\infty \tN_m A_m.
\end{align}
Here we note that the sequence $\tN_1,\tN_2,\ldots$ is increasing
and satisfies $(\frac53)^n\tN_m\leq\tN_{m+1}\leq 2^{nd}\tN_m$ for all $m\geq1$
(where the first inequality comes from behavior of the factor corresponding to $i=1$ in \eqref{tNmdef}).
Letting $m_0$ be the unique real number satisfying 
$\lambda_12^{m_0}=a^{1/{\kappa}}$,
it follows that the sequence 
$\tN_mA_m$ is geometrically increasing 
with a ratio $\geq(\frac53)^n$ for $m\leq m_0$
and geometrically decreasing with the ratio $2^{nd-{\kappa}}$ for $m\geq m_0$.
Hence if $m_0\leq0$
(viz., if $\lambda_1\geq a^{1/{\kappa}}$)
then $\Phi^{(\kappa)}_{a,1}(g)\ll_{{\kappa},d,n}\tN_1A_1\ll_{d,n}\lambda_1^{-{\kappa}}$,
while if $m_0\geq0$ then
\begin{align}\notag
\Phi^{(\kappa)}_{a,1}(g)\ll_{{\kappa},d,n}\tN_{\lfloor m_0\rfloor}A_{\lfloor m_0\rfloor}+
\tN_{\lceil m_0\rceil}A_{\lceil m_0\rceil}
\leq a^{-1}\bigl(\tN_{\lfloor m_0\rfloor}+\tN_{\lceil m_0\rceil}\bigr)
\ll_{n,d} a^{-1}\tN_{m_0}
\hspace{60pt}
\\
=a^{-1}\prod_{i=1}^d\Bigl(1+\frac{a^{1/{\kappa}}}{\lambda_i}\Bigr)^n.
\end{align}\notag
\end{proof}

We will also make use of Rogers' formula,
\cite[Theorem 4]{cR55},
which can be stated as follows
(see \cite[Theorem 1.5 and Sec.\ 2]{LLdual}).
Recall from Section \ref{ss:Heckepts} that $\mu_0$ denotes the invariant probability measure on $\SL_d(\Z)\bs\SL_d(\R)$.
\begin{theorem}\label{lem:Rogers}
For any $1\leq n<d$, and for any Borel measurable function
$\rho:\M_{n\times d}(\R)\to\R_{\geq0}$ we have
\begin{align}\label{L1boundproppf1}
\int_{{\rm SL}_d(\mathbb{Z})\backslash {\rm SL}_d(\mathbb{R})}
\sum_{\substack{X\in {\rm M}_{n\times d}(\mathbb{Z})\\ X\neq\vecnull}} \rho(Xg)\,d\mu_0(g)
=\sum_{m=1}^n\sum_{B\in A_{n,m}}\int_{\M_{m\times d}(\R)}\rho(BX)\,dX,
\end{align}
where for each $m\in\{1,\ldots,n\}$,
$A_{n,m}$ is a subset of $\M_{n\times m}(\Z)$
such that the map $B\mapsto B\,\R^m$
is a bijection from $A_{n,m}$ onto the family of 
rational $m$-dimensional subspaces of $\R^n$, \footnote{Recall that a linear subspace $V\subset\R^n$
is said to be rational if $V=\Span_{\R}(V\cap\Z^n)$.}
and $(B\,\R^m)\cap\Z^n=B\,\Z^m$ for each $B\in A_{n,m}$;
furthermore, $dX$ is the standard $md$-dimensional Lebesgue measure on $\M_{m\times d}(\R)$.
\end{theorem}
In the above theorem, 
note that \eqref{L1boundproppf1} should be understood as
an identity between extended real numbers,
i.e.\ either both sides of the equality sign are finite and equal, or else both sides are $+\infty$.

\vspace{5pt}

Finally we are now ready to prove the main result of the present section.
\begin{proposition}\label{L1boundprop}
For any $\kappa>nd$ and $a,b>0$ we have
\begin{align}\label{L1boundpropres}
\int_{{\rm SL}_d(\mathbb{Z})\backslash {\rm SL}_d(\mathbb{R})} \Phi^{(\kappa)}_{a,b}(g) \, d\mu_0(g)\ll 
a^{-1}\Bigl(\frac ab\Bigr)^{\!\frac d{\kappa}}\Bigl(1+\frac ab\Bigr)^{(n-1)\frac d{\kappa}},
\end{align}
where the implied constant depends only on $n,d$ and $\kappa$.
\end{proposition}
In particular the proposition implies that the integral on the left-hand side of \eqref{L1boundpropres} is finite.

\begin{proof}
We apply \autoref{lem:Rogers} with the following choice of $\rho=\rho_{a,b}:\M_{n\times d}(\R)\to\R_{\geq0}$:
\begin{align*}
\rho_{a,b}(Y):=\frac1{a+b\|Y\|_\infty^{\kappa}}.
\end{align*}
With this choice, 
\eqref{L1boundproppf1} says that
\begin{align}
\int_{{\rm SL}_d(\mathbb{Z})\backslash {\rm SL}_d(\mathbb{R})} \Phi^{(\kappa)}_{a,b}(g) \, d\mu_0(g)
=\sum_{m=1}^n J_{a,b}(m),
\end{align}
where
\begin{align*}
J_{a,b}(m):=\sum_{B\in A_{n,m}}\int_{\M_{m\times d}(\R)}\rho_{a,b}(BX)\,dX.
\end{align*}
Note that for any $c>0$, by substituting $X=c^{1/{\kappa}}X_{\new}$ we have
\begin{align}
J_{a,b}(m)=c^{md/{\kappa}} J_{a,bc}(m).
\end{align}
Furthermore, we have the trivial scaling property
\begin{align}
J_{a,b}(m)=cJ_{ac,bc}(m).
\end{align}
Combining these we get:
\begin{align}\label{L1boundproppf2}
J_{a,b}(m)=\Bigl(\frac ab\Bigr)^{md/{\kappa}}J_{a,a}(m)
=a^{-1}\Bigl(\frac ab\Bigr)^{md/{\kappa}}J_{1,1}(m).
\end{align}

Let us now also note that,
by \autoref{LATTICECOUNTBOUNDlem1} and \autoref{LATTICECOUNTBOUNDlem2},
\begin{align*}
\Phi^{(\kappa)}_{1,1}(g)\ll_{n,d,{\kappa}}\#(\Z^dg\cap\scrB_1^d)^n,\qquad\forall g\in\SL_d(\R).
\end{align*}
Furthermore, by Schmidt \cite[Theorem 2]{wS58}
we have
\begin{align*}
\int_{{\rm SL}_d(\mathbb{Z})\backslash {\rm SL}_d(\mathbb{R})} \#(\Z^dg\cap\scrB_1^d)^n\, d\mu_0(g)<\infty.
\end{align*}
Hence 
$\int_{\SL_d(\Z)\bs\SL_d(\R)}\Phi^{(\kappa)}_{1,1}(g)\,d\mu_0(g)<\infty$,
and thus $J_{1,1}(m)<\infty$ for each $m\in\{1,\ldots,n\}$.
Hence %
\begin{align*}
\int_{{\rm SL}_d(\mathbb{Z})\backslash {\rm SL}_d(\mathbb{R})} \Phi^{(\kappa)}_{a,b}(g) \, d\mu_0(g)
=\sum_{m=1}^n J_{a,b}(m)
=\sum_{m=1}^n a^{-1}\Bigl(\frac ab\Bigr)^{md/{\kappa}}J_{1,1}(m)
&\ll_{n,d,{\kappa}}\sum_{m=1}^n a^{-1}\Bigl(\frac ab\Bigr)^{m\frac d{\kappa}}
\\
&\ll_{n} a^{-1}\max\Bigl(\Bigl(\frac ab\Bigr)^{\frac d{\kappa}},\Bigl(\frac ab\Bigr)^{n\frac d{\kappa}}\Bigr),
\end{align*}
viz., the bound in \eqref{L1boundpropres} holds.
\end{proof}

\section{Proof of the main theorem}
\label{MAINTHMPFsec}

In this section we give the proof of \autoref{MAINTHM}.
The proof is split into the two cases $n=d$ and $n<d$.
The first of these is treated in Section \ref{pfcaseneqd}:
as we will see, the proof in this case is a fairly easy
consequence of the bounds on the matrix Kloosterman sums proved in
Section \ref{s:MATRIXKLOOSTERMANsec}.
The proof in the case $n<d$ is carried out in Sections \ref{casenledSEC}--\ref{E2qfSEC};
the proof depends crucially on the bounds in
Section \ref{s:MATRIXKLOOSTERMANsec} in this case as well, but we additionally need to invoke
Hecke equidistribution and methods from geometry of numbers.

We stress that throughout the present section,
the implied constant in any ``$\ll$'' 
may depend on $d$ (thus may also depend on $n$),
without this being explicitly indicated in the notation.

\subsection{The case $n=d$}
\label{pfcaseneqd}
In this case we have $\kappa=2n^2$ and $k=2n^2+1$
in the statement of \autoref{MAINTHM}.
Furthermore, $\scrR_q=\GL_n(\Z/q\Z)$ and $\scrB_q=\{I_n\}$
(see Section \ref{RqSEC});
hence by \autoref{lem:mtx_relation}, 
for any $f\in {\rm C}_b^k({\rm M}_{n}(\mathbb{R}/\Z)\times \Gamma \backslash \Gamma {\rm H})$
and $q\in\Z^+$, we have
\begin{multline}\label{caseneqdpf1}
\mathcal{A}_q(f) = \frac{1}{\#\mathcal{R}_q} \sum_{R\in \mathcal{R}_q} f\left(q^{-1}R, \tn_+(q^{-1}R) D(q)\right)
\\ = \frac{1}{\#{\rm GL}_{n}(\mathbb{Z}/q\mathbb{Z})} 
\sum_{R\in \GL_n(\Z/q\Z)}
f\left(q^{-1} R, \tn_-(q^{-1}R^{-1}) \right). 
\end{multline}
Note also that ${\rm H}\cong {\rm M}_{n}(\mathbb{R})$
and $\Gamma\backslash \Gamma {\rm H}\cong {\rm M}_{n}(\mathbb{R}/\Z)$
in the present case;
hence we have the following Fourier expansion, for $X_1, X_2\in {\rm M}_{n}(\mathbb{R}/\Z)$:
\begin{equation}\label{caseneqdpf2}
f(X_1, \tn_-(X_2)) = \sum_{N,M\in {\rm M}_{n}(\mathbb{Z})} \widehat{f}(N, M) e^{2\pi i {\rm tr}(\trans N X_1)}e^{2\pi i {\rm tr}(\trans M X_2)}, 
\end{equation}
where 
\begin{equation}
\widehat{f}(N, M) = \int_{{\rm M}_{n}(\mathbb{R}/\Z)} \int_{{\rm M}_{n}(\mathbb{R}/\Z)} f(T_1, \tn_-(T_2)) e^{-2\pi i {\rm tr}(\trans N T_1)} e^{-2\pi i {\rm tr}(\trans M T_2)} \, dT_1 \, dT_2. 
\end{equation}
The sum in \eqref{caseneqdpf2} is absolutely convergent, uniformly with respect to $X_1,X_2$,
since $f\in{\rm C}_b^k$ with $k=2n^2+1>n^2$
\cite[Theorem 3.2.16]{Grafakos}.

By applying integration by parts
in a similar way as in \autoref{lem:Fouriercoeff_bound}, 
we have, for any $0\leq\lambda\leq k$ and 
$N,M\in\M_{n\times n}(\Z)$,
\begin{equation}\label{e:Fourier_bound_n=d}
\bigl|\widehat{f}(N, M)\bigr| 
\ll\min\biggl(\frac{S_{\infty,\lambda}(f)}{1+\|N\|_\infty^\lambda},
\frac{S_{\infty,\lambda}(f)}{1+\|M\|_\infty^\lambda}\biggr)
\ll \frac{S_{\infty,\lambda}(f)}{1+\|N\|_\infty^\lambda+\|M\|_\infty^\lambda}.
\end{equation}

Substituting \eqref{caseneqdpf2} into \eqref{caseneqdpf1},
and then using the definition of the 
matrix Kloosterman sum $K_n(A,B;q)$ in \eqref{Kndefrep}
and the basic identity $K_n(\trans A,\trans B;q)=K_n(A,B;q)$,
we obtain
\begin{multline}\label{caseneqdpf10}
\mathcal{A}_q(f) 
= \widehat{f}(\vecnull, \vecnull) 
+ \frac{1}{\#{\rm GL}_{n}(\mathbb{Z}/q\mathbb{Z})} 
\sum_{\substack{N, M\in {\rm M}_{n}(\mathbb{Z})\\ N\neq \vecnull \text{ or } M\neq \vecnull}} 
\widehat{f}(N, M)
\sum_{R\in {\rm GL}_n(\mathbb{Z}/q\mathbb{Z})}
e^{2\pi i \frac{{\rm tr}(\trans N R)+{\rm tr}(\trans M R^{-1})}{q}}
\\ = 
\int_{{\rm M}_{n}(\mathbb{R}/\Z)} \int_{{\rm M}_{n}(\mathbb{R}/\Z)}
f\left(T_1, \tn_-(T_2)\right) \, dT_1 \, dT_2
+E(q),
\end{multline}
where
\begin{align*}
E(q):=\frac{1}{\#{\rm GL}_{n}(\mathbb{Z}/q\mathbb{Z})} 
\sum_{\substack{N, M\in {\rm M}_{n}(\mathbb{Z})\\ N\neq \vecnull \text{ or } M\neq \vecnull}} 
\widehat{f}(N, M) K_n(N, M; q).
\end{align*}
Here, in order to bound the Kloosterman sum 
for $M=\bn$ we apply \eqref{e:Kn0A_bound} in \autoref{cor:Kn0A_boundCOR},
while for $M\neq\bn$
we use \eqref{KBOUNDmaintheoremRES2} in \autoref{KBOUNDmaintheorem}
if $n\geq2$, and the classical Weil bound if $n=1$.
Using also 
\begin{align}\label{GLnZmodqbound}
\#\GL_n(\Z/q\Z)=q^{n^2}\prod_{p\mid q}\prod_{j=1}^n(1-p^{-j})
> q^{n^2}\prod_{j=2}^n\zeta(j)^{-1}\prod_{p\mid q}(1-p^{-1})
\gg q^{n^2}(\log q)^{-1}
\end{align}
(where the first equality holds by \autoref{RqcardLEM},
and the last relation holds by Mertens' third theorem),
and \eqref{e:Fourier_bound_n=d},
we obtain:
\begin{multline}\label{caseneqdpf3}
\bigl|E(q)\bigr| \ll_{\ve} 
\frac{S_{\infty,\lambda}(f)}{q^{n^2-\ve}}
\Biggl(\sum_{\substack{N\in {\rm M}_{n}(\mathbb{Z})\\ N\neq\vecnull}} 
\frac{q^{n^2-n} \gcd(q, N)^{n}}
{\|N\|_\infty^\lambda}
+\sum_{\substack{N, M\in {\rm M}_{n}(\mathbb{Z})\\ M\neq \vecnull}} 
\frac{q^{n^2-\vartheta+\ve} \gcd(q, M)^{\vartheta}}
{1+\|N\|_\infty^\lambda+\|M\|_\infty^\lambda}\Biggr), 
\end{multline}
where $\vartheta$ is as in the statement of \autoref{MAINTHM},
viz., $\vartheta=n-1$ if $n>1$ and $\vartheta=\frac12$ if $n=1$.

In the first sum, we substitute $\ell=\gcd(q,N)$ and $N=\ell\, N_{\new}$;
this gives
\begin{align*}
\sum_{\substack{N\in {\rm M}_{n}(\mathbb{Z})\\ N\neq\vecnull}} 
\frac{\gcd(q, N)^{n}}{\|N\|_\infty^\lambda}
=\sum_{\ell\mid q}\ell^{n-\lambda}\sum_{\substack{N\in {\rm M}_{n}(\mathbb{Z})\\ N\neq\vecnull}}
\|N\|_\infty^{-\lambda}.
\end{align*}
Using here the fact that 
\begin{equation}
\left|\left\{N\in {\rm M}_{n}(\mathbb{Z}): \|N\|_\infty= m\right\}\right| 
\leq 2n^2(2m+1)^{n^2-1}
\qquad(\forall m\in\Z_{\geq0}),
\end{equation}
we obtain
\begin{align}\label{caseneqdpf5}
\sum_{\substack{N\in {\rm M}_{n}(\mathbb{Z})\\ N\neq\vecnull}} 
\frac{\gcd(q, N)^{n}}{\|N\|_\infty^\lambda}
\ll\sum_{\ell\mid q}\ell^{n-\lambda}\sum_{m=1}^\infty m^{n^2-1-\lambda}.
\end{align}
The last sum converges if and only if $\lambda>n^2$,
and when this holds the total expression is bounded above
by a constant which only depends on $\lambda$.

Similarly, regarding the second sum in \eqref{caseneqdpf3} we have:
\begin{align*}
\sum_{\substack{N, M\in {\rm M}_{n}(\mathbb{Z})\\ M\neq \vecnull}} 
&\frac{\gcd(q, M)^{\vartheta}}{1+\|N\|_\infty^\lambda+\|M\|_\infty^\lambda}
\ll\sum_{\ell\mid q}\ell^{\vartheta}
\sum_{\substack{N, M\in {\rm M}_{n}(\mathbb{Z})\\ M\neq \vecnull}}
\frac1{1+\|N\|_\infty^\lambda+\ell^\lambda\|M\|_\infty^\lambda}
\\
&\ll\sum_{\ell\mid q}\ell^{\vartheta}\sum_{u=0}^\infty\sum_{m=1}^\infty
\frac{(u+1)^{n^2-1}m^{n^2-1}}{1+u^\lambda+\ell^\lambda m^\lambda}
\\
&\ll\sum_{\ell\mid q}\ell^{\vartheta}
\biggl(\sum_{m=1}^\infty \ell^{-\lambda}m^{n^2-1-\lambda}
\sum_{u=0}^{\ell m}(u+1)^{n^2-1}
+\sum_{u=\ell+1}^\infty u^{n^2-1-\lambda}\sum_{1\leq m<u/\ell}m^{n^2-1}\biggr)
\\
&\ll\sum_{\ell\mid q}\ell^{\vartheta}
\biggl(\sum_{m=1}^\infty \ell^{-\lambda}m^{n^2-1-\lambda}\,(\ell m)^{n^2}
+\sum_{u=\ell+1}^\infty u^{n^2-1-\lambda}\,(u/\ell)^{n^2}\biggr).
\end{align*}
Both the sums in the last expression are convergent if and only if $\lambda>2n^2$,
and \textit{if} $\lambda>2n^2$ then we obtain
\begin{align}\label{caseneqdpf4}
\sum_{\substack{N, M\in {\rm M}_{n}(\mathbb{Z})\\ M\neq \vecnull}} 
\frac{\gcd(q, M)^{\vartheta}}{1+\|N\|_\infty^\lambda+\|M\|_\infty^\lambda}
\ll_{\lambda}\sum_{\ell\mid q}\ell^{\vartheta+n^2-\lambda}
\ll_{\ve}q^{\ve}.
\end{align}
(If $n\geq2$ then we even have
$\sum_{\ell\mid q}\ell^{\vartheta+n^2-\lambda}\ll1$.)

In conclusion,
using \eqref{caseneqdpf5} and \eqref{caseneqdpf4} in
\eqref{caseneqdpf3},
it follows that for any $\lambda>2n^2$,
\begin{align*}
\bigl|E(q)\bigr| \ll_{\lambda,\ve} S_{\infty,\lambda}(f)\, \bigl(q^{-n+\ve}+q^{-\vartheta+3\ve}\bigr)
\ll S_{\infty,\lambda}(f)\, q^{-\vartheta+3\ve}.
\end{align*}
Using this in \eqref{caseneqdpf10},
setting $\ve=\frac13\ve_{\new}$
and then choosing $\lambda=2n+\ve_{\new}$,
we obtain the relation
\eqref{MAINTHMres},
i.e.\ we have proved \autoref{MAINTHM} in the case $n=d$.
\hfill$\square$

\subsection{The case $n<d$}\label{casenledSEC}
To start the proof in this case,
let $\kappa,\vartheta,\kappa',\vartheta',k,\ve,f$ and $q$ be given as in the statement of 
\autoref{MAINTHM}.
By \autoref{lem:scrRq_parameter} and \autoref{lem:mtx_relation},
\begin{multline}
\mathcal{A}_q(f) 
= \frac{1}{\#\mathcal{R}_q} \sum_{\gamma\in \mathcal{B}_{q}} \sum_{U\in {\rm GL}_n(\mathbb{Z}/q\mathbb{Z})} f\left(q^{-1}\gamma^{-1} \bpm\vecnull\\ U\ebpm, \tn_-\left(q^{-1}\bigl(\bn\hspace{7pt} U^{-1}\bigr)\right) \matr{D_q\gamma}{}{}{I_n}\right)
\\ = \frac{1}{\#\mathcal{R}_q} \sum_{\gamma\in \mathcal{B}_{q}} \sum_{U\in {\rm GL}_n(\mathbb{Z}/q\mathbb{Z})}
\sum_{N\in {\rm M}_{d\times n}(\mathbb{Z})} 
\widehat{f}\left(N; \tn_-\left(q^{-1}\bigl(\bn\hspace{7pt} U^{-1}\bigr)\right) \matr{D_q\gamma}{}{}{I_n} \right) 
e^{2\pi i \frac{{\rm tr} \left(\trans N \gamma^{-1}\sm \vecnull\\ U\esm\right)}{q}}, 
\end{multline}
where, for $N\in {\rm M}_{d\times n}(\mathbb{Z})$ and $h\in {\rm H}$,
\begin{equation}\label{hfNgdef}
\widehat{f}(N; h)
= \int_{\M_{d \times n}(\R/\Z)} f(T, h) e^{-2\pi i\, {\rm tr}(\trans N T)}\, dT.
\end{equation}
Recall from the statement of \autoref{MAINTHM}
that $f\in {\rm C}_b^k(\M_{d \times n}(\R/\Z) \times \Gamma \backslash \Gamma {\rm H})$.
By applying integration by parts in a similar way as in \autoref{lem:Fouriercoeff_bound}, 
we have, for any $0\leq\lambda\leq k$,
$N\in\M_{d\times n}(\Z)$ and $h\in{\rm H}$,
\begin{align}\label{hatfbound}
\bigl|\widehat{f}(N; h) \bigr|\ll\frac{S_{\infty,\lambda}(f)}{1+\|N\|_\infty^\lambda}.
\end{align}
Recall our parametrization of ${\rm H}$ in \eqref{Hparametrization};
note that this can be expressed as
$h=n_-(X)\matr{g}{}{}{I_n}$. 
In line with this we set,
for
$g\in\SL_d(\R)$, $X\in\M_{n \times d}(\R/\Z)$
and $N\in {\rm M}_{d\times n}(\mathbb{Z})$:
\begin{equation}\label{FgXNdef}
F(g, X; N) = \widehat{f}\left(N, \tn_-(X)\matr g{}{}{I_n}\right). 
\end{equation}
By \eqref{e:Fourier_ASL}, 
\begin{equation}
F(g, X; N) = \sum_{M\in {\rm M}_{n\times d}(\mathbb{Z})} \widehat{F}(g; M, N) e^{2\pi i {\rm tr}(\trans M X)} 
\end{equation}
where 
\begin{multline}\label{hatFdef}
\widehat{F}(g; M, N) = \int_{\M_{n\times d}(\R/\Z)} F(g, T; N) e^{-2\pi i{\rm tr}(\trans M T)}\, dT
\\ = \int_{\M_{d \times n}(\R/\Z)} \int_{\M_{n\times d}(\R/\Z)}
f\left(T_1, \tn_-(T_2)\matr g{}{}{I_n} \right) 
e^{-2\pi i {\rm tr}(\trans N T_1)} e^{-2\pi i{\rm tr}(\trans M T_2)}\, dT_2\, dT_1. 
\end{multline}
Hence we have 
\begin{align*}
\mathcal{A}_q(f) 
= \frac{1}{\#\mathcal{R}_q} 
\sum_{\substack{N\in {\rm M}_{d\times n}(\mathbb{Z})\\ M\in\M_{n\times d}(\Z)}}
\sum_{\gamma\in \mathcal{B}_{q}}
\sum_{U\in {\rm GL}_n(\mathbb{Z}/q\mathbb{Z})}
\widehat{F}\left(D_q\gamma; M, N\right)
e^{2\pi i \frac{{\rm tr}\left(\trans M (\bn\: U^{-1})\right) + {\rm tr} \left(\trans N \gamma^{-1} \sm \vecnull \\ U\esm\right)}{q}}.
\end{align*}
Recalling now the definition of the matrix Kloosterman sum,
\eqref{Kndefrep},
and using
$\cmatr{\bn}U=\cmatr{\bn}{I_n} U$ and
\begin{align}
\tr\bigl(\trans M\bigl(\bn\: U^{-1}\bigr)\bigr)
=\tr\bigl(\trans M U^{-1}\bigl(\bn\: I_n\bigr)\bigr)
=\tr\bigl(\bigl(\bn\: I_n\bigr)\trans M U^{-1}\bigr),
\end{align}
we obtain
\begin{equation}
\mathcal{A}_q(f) 
=\frac{1}{\#\mathcal{R}_q} 
\sum_{\substack{N\in {\rm M}_{d\times n}(\mathbb{Z})\\ M\in\M_{n\times d}(\Z)}}
\sum_{\gamma\in \mathcal{B}_{q}}
\widehat{F}\bigl(D_q\gamma; M, N\bigr)
K_n\biggl(\bigl( \bn\: I_n\bigr) \trans M , \trans N\gamma^{-1}\cmatr{\bn}{I_n};q\biggr).
\end{equation}
We split this sum into three parts
by separating out the two cases
$N=M=\bn$
and [$N\neq\bn$, $M=\bn$]: 
\begin{align}\label{e:scrSqf_decomp}
\scrA_q(f)=E_{0,q}(f)+E_{1,q}(f)+E_{2,q}(f),
\end{align}
where
\begin{align}\label{e:E0q_def}
&E_{0,q}(f)=
\frac{\#{\rm GL}_n(\mathbb{Z}/q\mathbb{Z})}{\#\mathcal{R}_q} 
\sum_{\gamma\in \mathcal{B}_{q}}
\widehat{F}\left(D_q \gamma; \vecnull, \vecnull \right),
\\\label{e:E1q_def}
&E_{1,q}(f)=
\frac{1}{\#\mathcal{R}_q} 
\sum_{\substack{N\in {\rm M}_{d\times n}(\mathbb{Z})\\ N\neq\bn}}
\sum_{\gamma\in \mathcal{B}_{q}}
\widehat{F}\left(D_q \gamma; \vecnull, N\right)
K_n\left(\vecnull, \trans N\gamma^{-1}\cmatr{\bn}{I_n}; q\right)
\end{align}
and
\begin{align}\label{e:E2q_def}
E_{2, q}(f) 
=  \frac{1}{\#\mathcal{R}_q} 
\sum_{\substack{N\in {\rm M}_{d\times n}(\mathbb{Z})\\ M\in\M_{n\times d}(\Z)\setminus\{\bn\}}}
\sum_{\gamma\in \mathcal{B}_{q}}
\widehat{F}\left(D_q \gamma; M, N\right)
K_n\biggl(\bigl( \bn\: I_n\bigr) \trans M , \trans N\gamma^{-1}\cmatr{\bn}{I_n};q\biggr).
\end{align}

\subsection{The main term: $E_{0,q}(f)$}

We apply the equidistribution of Hecke points to the sum $E_{0,q}(f)$ in \eqref{e:E0q_def}.
Note that by \autoref{lem:Fouriercoeff_semiautomorphic},
$\widehat{F}(g; \vecnull, \vecnull)$ %
is a left $\SL_d(\Z)$-invariant function of $g\in\SL_d(\R)$.
Using \eqref{e:Heckeop_def}
and $\#\scrR_q=\#\scrB_q\cdot\#\GL_n(\Z/q\Z)$
(which holds by \autoref{lem:scrRq_parameter}),
we have 
\begin{equation}\label{E0qHeckeformula}
E_{0,q}(f)
= \frac{1}{\#\mathcal{B}_q} 
\sum_{\gamma\in \mathcal{B}_q} \widehat{F}\left(D_q \gamma; \vecnull, \vecnull\right)
= \big(T_{D_q} \widehat{F}(\,\cdot\:; \vecnull, \vecnull)\big)(I_d). 
\end{equation} 
Recall that we are keeping $1\leq n<d$,
and that $\kappa',\vartheta',k,\ve,f$ and $q$ are 
given as in the statement of \autoref{MAINTHM};
in particular we have
$f\in {\rm C}_b^k((\R/\Z)^{dn} \times \Gamma \backslash \Gamma {\rm H})$
where $k$ is an integer with $k>\kappa'+\ve$.
It follows that $\widehat{F}(\,\cdot\: ; \vecnull, \vecnull)\in {\rm C}_b^k({\rm SL}_d(\mathbb{Z})\backslash {\rm SL}_d(\mathbb{R}))$,
and now by \eqref{E0qHeckeformula} and \autoref{HeckeequidistrPROP} we have
\begin{align}\label{Heckeequidistrapplic}
\left|E_{0,q}(f)
- \int_{{\rm SL}_d(\mathbb{Z})\backslash {\rm SL}_d(\mathbb{R})} \widehat{F}(g; \vecnull, \vecnull)\, d\mu_0(g)\right|
\ll_{\ve} S_{2,\kappa'+\ve}(\widehat{F}(\,\cdot\:; \vecnull, \vecnull))\, q^{-\vartheta'+\ve}.
\end{align}
Here
\begin{multline}\label{Heckeequidistrapplicfact1}
\int_{{\rm SL}_d(\mathbb{Z})\backslash {\rm SL}_d(\mathbb{R})} \widehat{F}(g; \vecnull, \vecnull)\, d\mu_0(g)
\\ = \int_{{\rm SL}_d(\mathbb{Z})\backslash {\rm SL}_d(\mathbb{R})} 
\int_{\M_{d \times n}(\R/\Z)}
 \int_{\M_{n\times d}(\R/\Z)}
f\left(T_1, \tn_-(T_2)\matr g{}{}{I_n} \right) \, dT_2\ dT_1 \, d\mu_0(g)
\\ = \int_{\M_{d \times n}(\R/\Z)} 
\int_{\Gamma\backslash \Gamma {\rm H}} f(T, g) \, d\mu_{\rm H}(g)\, dT.
\end{multline}
Finally, in order to compare 
$S_{2,\kappa'+\ve}(\widehat{F}(\,\cdot\:; \vecnull, \vecnull))$
with $S_{2,\kappa'+\ve}(f)$,
let $\scrB$ and $\scrB'$ be the fixed linear bases for
the Lie algebra of $\SL_d(\R)$ 
and the Lie algebra of $\M_{d\times n}(\R)\times{\rm H}$ 
which are used in the definitions of the Sobolev norms; 
we may then assume that $\scrB\subset\scrB'$
when the Lie algebra of $\SL_d(\R)$ is embedded in the
Lie algebra of $\M_{d\times n}(\R)\times{\rm H}$ via the differential of the
homomorphism
$g\mapsto \left(\bn,\sm g & \bn \\ \bn & I_n\esm\right)$. 
Then for any monomial $\scrD$ in $\scrB$ of order $\leq k$, and every $g\in\SL_d(\R)$, we have
\begin{align*}
\bigl[\scrD \hF\bigr](g;\bn,\bn)
=\int_{\M_{d \times n}(\R/\Z)}
 \int_{\M_{n\times d}(\R/\Z)}
[\scrD f]\left(T_1, \tn_-(T_2)\matr g{}{}{I_n} \right) \, dT_2\, dT_1,
\end{align*}
and hence
\begin{align*}
\bigl|\bigl[\scrD \hF\bigr](g;\bn,\bn)\bigr|^2
\leq\int_{\M_{d \times n}(\R/\Z)}
 \int_{\M_{n\times d}(\R/\Z)}
\left|[\scrD f]\left(T_1, \tn_-(T_2)\matr g{}{}{I_n} \right)\right|^2 \, dT_2\, dT_1.
\end{align*}
Integrating the last inequality over $g$, it follows that
\begin{align*}
\bigl\|\bigl[\scrD \hF\bigr](\:\cdot\:;\bn,\bn)\bigr\|_{{\rm L}^2(\SL_d(\Z)\bs\SL_d(\R))}
=\sqrt{\int_{\SL_d(\Z)\bs\SL_d(\R)}\bigl|\bigl[\scrD \hF\bigr](g;\bn,\bn)\bigr|^2\,d\mu_0(g)}
\hspace{50pt}
\\
\leq
\|\scrD f\|_{{\rm L}^2((\R/\Z)^{dn}\times\Gamma\bs\Gamma{\rm H})}.
\end{align*}
For any integer $0\leq k_1\leq k$,
by summing the above inequality over all monomials in $\scrB$ of order $\leq k_1$,
it follows that
$S_{2,k_1}\bigl(\hF(\:\cdot\: ;\bn,\bn)\bigr)\leq S_{2,k_1}(f)$.
Hence also
$S_{2,\lambda}\bigl(\hF(\:\cdot\: ;\bn,\bn)\bigr)\leq S_{2,\lambda}(f)$
for any real number $0\leq\lambda\leq k$.
Using this fact together with \eqref{Heckeequidistrapplicfact1}
in \eqref{Heckeequidistrapplic}, we conclude:
\begin{equation}\label{e:separating_main_scrSqf}
\left|E_{0,q}(f)-\int_{\M_{d \times n}(\R/\Z)} 
\int_{\Gamma\backslash \Gamma {\rm H}} f(T, g) \,d\mu_{\rm H}(g) \,dT \right|
\ll_{\ve}
S_{2,\kappa'+\ve}(f)\, q^{-\vartheta'+\ve}.
\end{equation}

\subsection{Error term 1: $E_{1, q}(f)$} 

It follows from \eqref{hatFdef} and \eqref{hatfbound} that
\begin{align*}
\bigl|\hF(g;M,N)\bigr|\ll \frac{S_{\infty,\lambda}(f)}{1+\|N\|_\infty^\lambda}
\end{align*}
for all $0\leq\lambda\leq k$, $g\in\SL_d(\R)$ and $N\in\M_{d\times n}(\Z)$,
$M\in\M_{n\times d}(\Z)$.
Using 
this bound together with 
\autoref{cor:Kn0A_boundCOR}
in \eqref{e:E1q_def},
we obtain:
\begin{align}\notag
|E_{1, q}(f)|
 \ll S_{\infty,\lambda}(f)
\frac{q^{n^2}}{\#\mathcal{R}_q} 
\sum_{\substack{N\in {\rm M}_{d\times n}(\mathbb{Z})\\ N\neq \vecnull}}
\|N\|_{\infty}^{-\lambda}
\sum_{\gamma\in \mathcal{B}_{q}}
\left(\frac{q}{\gcd\bigl(q, \trans N \gamma^{-1}\scmatr{\bn}{I_n}\bigr)}\right)^{-n}
\\ \label{E1qtreatment1}
 \leq S_{\infty,\lambda}(f) 
\frac{q^{n^2-n}}{\#\mathcal{R}_q} 
\sum_{\substack{N\in {\rm M}_{d\times n}(\mathbb{Z})\\ N\neq \vecnull}}
\|N\|_{\infty}^{-\lambda}
\sum_{\ell\mid q} \ell^n
\scrA_\ell(N),
\end{align}
where
\begin{align*}
\scrA_\ell(N)=\#\bigl\{\gamma\in \mathcal{B}_q\col  \ell \mid \trans N \gamma^{-1}\scmatr{\bn}{I_n}\bigr\}.
\end{align*}

For any $N\in\M_{d\times n}(\Z)$, $\ell\mid q$,
$\gamma\in\mathcal{B}_q$ and $U\in\GL_n(\Z/q\Z)$,
by multiplying by $U$ from the right,
it follows that the relation 
$\ell \mid \trans N \gamma^{-1}\scmatr{\bn}{I_n}$
is equivalent with 
$\ell \mid \trans N \gamma^{-1}\scmatr{\bn}{U}$.
Hence by \autoref{lem:scrRq_parameter},
\begin{align}\notag
\scrA_\ell(N)
&\leq\frac{\#\bigl\{X\in\M_{d\times n}(\Z/q\Z)\col \trans N X\equiv\bn\mod\ell\bigr\}}{\#\GL_n(\Z/q\Z)}
\\\label{E1qtreatment3}
&=\frac{\#\bigl\{X\in\M_{d\times n}(\Z/q\Z)\col \trans N' X\equiv\bn\mod\ell'\bigr\}}{\#\GL_n(\Z/q\Z)},
\end{align}
where we write $N':=\gcd(\ell,N)^{-1}N$ and $\ell':=\gcd(\ell,N)^{-1}\ell$.
To bound the last expression,
note that 
$X\in\M_{d\times n}(\Z/q\Z)$ satisfies the relation $\trans N' X\equiv\bn\mod\ell'$
if and only if
$\trans N'X\equiv\bn\mod \gcd(p^r,\ell')$
holds for every prime power $p^r$ dividing $q$ (with %
$r\geq1$).
But by construction we have
$\gcd(\ell',N')=1$;
hence if $p\mid\ell'$ then
$\trans N'$ has at least one row, 
say $\vecn=\vecn(p)\in\Z^d$, which is not divisible by $p$,
which means that there are exactly 
$p^{rdn}/\gcd(p^r,\ell')^n$ matrices
$X\in\M_{d\times n}(\Z/p^r\Z)$ satisfying $\vecn X\equiv\bn\mod\gcd(p^r,\ell')$.
Hence
\begin{align*}
\#\bigl\{X\in\M_{d\times n}(\Z/p^r\Z)\col \trans N' X\equiv\bn\mod\gcd(p^r,\ell')\bigr\}
\leq\frac{p^{rdn}}{\gcd(p^r,\ell')^n}.
\end{align*}
Using this bound for each prime power $p^r$ dividing $q$,
and multiplying, it follows that
\begin{align}
\#\bigl\{X\in\M_{d\times n}(\Z/q\Z)\col \trans N' X\equiv\bn\mod\ell'\bigr\}
\leq\frac{q^{dn}}{{\ell'\,}^n}
=\frac{q^{dn}}{\ell^n}\gcd(\ell,N)^n.
\end{align}
Recalling also \eqref{GLnZmodqbound},
we conclude:
\begin{align}\label{E1qtreatment2}
\scrA_\ell(N)\ll_{n,\ve} q^{(d-n)n+\ve}\ell^{-n}\gcd(\ell,N)^{n}.
\end{align}

Let us also note that,
by \autoref{RqcardLEM} and since $d>n$,
\begin{align}\label{E1qtreatment4}
\#\scrR_q=q^{dn}\prod_{p\mid q}\prod_{j=d+1-n}^d(1-p^{-j})
>q^{dn}\prod_{j=d+1-n}^d\zeta(j)^{-1}
\gg q^{dn}.
\end{align}
Using the bounds \eqref{E1qtreatment3}, \eqref{E1qtreatment2}
and \eqref{E1qtreatment4}
in \eqref{E1qtreatment1},
we obtain:
\begin{align*}
|E_{1, q}(f)|\ll_{\ve} S_{\infty,\lambda}(f) 
\,q^{-n+\ve}
\sum_{\substack{N\in {\rm M}_{d\times n}(\mathbb{Z})\\ N\neq \vecnull}}
\|N\|_{\infty}^{-\lambda}
\sum_{\ell\mid q} \gcd(\ell,N)^{n}.
\end{align*}
Recall that this holds for any real number $\lambda$ in the interval $0\leq\lambda\leq k$,
with $k$ as in \autoref{MAINTHM}.

Now note that for each positive integer $\ell$ we have\footnote{Here we work with
nonnegative sums taking values in $\R_{\geq0}\cup\{+\infty\}$;
note that a priori we may have
$\sum_{\substack{N\in {\rm M}_{d\times n}(\mathbb{Z})\\ N\neq \vecnull}}
\|N\|_{\infty}^{-\lambda}\gcd(\ell,N)^{n}=+\infty$;
however our computation shows that
$\sum_{\substack{N\in {\rm M}_{d\times n}(\mathbb{Z})\\ N\neq \vecnull}}
\|N\|_{\infty}^{-\lambda}\gcd(\ell,N)^{n}<\infty$ whenever $\lambda>dn$.}
\begin{align}\label{E1qtreatment5}
\sum_{\substack{N\in {\rm M}_{d\times n}(\mathbb{Z})\\ N\neq \vecnull}}
\|N\|_{\infty}^{-\lambda}\gcd(\ell,N)^{n}
\leq\sum_{\delta\mid\ell}
\sum_{\substack{N\in {\rm M}_{d\times n}(\mathbb{\delta Z})\\ N\neq \vecnull}}
\|N\|_{\infty}^{-\lambda}\delta^{n}
=\sum_{\delta\mid\ell}\delta^{n-\lambda}
\sum_{\substack{N'\in {\rm M}_{d\times n}(\mathbb{Z})\\ N'\neq \vecnull}}
\|N'\|_{\infty}^{-\lambda},
\end{align}
where we substituted $N=\delta N'$.
However, using the fact that
\begin{align}\label{countinginMdnZ}
\#\bigl\{N'\in\M_{d\times n}(\Z)\col \|N'\|_\infty=m\bigr\}
\leq 2dn(2m+1)^{dn-1}
\qquad (\forall m\in\Z_{\geq0}),
\end{align}
one verifies that the sum 
$\sum_{N'\neq\bn}\|N'\|_{\infty}^{-\lambda}$ is finite whenever $\lambda>dn$;
and in this case the expression in 
\eqref{E1qtreatment5}
is $\ll_{\lambda}\sum_{\delta\mid\ell}\delta^{n-\lambda}\ll_{\lambda} 1$,
since $n-\lambda<n-dn\leq-1$.
We may here choose $\lambda=dn+1$ (this is permissible since $dn+1\leq 2dn<k$),
and conclude:
\begin{align}\label{E1concl}
|E_{1, q}(f)|\ll_{\ve} S_{\infty,dn+1}(f)\, q^{-n+\ve}\sum_{\ell\mid q}1
\ll_{\ve} S_{\infty,dn+1}(f)\, q^{-n+2\ve}.
\end{align}

\subsection{Error term 2: $E_{2, q}(f)$}\label{E2qfSEC}
Recalling \eqref{e:E2q_def}, 
for any $N\in {\rm M}_{d\times n}(\mathbb{Z})$ we let 
\begin{equation}\label{e:E2qfN_def}
E_{2, q}(f;N) 
=  \frac{1}{\#\mathcal{R}_q} 
\sum_{\substack{M\in {\rm M}_{n\times d}(\mathbb{Z})\\ M\neq \vecnull}} 
\sum_{\gamma\in \mathcal{B}_{q}}
\widehat{F}\left(D_q \gamma; M, N\right)
K_n\biggl(\bigl( \bn\: I_n\bigr) \trans M , \trans N\gamma^{-1}\cmatr{\bn}{I_n};q\biggr),
\end{equation}
so that
\begin{equation}\label{E2qassum}
E_{2, q}(f) = \sum_{N\in {\rm M}_{d\times n}(\mathbb{Z})} E_{2, q}(f; N). 
\end{equation}

By \autoref{lem:Fouriercoeff_bound},
using also \eqref{hfNgdef} and \eqref{FgXNdef}, we have
\begin{equation}\label{e:widehatF_bound_qin}
\left|\widehat{F}\left(D_q \gamma; M, N\right)\right|
\ll \frac{S_{\infty,\lambda}(f)}{1+\|M \trans(D_q\gamma)^{-1}\|_{\infty}^{\lambda}},
\end{equation}
for any real number $\lambda$ in the interval $0\leq\lambda\leq k$.
Also, by \eqref{hatfbound} and \eqref{hatFdef}, 
\begin{align*}
\left|\widehat{F}\left(D_q \gamma; M, N\right)\right|
\ll \frac{S_{\infty,\lambda}(f)}{1+\|N\|_{\infty}^{\lambda}}.
\end{align*}
Hence for all $M\in\M_{n\times d}(\Z)$ and $N\in\M_{d\times n}(\Z)$, we have
\begin{align*}
\left|\widehat{F}\left(D_q \gamma; M, N\right)\right|
\ll\frac{S_{\infty,\lambda}(f)}{1+\|M \trans(D_q\gamma)^{-1}\|_{\infty}^{\lambda}+\|N\|_{\infty}^{\lambda}}.
\end{align*}
Using this bound in \eqref{e:E2qfN_def},
it follows that for every 
$N\in {\rm M}_{d\times n}(\mathbb{Z})$,
\begin{align}\notag
\left|E_{2, q}(f;N) \right|
&\ll \frac{S_{\infty,\lambda}(f) }{\#\mathcal{R}_q} 
\sum_{\substack{M\in {\rm M}_{n\times d}(\mathbb{Z})\\ M\neq \vecnull}} 
\sum_{\gamma\in \mathcal{B}_{q}}
\frac{\Bigl|K_n\Bigl(\bigl( \bn\: I_n\bigr) \trans M , \trans N\gamma^{-1}\scmatr{\bn}{I_n};q\Bigr)\Bigr|}
{1+\|M \trans (D_q\gamma)^{-1}\|_{\infty}^{\lambda}+\|N\|_{\infty}^{\lambda}}
\\\label{E2qtreatment1}
&\ll_{\ve}
\frac{S_{\infty,\lambda}(f) }{\#\mathcal{R}_q} 
\sum_{\substack{M\in {\rm M}_{n\times d}(\mathbb{Z})\\ M\neq \vecnull}} 
\sum_{\gamma\in \mathcal{B}_{q}}
\frac{q^{n^2-\vartheta+\ve}\gcd\bigl(q,\bigl(\bn\: I_n\bigr)\trans M \bigr)^{\vartheta}}
{1+\|M \trans (D_q\gamma)^{-1}\|_{\infty}^{\lambda}+\|N\|_{\infty}^{\lambda}},
\end{align}
where we recall that $\vartheta=n-1$ if $n\geq2$, $\vartheta=\frac12$ if $n=1$;
in the last step we used \eqref{KBOUNDmaintheoremRES2} in \autoref{KBOUNDmaintheorem}
if $n\geq2$, and the classical Weil bound if $n=1$.
Writing here $M=\bigl(M_0\hspace{7pt} M_1\bigr)$ with
$M_0\in\M_{n\times (d-n)}(\Z)$ and $M_1\in\M_{n}(\Z)$,
and setting $\ell:=\gcd\bigl(q,\bigl(\bn\hspace{6pt} I_n\bigr)\trans M\bigr)=\gcd(q,M_1)$
and $M_1':=\ell^{-1}M_1$, it follows that
\begin{align*}
\left|E_{2, q}(f;N) \right|
\ll_{\ve}S_{\infty,\lambda}(f)\frac{q^{n^2-\vartheta+\ve}}{\#\mathcal{R}_q} 
\sum_{\ell\mid q}\ell^{\vartheta}
\sum_{M_0\in\M_{n\times(d-n)}(\Z)}\sum_{\substack{M_1'\in\M_{n}(\Z)\\ M_0=\bn\Rightarrow M_1'\neq\bn}}
\hspace{50pt}
\\
\sum_{\gamma\in \mathcal{B}_{q}}
\Bigl(1+\bigl\|\bigl(M_0\hspace{7pt} \ell M_1'\bigr) D_q^{-1}\trans\gamma^{-1}\bigr\|_{\infty}^{\lambda}+\|N\|_{\infty}^{\lambda}\Bigr)^{-1}.
\end{align*}
Setting now $X:=\bigl(M_0 \hspace{5pt} M_1'\bigr)$
we have, using \eqref{Dqdefrep},
\begin{align*}
\bpm M_0 & \ell M_1'\ebpm D_q^{-1}
=\ell^{\frac nd}X D_{q/\ell}^{-1},
\end{align*}
and thus the last bound can be expressed as follows:
\begin{align}\notag
\bigl|E_{2, q} &(f;N) \bigr|
\\\label{e:E2qNbound_separate}
&\ll_{\ve}S_{\infty,\lambda}(f)\frac{q^{n^2-\vartheta+\ve}}{\#\mathcal{R}_q} 
\sum_{\ell\mid q}\ell^{\vartheta}
\sum_{\substack{X\in\M_{n\times d}(\Z)\\ X\neq\bn}}
\sum_{\gamma\in \mathcal{B}_{q}}
\Bigl(1+\|N\|_{\infty}^{\lambda}+\ell^{\frac nd \lambda}\bigl\|X D_{q/\ell}^{-1}\,\trans\gamma^{-1}\bigr\|_{\infty}^{\lambda}\Bigr)^{-1}.
\end{align}

Assuming from now on that $\lambda>nd$,
and using the majorant function $\Phi_{a,b}^{(\kappa)}$ introduced
in \eqref{e:Phiell_def},
the last bound can be expressed:
\begin{align}\label{e:E2qNbound_separate_rep}
\bigl|E_{2, q} (f;N) \bigr|
\ll_{\ve}S_{\infty,\lambda}(f)\frac{q^{n^2-\vartheta+\ve}}{\#\mathcal{R}_q} 
\sum_{\ell\mid q}\ell^{\vartheta}
\sum_{\gamma\in \mathcal{B}_{q}}
\Phi^{(\lambda)}_{1+\|N\|_{\infty}^\lambda,\ell^{\lambda n/d}}\bigl(D_{q/\ell}^{-1}\trans \gamma^{-1}\bigr).
\end{align}

We will need the following simple lemma.
\begin{lemma}\label{Heckesimplelem}
For any function $\Phi: {\rm SL}_d(\mathbb{Z}) \backslash {\rm SL}_d(\mathbb{R}) \to \mathbb{C}$,
$u\mid q$ and $g\in\SL_d(\R)$,
\begin{align}\label{HeckesimplelemRES}
\sum_{\gamma\in \mathcal{B}_{q}}\Phi\bigl(D_u^{-1}\,\trans\gamma^{-1} g\bigr)
=\#\scrB_q\cdot\bigl(T_{D_u}^*\Phi\bigr)(g).
\end{align}
\end{lemma}
\begin{proof}
We have
\begin{align*}
\sum_{\gamma\in \mathcal{B}_{q}}\Phi\bigl(D_u^{-1}\,\trans\gamma^{-1} g\bigr)
=\sum_{\gamma_1\in\Gamma^0(q)\bs\Gamma^0(u)}
\sum_{\gamma_2\in\scrB_u}\Phi\bigl(D_u^{-1}\,\trans(\gamma_1\gamma_2)^{-1}\, g\bigr).
\end{align*}
But $\gamma_1\in\Gamma^0(u)$ implies
$D_u\,\gamma_1 D_u^{-1}\in\SL_d(\Z)$;
hence
$D_u^{-1}\,\trans\gamma_1^{-1} D_u\in\SL_d(\Z)$
and
$\Phi\bigl(D_u^{-1}\,\trans(\gamma_1\gamma_2)^{-1}\, g\bigr)
=\Phi\bigl(D_u^{-1}\,\trans\gamma_2^{-1}\, g\bigr)$,
and so we get
\begin{align*}
\sum_{\gamma\in \mathcal{B}_{q}}\Phi\bigl(D_u^{-1}\,\trans\gamma^{-1} g\bigr)
=\#\bigl(\Gamma^0(q)\bs\Gamma^0(u)\bigr)
\sum_{\gamma_2\in\scrB_u}\Phi\bigl(D_u^{-1}\,\trans\gamma_2^{-1} g\bigr)
\hspace{140pt}
\\
=\#\bigl(\Gamma^0(q)\bs\Gamma^0(u)\bigr)
\cdot\#\scrB_u\cdot (T_{D_u}^*\Phi)(g)
=\#\scrB_q\cdot (T_{D_u}^*\Phi)(g),
\end{align*}
where the second equality holds by 
\eqref{TDqdualformula}.
\end{proof}
Using \autoref{Heckesimplelem}
and $\#\scrR_q=\#\scrB_q\cdot\#\GL_n(\Z/q\Z)$,
the bound in \eqref{e:E2qNbound_separate_rep}
can be rewritten as follows:
\begin{align}\label{E2qtreatment4}
\bigl|E_{2, q} (f;N) \bigr|
\ll_{\ve}S_{\infty,\lambda}(f)\frac{q^{n^2-\vartheta+\ve}}{\#\GL_n(\Z/q\Z)} 
\sum_{\ell\mid q}\ell^{\vartheta}\cdot
\Big[T^*_{D_{q/\ell}}\Phi^{(\lambda)}_{1+\|N\|_{\infty}^\lambda,\ell^{\lambda n/d}}\Bigr]\bigl(I_d\bigr).
\end{align}

We will bound 
$\bigl(T^*_{D_{q/\ell}}\Phi^{(\lambda)}_{a,b}\bigr)\bigl(I_d\bigr)$ from above by an integral over 
$\SL_d(\Z)\bs\SL_d(\R)$.
Fix a fundamental domain $\scrF_d$ for 
$\SL_d(\Z)\bs\SL_d(\R)$
containing $I_d$ in its interior,
and then fix an open neighbourhood $\Omega\subset\scrF_d$ of $I_d$ so small that
for every $w\in \Omega$ and every $\vecv\in \mathbb{R}^d$, 
\begin{equation}\label{Omegaass}
\frac{1}{2}\|\vecv\|_\infty \leq \|\vecv w\|_\infty \leq 2\|\vecv\|_\infty.
\end{equation}
This implies that for every $w\in \Omega$ and every $A\in {\rm M}_{n\times d}(\mathbb{R})$, 
\begin{equation}\label{e:size_Omega}
\frac{1}{2}\|A\|_\infty \leq \|A w\|_\infty \leq 2\|A\|_\infty. 
\end{equation}
Hence for any $a,b>0$, $g\in\SL_d(\R)$, $w\in\Omega$
and $X\in\M_{n\times d}(\R)$,
we have
\begin{align*}
2^{-\lambda}\bigl(a+b\|Xg\|_\infty^{\lambda}\bigr)
\leq a+b\|Xgw\|_\infty^{\lambda}
\leq 2^{\lambda}\bigl(a+b\|Xg\|_\infty^\lambda\bigr),
\end{align*}
and thus, recalling \eqref{e:Phiell_def},
we conclude that
\begin{equation}\label{E2qtreatment2}
2^{-\lambda} \Phi^{(\lambda)}_{a,b}(g)\leq \Phi^{(\lambda)}_{a,b}(gw) \leq 2^{\lambda}\Phi^{(\lambda)}_{a,b}(g). 
\end{equation}
Recalling now that, by \eqref{TDqdualformula},
\begin{align*}
\bigl[T^*_{D_{q/\ell}}\Phi\bigr](g)
=\frac1{\#\scrB_{q/\ell}}\sum_{\gamma\in\scrB_{q/\ell}}\Phi(D_{q/\ell}^{-1}\,\trans\gamma^{-1} g),
\end{align*}
and applying the left inequality in \eqref{E2qtreatment2} with $g=D_{q/\ell}^{-1}\,\trans\gamma^{-1}$ for each $\gamma\in\scrB_{q/\ell}$,
we conclude that
\begin{align*}
2^{-\lambda}\bigl[T^*_{D_{q/\ell}} \Phi^{(\lambda)}_{a,b}\bigr](I_d)
\leq \bigl[T^*_{D_{q/\ell}} \Phi^{(\lambda)}_{a,b}\bigr](w),
\qquad\forall w\in\Omega.
\end{align*}
Hence
\begin{align}\label{E2qtreatment5}
\bigl[T^*_{D_{q/\ell}}\Phi^{(\lambda)}_{a,b}\bigr](I_d)\leq
\frac{2^\lambda}{\int_{\Omega}d\mu_0(w)}\int_{\Omega}
\bigl[T^*_{D_{q/\ell}} \Phi^{(\lambda)}_{a,b}\bigr](w)\,d\mu_0(w)
\leq
\frac{2^\lambda}{\int_{\Omega}d\mu_0(w)}\int_{\scrF_d}\bigl[T^*_{D_{q/\ell}} \Phi^{(\lambda)}_{a,b}\bigr](g)\,d\mu_0(g),
\end{align}
where the last inequality holds since $\Phi^{(\lambda)}_{a,b}(g)>0$ everywhere.
It should be noted that in \eqref{E2qtreatment5} we are again working with
nonnegative sums and integrals taking values in $\R_{\geq0}\cup\{+\infty\}$;
a priori one or both of the integrals in \eqref{E2qtreatment5} may equal $+\infty$,
however we will see below that this is not the case.

Next, using \eqref{TDqdual} 
we have
$\langle T^*_{D_q}\Phi,1\rangle=\langle\Phi,T_{D_q}1\rangle=\langle\Phi,1\rangle$
for all $\Phi\in{\rm L}^2(\SL_d(\Z)\bs\SL_d(\R))$,
i.e.,
\begin{align}\label{E2qtreatment5aux}
\int_{\scrF_d}\bigl[T^*_{D_q}\Phi\bigr](g)\,d\mu_0(g)=\int_{\scrF_d}\Phi(g)\,d\mu_0(g).
\end{align}
It follows that \eqref{E2qtreatment5aux} also holds as a relation in
$\R_{\geq0}\cup\{+\infty\}$, for any left $\SL_d(\Z)$-invariant 
Borel measurable function $\Phi:\SL_d(\R)\to\R_{\geq0}$.
Using this fact in \eqref{E2qtreatment5} we conclude:
\begin{align}\label{gwconcl}
\bigl[T^*_{D_{q/\ell}}\Phi^{(\lambda)}_{a,b}\bigr](I_d)
\leq
\frac{2^\lambda}{\int_{\Omega}d\mu_0(w)}\int_{\scrF_d}\Phi^{(\lambda)}_{a,b}(g)\,d\mu_0(g),
\end{align}
and by \autoref{L1boundprop} (and since we are assuming $\lambda>nd$)
this implies
\begin{align}\label{gwconclnew}
\bigl[T^*_{D_{q/\ell}}\Phi^{(\lambda)}_{a,b}\bigr](I_d)
\ll_{\lambda}
a^{-1}\Bigl(\frac ab\Bigr)^{\!\frac d{\lambda}}\Bigl(1+\frac ab\Bigr)^{(n-1)\frac d{\lambda}}.
\end{align}

Using \eqref{GLnZmodqbound}
and \eqref{gwconclnew}
in \eqref{E2qtreatment4},
we obtain:
\begin{equation}
\left|E_{2, q}(f;N) \right|
\ll_{\lambda,\ve}
S_{\infty,\lambda}(f) 
q^{-\vartheta+2\ve}
\sum_{\ell\mid q}\ell^{\vartheta}\cdot
\bigl(1+\|N\|_\infty^{\lambda}\bigr)^{\frac d{\lambda}-1}\ell^{-n}\biggl(1+\frac{1+\|N\|_\infty^{\lambda}}{\ell^{\lambda n/d}}\biggr)^{(n-1)\frac d{\lambda}}.
\end{equation}
Hence, using also \eqref{E2qassum} and
\eqref{countinginMdnZ}
(with $m=a-1$), we have
\begin{align*}
\left|E_{2, q}(f) \right|
\ll_{\lambda,\ve}
S_{\infty,\lambda}(f) 
q^{-\vartheta+2\ve}
\sum_{\ell\mid q}\ell^{\vartheta-n}\cdot
\sum_{a=1}^\infty a^{dn-1}\cdot a^{d-\lambda}
\biggl(1+\frac{a}{\ell^{n/d}}\biggr)^{(n-1)d}.
\end{align*}
Here we must require $\lambda>2dn$ in order for the 
sum over $a$ to converge.
\textit{Assuming} $\lambda>2dn$,
the sum over $a$ is bounded independently of $\ell$, since
\begin{align*}
\sum_{a=1}^\infty a^{dn-1}\cdot a^{d-\lambda}
\biggl(1+\frac{a}{\ell^{n/d}}\biggr)^{(n-1)d}
\leq \sum_{a=1}^\infty (a+1)^{2dn-1-\lambda}
\ll_{\lambda} 1.
\end{align*}
Hence, using also $\vartheta<n$, we obtain:
\begin{align}\label{E2qtreatment6}
\left|E_{2, q}(f) \right|
\ll_{\lambda,\ve}
S_{\infty,\lambda}(f) 
q^{-\vartheta+3\ve}.
\end{align}
Setting here
$\ve=\frac13\ve_{\new}$
and then choosing $\lambda=\kappa+\ve_{\new}$ (recall that $\kappa=2dn$),
the bound becomes $S_{\infty,\kappa+\ve}(f)q^{-\vartheta+\ve}$.
Note that this bound subsumes the one in 
\eqref{E1concl}, 
since $\vartheta<n$ and $dn+1<\kappa$.
Hence, recalling that 
$\scrA_q(f)=E_{0,q}(f)+E_{1,q}(f)+E_{2,q}(f)$
(see \eqref{e:scrSqf_decomp})
and using \eqref{e:separating_main_scrSqf},
\eqref{E1concl}
and \eqref{E2qtreatment6},
we obtain \eqref{MAINTHMres},
i.e.\ we have proved \autoref{MAINTHM} in the case $n<d$.
\hfill$\square$

\newpage

\section{An application and a by-product}
In this section, we illustrate how both our result and its proof can be used to prove statements about solutions of Diophantine equations over a finite field $\F_p$, $p \geq 3$ prime, in small boxes (see also \cite{Shparlinski2015} for a comprehensive survey).

Indeed, Section \ref{ss:smallsolutions} %
is an application of our main theorem, \autoref{MAINTHM}, to estimating the probability that a randomly chosen (according to a rather general probability measure) system of affine congruences has a given number of small solutions.

In Section \ref{ss:byprod},  %
which is rather an application of the technique introduced in the proof of the main theorem, we give a sharp upper bound -- and prove the corresponding lower bound in a much more elementary way -- for the number of $\F_p$-points in small boxes on the variety of the set of (rectangular) matrices with a given rank.

\subsection{Application: small solutions of linear congruences}\label{ss:smallsolutions}

Let $p$ be an odd prime,
and consider the affine variety $V\subset\mathbb{A}^d$ 
defined by the system of equations
\begin{align}\label{congrsyst}
f_1(x_1,\ldots,x_d)=\cdots=f_n(x_1,\ldots,x_d)=0,
\end{align}
where $f_1,\ldots,f_n$ are polynomials in $\F_p[X_1,\ldots,X_d]$.
An important question is %
what can be said about existence
of $\F_p$-points of $V$, or about the number of $\F_p$-points of $V$,
inside a small ``box'' or more general small domain in $\F_p^d$; %
see, e.g., the recent survey \cite{Shparlinski2015}. %
In particular, a much studied problem is 
how small \textit{integer} solutions the system \eqref{congrsyst} has.
For a random choice of polynomials $f_1,\ldots,f_n$,
one expects the size of the smallest integer solution to typically be of
size comparable to $p^{n/d}$.

In \cite{StrombergssonVenkatesh2005},
this question was studied for a random system of
linear %
congruences.
It was proved in 
\cite{StrombergssonVenkatesh2005} that if 
the variety $V$ is taken uniformly random among
all linear, or all affine linear,
subspaces of $\F_p^d$ of codimension $n$,
then for any given nice subset $\Omega$ of $\R^d$,
as $p\to\infty$, 
there exists an explicit limit distribution
for the number of integer points which lie in $p^{n/d}\Omega$
(viz., are "small")
and which project 
to points in $V$.

In \autoref{smallsolTHM} below we prove a variant 
of these results, %
where instead 
the linear polynomials $f_1,\ldots,f_n$ in \eqref{congrsyst}
are taken random with respect to a given %
probability measure of a fairly general type: %
We %
take the constant terms of
$f_1,\ldots,f_n$ to be arbitrary \textit{fixed} integers $b_1,\ldots,b_n$,
while the %
tuple of 
degree one coefficients is chosen uniformly random
among all points $R$ in $\F_p^{dn}$
such that $p^{-1}R$ belongs to a given (nice) subset $U$ of the torus $(\R/\Z)^{dn}$
and the equations are linearly independent. 
In other words, we ask about the number of %
integer solutions $\vecx\in\Z^d$ of size $\ll p^{n/d}$ 
to the congruence equation
$\vecx R\equiv\vecb\mod p$,
for fixed $\vecb\in\Z^n$ and $R$ chosen 
uniformly random in the set $\{R\in\scrR_p\col p^{-1}R\in U\}$.
We will prove that, %
for any given nice subset $\Omega$ of $\R^d$, there exists an explicit limit distribution for the number of such solutions $\vecx$ in $\Z^d\cap p^{n/d}\Omega$, as $p\to\infty$.
In fact, our proof allows the modulus $p$ to run through all \textit{integers},
and we will state the theorem in this form, writing $q$ in place of $p$.

We say that a subset $\Omega$ of Euclidean space $\R^m$ or of the torus $(\R/\Z)^m$ ($m\geq1$) 
is \textit{smooth} if $\vol(\partial_{\ve}\Omega)\ll\ve$ as $\ve\to0$,
where $\partial_{\ve}\Omega$ is the $\ve$-neighborhood of the boundary of $\Omega$.
In the following we will view $\M_{d\times n}(\Z/q\Z)$ as a subset of 
$(\R/q\Z)^{dn}$; this means that for any subset $U\subset(\R/\Z)^{dn}$,
we can write $\scrR_q\cap qU$ for the set of all $R\in\scrR_q$ satisfying $q^{-1}R\in U$.

\begin{theorem}\label{smallsolTHM}
Let $d>n\geq1$; 
let $U$ be a smooth subset of $(\R/\Z)^{dn}$ of positive volume;
let $\Omega$ be a smooth and bounded subset of $\R^d$;
let $\vecb\in\Z^n$, and let $\ve>0$.
If $\vecb=\bn$ then we assume that $\Omega$ contains
a neighborhood of the origin.
Then for any $r\in\Z_{\geq0}$ there exists a constant 
$c(\Omega,\vecb,r)\geq0$ such that,
for any positive integer $q$,
the number of $R\in\scrR_q\cap qU$ such that the congruence equation
$\vecx R\equiv\vecb\mod q$ has exactly $r$ solutions 
$\vecx$ in $\Z^d\cap q^{n/d}\Omega$
is
\begin{align}\label{smallsolTHMres1}
\#\bigl(\scrR_q\cap qU\bigr) %
\cdot \bigl(c(\Omega,\vecb,r)+O_{U,\Omega,\vecb,r,\ve}(q^{-\alpha+\ve})\bigr),
\end{align}
where $\alpha=\alpha(d,n)=
{\displaystyle \min\Bigl(\frac{n-1}{1+2dn},\frac{n}{d^2},\frac{d-n}{d^2}\Bigr)}$
if $n\geq2$, and 
$\alpha(d,1)={\displaystyle\min\Bigl(\frac{1}{2+4d},\frac1{d^2}\Bigr)}$.
\end{theorem}

In order to state the explicit formula for the 
limit probabilities $c(\Omega,\vecb,r)$ in \autoref{smallsolTHM}, %
let $\ASL_d(\R)$ be the affine special linear group of order $d$,
that is, $\ASL_d(\R)=\SL_d(\R)\ltimes\R^d$ with multiplication law
\begin{align*}
(g,\vecv)(g',\vecv')=(gg',\vecv g'+\vecv')
\qquad \bigl( g,g'\in\SL_d(\R),\:\vecv,\vecv'\in\R^d %
\bigr).
\end{align*}
(Note that $\ASL_d(\R)$ is isomorphic with our group ${\rm H}$ in the special case $n=1$.)
The group $\ASL_d(\R)$  %
acts on $\R^d$ from the right through
$\vecx(g,\vecv):=\vecx g+\vecv$ ($\vecx\in\R^d$).
We identify the homogeneous space $\ASL_d(\Z)\bs\ASL_d(\R)$
with the space of \textit{grids} (=translates of lattices) of covolume one in $\R^d$,
through $\ASL_d(\Z)g\leftrightarrow \Z^dg$ ($g\in\ASL_d(\R)$),
and we denote by $\mu$ the invariant probability measure on 
$\ASL_d(\Z)\bs\ASL_d(\R)$.
We take $\SL_d(\R)$ to be embedded in $\ASL_d(\R)$ through $g\mapsto(g,\bn)$;
thus $\SL_d(\Z)\bs\SL_d(\R)$ becomes identified in the standard way
with the space of lattices of covolume one in $\R^d$. %
Recall that $\mu_0$ denotes the 
invariant probability measure on
$\SL_d(\Z)\bs\SL_d(\R)$.
Now we have:
\begin{align}\label{smallsolTHMres2}
c(\Omega,\vecb,r)=\begin{cases}
\mu_0(\{g\in\SL_d(\Z)\bs\SL_d(\R)\col\#(\Z^dg\cap\Omega)=r\})
&\text{if }\: \vecb=\bn;
\\
\mu(\{g\in\ASL_d(\Z)\bs\ASL_d(\R)\col\#(\Z^dg\cap\Omega)=r\})
&\text{if }\: \vecb\neq\bn.
\end{cases}
\end{align}
In particular note that for $\vecb\neq\bn$,
$c(\Omega,\vecb,r)$ is independent of $\vecb$!

\begin{remark}\label{SVspecREM}
The formulas for the limit probabilities in 
\eqref{smallsolTHMres2} are the same as those in
\cite{StrombergssonVenkatesh2005}.
In fact, in the case $\vecb=\bn$,
by specializing to $U=(\R/\Z)^{dn}$
and restricting $q$ to run through primes,
\autoref{smallsolTHM} gives back %
\cite[Theorem 2]{StrombergssonVenkatesh2005} but with a weaker error term.
Indeed, as $R$ runs through $\scrR_p$,
the set $\{\vecx\in\F_p^d\col \vecx R=\bn\}$ runs through all the linear subspaces of $\F_p^d$
of codimension $n$, visiting each such subspace exactly
$\prod_{j=1}^{n-1}(p^n-p^j)$ times.
Similarly, %
the limit result of
\cite[Theorem 3]{StrombergssonVenkatesh2005} (without an error term)
follows \textit{formally} by applying
\autoref{smallsolTHM} with $q=p$ prime, $U=(\R/\Z)^{dn}$,
and averaging over all $\vecb$ in $\F_p^n$;
this is of course not a rigorous deduction,
since $\vecb$ is required to be a fixed \textit{integer} vector in
\autoref{smallsolTHM},
and the error term in \eqref{smallsolTHMres1} is allowed to depend on $\vecb$
in an uncontrolled way.
\end{remark}

\begin{remark}\label{smallsolTHMgenREM}
As we will see, the proof of \autoref{smallsolTHM}
can easily be extended to
give the following more general statement:
Let $d,n,U$ be as in \autoref{smallsolTHM};
let $k\in\Z^+$, 
and let $\Omega_1,\ldots,\Omega_k$ be smooth
and bounded subsets of $\R^d$.
Let $\vecb_1,\ldots,\vecb_k\in\Z^n$;
for each $j$ such that $\vecb_j=\bn$,
we assume that $\Omega_j$ contains a neighbourhood of the origin.
Let $r_1,\ldots,r_k\in\Z_{\geq0}$.
Then for any $q\in\Z^+$,
the number of $R\in\scrR_q\cap qU$ such that for each $j=1,\ldots,k$,
the equation
$\vecx R\equiv\vecb_j\mod q$ has exactly $r_j$ solutions 
$\vecx\in\Z^d\cap q^{n/d}\Omega_j$, is
\begin{align}\label{smallsolTHMgenREMres}
\#\bigl(\scrR_q\cap qU\bigr) \cdot \bigl(c+O(q^{-\alpha+\ve})\bigr),
\end{align}
where $\alpha=\alpha(d,n)$ is as before,
$c\in\R_{\geq0}$ is a constant which depends on 
$\vecb_1,\ldots,\vecb_k$,
$\Omega_1,\ldots,\Omega_k$
and $r_1,\ldots,r_k$
(see \eqref{smallsolTHMgenREMresCONST} below),
and where the implied constant in the ``big $O$''
may depend on
$U$, $\vecb_1,\ldots,\vecb_k,\Omega_1,\ldots,\Omega_k,r_1,\ldots,r_k,\ve$.
\end{remark}

\begin{remark}\label{smallsolTHMneqdREM}
In the case $d=n$ %
we have $\scrR_q=\GL_n(\Z/q\Z)$
for every $q\in\Z^+$,
and hence for any $\vecb\in\Z^n$,
the equation $\vecx R\equiv\vecb\mod q$ has a unique
solution $\vecx=\vecb R^{-1}$ in $(\Z/q\Z)^n$.
We now have the following result analogous to \autoref{smallsolTHM}:
For any smooth subsets $U\subset (\R/\Z)^{n^2}$
and $\Omega\subset(\R/\Z)^n$ of positive volume,
and any $\vecb\in\Z^n\setminus\{\bn\}$ and 
$q\in\Z^+$, %
the number of $R\in\scrR_q\cap qU$
such that $\vecx=\vecb R^{-1}$ lies in $q\Omega$ equals
\begin{align*}
\#\bigl(\scrR_q\cap qU\bigr)\vol(\Omega)(1+O_{U,\Omega,\vecb,\ve}(q^{-\beta+\ve})),
\end{align*}
where $\beta=\beta(n)=\frac{n-1}{1+2n^2}$ if $n\geq2$, and $\beta(1)=\frac16$.
We give the proof at the end of the present section.
In analogy with Remark \ref{smallsolTHMgenREM},
the above result may also be generalized into an asymptotic formula 
for the number of 
$R\in\scrR_q\cap qU$
such that $\vecx=\vecb_j R^{-1}$ lies in $q\Omega_j$ for each $j=1,\ldots,k$.
\end{remark}

We now start preparing for the proof of \autoref{smallsolTHM}.
For each $\vecb\in\R^n$ we denote by $J_{\vecb}$ the following 
Lie group homomorphism:
\begin{align}\label{JbDEF}
J_{\vecb}:{\rm H}\to\ASL_d(\R),\qquad J_{\vecb}\matr Z{\bn}V{I_n}=(Z,\vecb V).
\end{align}
If $\vecb\in\Z^n$ then $J_{\vecb}(\Gamma\cap {\rm H})\subset\ASL_d(\Z)$,
and hence $J_{\vecb}$ induces a smooth map
\begin{align*}
\tJ_{\vecb}:\Gamma\bs\Gamma {\rm H}\to\ASL_d(\Z)\bs\ASL_d(\R).
\end{align*}

\begin{lemma}\label{smallsolTHMLEM1}
For any $q\in\Z^+$, $R\in\scrR_q$, $\vecb\in\Z^n$,
and any subset $\Omega\subset\R^d$, 
the number of solutions $\vecx\in\Z^d\cap q^{n/d}\Omega$
to the congruence equation
$\vecx R\equiv\vecb\mod q$ equals
$\#\bigl(\Z^d \tJ_{\vecb}\bigl(\tn_+(q^{-1}R)D(q)\bigr)\cap\Omega\bigr)$.
\end{lemma}
(Here for any $\alpha\in\ASL_d(\Z)\bs\ASL_d(\R)$
we write ``$\Z^d \alpha$'' for the corresponding grid;
thus $\Z^d \alpha:=\Z^d g$ for any $g\in\ASL_d(\R)$ such that $\alpha=\ASL_d(\Z) g$.)
\begin{proof}
Let $R'$ be a lift of $R$ to $\M_{d\times n}(\Z)$.
We know from \autoref{fScharLEM} %
that $\tn_+(q^{-1}R)D(q)\in\Gamma\bs\Gamma {\rm H}$;
hence there exists some
$\gamma\in\Gamma$ 
such that $\gamma\, n_+(q^{-1}R')D(q)\in {\rm H}$.
Writing $\gamma=\matr ABCD$ 
(with $A\in\M_d(\Z)$, $B\in\M_{d\times n}(\Z)$, $C\in\M_{n\times d}(\Z)$, $D\in\M_{n}(\Z)$),
we then have   %
\begin{align}\label{smallsolTHMLEM1pf1}
\matr ABCD \matr{q^{-\frac nd}I_d}{R'}{\bn}{qI_n}
=\matr{q^{-\frac nd}A}{\bn}{q^{-\frac nd}C}{I_n}.
\end{align}
In particular we have $AR'+qB=\bn$ in $\M_{d\times n}(\mathbb{Z})$,
and this implies that the lattice $\Z^dA$ is contained in the kernel of the homomorphism
$\vecx\mapsto[\vecx R'\mod q]$ from $\Z^d$ to $\Z^n/q\Z^n$.
This homomorphism is surjective since $R\in\scrR_q$;
hence the kernel is a subgroup of index $q^n$ in $\Z^d$;
furthermore, \eqref{smallsolTHMLEM1pf1} implies $\det A=q^n$,
so that also $\Z^dA$ has index $q^n$ in $\Z^d$.
Hence $\Z^dA$ in fact \textit{equals} the kernel:
\begin{align}\label{smallsolTHMLEM1pf2}
\Z^dA=\{\vecx\in\Z^d\col \vecx R'\equiv \bn\mod q\}.
\end{align}
Also from \eqref{smallsolTHMLEM1pf1} we have
$CR'+qD=I_n$; hence $\vecb CR'\equiv\vecb\mod q$.
This fact combined with \eqref{smallsolTHMLEM1pf2} implies
\begin{align*}
\{\vecx\in\Z^d\col \vecx R\equiv \vecb\mod q\}=\Z^d A+\vecb C.
\end{align*}
Note also that \eqref{smallsolTHMLEM1pf1} implies
$\Z^d\tJ_{\vecb}\bigl(\tn_+(q^{-1}R)D(q)\bigr)
=q^{-\frac nd}(\Z^dA+\vecb C).$
The lemma follows from the last two facts.
\end{proof}

\begin{lemma}\label{JBpushforwardLEM}
For any $\vecb\in\Z^n\setminus\{\bn\}$,
$\mu_{\rm H}\circ \tJ_{\vecb}^{-1}=\mu$.
On the other hand, for $\vecb=\bn$ we have
$\tJ_{\bn}(\Gamma\bs\Gamma {\rm H})=\SL_d(\Z)\bs\SL_d(\R)$
and $\mu_{\rm H}\circ\tJ_{\bn}^{-1}=\mu_0$.
\end{lemma}

\begin{proof}
As before, let $\scrF_d\subset\SL_d(\R)$ be a fundamental domain for 
$\SL_d(\Z)\bs\SL_d(\R)$;
also let $\tmu_0$ be the Haar measure on $\SL_d(\R)$
which induces the measure $\mu_0$ on $\SL_d(\Z)\bs\SL_d(\R)$.
Then the measure $\mu_{\rm H}$ can be explicitly described as follows:
For any Borel set $B\subset\Gamma\bs\Gamma {\rm H}$,
\begin{align}\label{JBpushforwardLEMpf1}
\mu_{\rm H}(B)=\int_{\scrF_d}\vol\biggl(\biggl\{X\in\M_{n\times d}(\R/\Z)\col \tn_-(X)\matr g{\bn}{\bn}{I_n}\in B\biggr\}\biggr)
\,d\tmu_0(g),
\end{align}
where $\vol$ is the Lebesgue measure on the torus $\M_{n\times d}(\R/\Z)\cong(\R/\Z)^{dn}$.
Similarly, if we introduce the map $\iota:(\R/\Z)^d\to\ASL_d(\Z)\bs\ASL_d(\R)$
by setting $\iota(\vecx):=\ASL_d(\Z)(I_d,\vecx')$ where $\vecx'$ is any lift to $\R^d$ of $\vecx\in(\R/\Z)^d$,
then for any Borel set $A\subset\ASL_d(\Z)\bs\ASL_d(\R)$,
\begin{align}\label{JBpushforwardLEMpf2}
\mu(A)=\int_{\scrF_d}\vol\bigl(\bigl\{\vecx\in(\R/\Z)^d\col \iota(\vecx)(g,\bn)\in A\bigr\}\bigr)\,d\tmu_0(g),
\end{align}
where now $\vol$ also denotes the Lebesgue measure on $(\R/\Z)^d$.
Assuming $\vecb\in\Z^d\setminus\{\bn\}$,
our task is to prove that 
\begin{align}\label{JBpushforwardLEMpf4}
\mu_{\rm H}(\tJ_{\vecb}^{-1}(A))=\mu(A)
\end{align}
holds for any Borel set $A\subset\ASL_d(\Z)\bs\ASL_d(\R)$.
This follows using the formulas
\eqref{JBpushforwardLEMpf1}
and \eqref{JBpushforwardLEMpf2},
together with the fact that
\begin{align*}
\tn_-(X)\matr g{\bn}{\bn}{I_n}\in\tJ_{\vecb}^{-1}(A)
\Leftrightarrow
\iota(\vecb X)\cdot(g,\bn)\in A
\qquad (\forall X\in\M_{n\times d}(\R/\Z),\: g\in\SL_d(\R)),
\end{align*}
and the fact that
for any Borel set $A'\subset(\R/\Z)^d$, 
\begin{align}\label{JBpushforwardLEMpf3}
\vol\bigl(\bigl\{X\in\M_{n\times d}(\R/\Z)\col \vecb X\in A'\bigr\}\bigr)=\vol(A').
\end{align}

In the remaining case, $\vecb=\bn$,
the statements of the lemma
are immediate from \eqref{JbDEF} and \eqref{JBpushforwardLEMpf1}.
\end{proof}

As in \cite{StrombergssonVenkatesh2005},
we introduce a notion of smoothness for subsets of arbitrary homogeneous spaces,
as follows:
Let $X=\Lambda\bs{L}$ where ${L}$ is a Lie group and $\Lambda$ a lattice in ${L}$,
and let $\tmu$ be the ${L}$-invariant probability measure on $X$.
(We will apply the following to the three cases 
$X=(\Gamma\cap {\rm H})\bs {\rm H}$, $X=\ASL_d(\Z)\bs\ASL_d(\R)$
and $X=\SL_d(\Z)\bs\SL_d(\R)$.)
We fix a left invariant Riemannian metric $\dd$ on ${L}$.
This metric descends to a Riemannian metric on $X=\Lambda\bs{L}$,
which we also denote by $\dd$,
and using this metric,
for any subset $\Omega\subset X$ and any $\ve>0$,
we define the $\ve$-neighborhood of the boundary of $\Omega$,
\begin{align*}
\partial_{\ve}\Omega:=\bigl\{p\in X\col \bigl[\exists q\in \partial\Omega\:\text{ s.t. }\dd(p,q)<\ve\bigr]\bigr\}.
\end{align*}
Now the set $\Omega$ is said to be \textit{smooth}
if $\tmu(\partial_{\ve}\Omega)\ll\ve$ as $\ve\to0$.

Next, for any subset $\Omega\subset\R^d$ and any $r\in\Z_{\geq0}$,
we let $\tOmega_r$ be the subset of $\ASL_d(\Z)\bs\ASL_d(\R)$ corresponding to those grids
of covolume one in $\R^d$
which intersect $\Omega$ in exactly $r$ points, viz.,
\begin{align}\label{tOmegarDEF}
\tOmega_r=\bigl\{\ASL_d(\Z)g\col g\in\ASL_d(\R),\:\#(\Z^dg\cap\Omega)=r\bigr\}.
\end{align}
\begin{lemma}\label{JbinvtOmegarISSMOOTHlem}
For any smooth subset $\Omega\subset\R^d$, any $r\in\Z_{\geq0}$
and any $\vecb\in\Z^n\setminus\{\bn\}$,
$\tJ_{\vecb}^{-1}(\tOmega_r)$ is a smooth subset of $\Gamma\bs\Gamma {\rm H}$.
\end{lemma}
\begin{proof}
As in \cite[Lemma 10]{StrombergssonVenkatesh2005},
one proves that $\tOmega_r$ is a smooth subset of $\ASL_d(\Z)\bs\ASL_d(\R)$.
Next, it is an immediate verification from \eqref{JbDEF}
that 
$\|d\!J_{\vecb}(\vecv)\|\leq C \|\vecv\|$
holds for any point $h\in {\rm H}$ and any tangent vector $\vecv\in T_h{\rm H}$,
where the two norms are
the Riemannian norms on $T_{J_{\vecb}(h)}(\ASL_d(\R))$,
and on $T_h{\rm H}$, respectively,
and where $C$ is a positive constant
which is independent of $h$ and $\vecv$.
It follows that
\begin{align*}
\dd(J_{\vecb}(h_1),J_{\vecb}(h_2))\leq C\,\dd(h_1,h_2),
\qquad\forall h_1,h_2\in {\rm H},
\end{align*}
and this, in turn, implies that
\begin{align*}
\dd(\tJ_{\vecb}(p_1),\tJ_{\vecb}(p_2))\leq C\,\dd(p_1,p_2),
\qquad\forall p_1,p_2\in(\Gamma\cap {\rm H})\bs {\rm H}.
\end{align*}

Now let $\ve>0$ be given. Then for any point $p\in \partial_{\ve}\bigl(\tJ_{\vecb}^{-1}(\tOmega_r)\bigr)$
there exist points $p_1,p_2\in\Gamma {\rm H}\bs {\rm H}$
satisfying $p_1\in \tJ_{\vecb}^{-1}(\tOmega_r)$,
$p_2\notin \tJ_{\vecb}^{-1}(\tOmega_r)$,
$\dd(p,p_1)<\ve$ and $\dd(p,p_2)<\ve$.
It follows that $\tJ_{\vecb}(p_1)\in\tOmega_r$,
$\tJ_{\vecb}(p_2)\notin\tOmega_r$,
$\dd(\tJ_{\vecb}(p_1),\tJ_{\vecb}(p))<C\,\ve$
and $\dd(\tJ_{\vecb}(p_2),\tJ_{\vecb}(p))<C\,\ve$;
and hence $\tJ_{\vecb}(p)\in\partial_{C\,\ve}\tOmega_r$.
We have thus proved:
\begin{align}
\partial_{\ve}\bigl(\tJ_{\vecb}^{-1}(\tOmega_r)\bigr)
\subset \tJ_{\vecb}^{-1}\bigl(\partial_{C\,\ve}\tOmega_r\bigr).
\end{align}
Hence, using also \autoref{JBpushforwardLEM} and the fact that
$\tOmega_r$ is smooth,
\begin{align*}
\mu_{\rm H}\bigl(\partial_{\ve}\bigl(\tJ_{\vecb}^{-1}(\tOmega_r)\bigr)\bigr)
\leq \mu_{\rm H}\bigl(\tJ_{\vecb}^{-1}\bigl(\partial_{C\,\ve}\tOmega_r\bigr)\bigr)
=\mu\bigl(\partial_{C\,\ve}\tOmega_r\bigr)
\ll \ve.
\end{align*}
Hence $\tJ_{\vecb}^{-1}(\tOmega_r)$ is smooth.
\end{proof}

\begin{remark}\label{bdependenceREM}
One verifies that the constant $C$ in the proof of the previous lemma
can be taken to be $C_1\|\vecb\|$,
where $\|\vecb\|$ is the Euclidean norm of $\vecb$ %
and where $C_1$ is a constant which only depends on
the Riemannian metrics on ${\rm H}$ and $\ASL_d(\R)$.
\end{remark}

The following is the analogue of \autoref{JbinvtOmegarISSMOOTHlem}
in the case $\vecb=\bn$.
\begin{lemma}\label{JbinvtOmegarISSMOOTHbzerolem}
For any smooth subset $\Omega\subset\R^d$ which contains a neighbourhood of the origin,
and for any $r\in\Z_{\geq0}$,
$\tJ_{\bn}^{-1}(\tOmega_r)$ is a smooth subset of $\Gamma\bs\Gamma {\rm H}$.
\end{lemma}
\begin{proof}
Since $\tJ_{\bn}$
maps $\Gamma\bs\Gamma{\rm H}$ into 
$\SL_d(\Z)\bs\SL_d(\R)$,
we have 
$\tJ_{\bn}^{-1}(\tOmega_r)=\tJ_{\bn}^{-1}(\tOmega_r')$
where
\begin{align*}
\tOmega_r'=\bigl\{\SL_d(\Z)g\col g\in\SL_d(\R),\:\#(\Z^dg\cap\Omega)=r\bigr\}.
\end{align*}
This set $\tOmega_r'$ is a smooth subset of $\SL_d(\Z)\bs\SL_d(\R)$,
by \cite[Lemma 4]{StrombergssonVenkatesh2005}.
Now the proof of \autoref{JbinvtOmegarISSMOOTHlem}
carries over to the present case.
\end{proof}

\begin{proof}[Proof of \autoref{smallsolTHM}]
Let $U,\Omega,\vecb,r$ be given as in the statement of the theorem.
Let $\chi_1:\M_{d\times n}(\R/\Z)\to\{0,1\}$ be the characteristic function of 
$U$; let $\chi_2:\Gamma\bs\Gamma {\rm H}\to\{0,1\}$ be the characteristic function
of $\tJ_{\vecb}^{-1}(\tOmega_r)$,
and let $\chi:\M_{d\times n}(\R/\Z)\times\Gamma\bs\Gamma {\rm H}\to\{0,1\}$
be the characteristic function of $U\times\tJ_{\vecb}^{-1}(\tOmega_r)$.
Then by \autoref{smallsolTHMLEM1} and \eqref{tOmegarDEF},
for any $q\in\Z^+$,
the number of $R\in\scrR_q\cap qU$ such that the equation
$\vecx R\equiv\vecb\mod q$ has exactly $r$ solutions 
$\vecx\in\Z^d\cap q^{n/d}\Omega$
is
\begin{align}\label{smallsolTHMpf1}
\sum_{R\in\scrR_q}\chi\bigl(q^{-1}R,\tn_+(q^{-1}R)D(q)\bigr).
\end{align}
Since $U$ is smooth by assumption, and 
$\tJ_{\vecb}^{-1}(\tOmega_r)$ is smooth by \autoref{JbinvtOmegarISSMOOTHlem}
or \autoref{JbinvtOmegarISSMOOTHbzerolem},
it follows from the proof of
\cite[Lemma 1]{StrombergssonVenkatesh2005}
that for every $0<\delta\leq\frac12$ there exist functions $f_{1,\pm}\in{\rm C}^{\infty}(\M_{d\times n}(\R/\Z))$
and $f_{2,\pm}\in{\rm C}^{\infty}(\Gamma\bs\Gamma {\rm H})$ 
satisfying %
\begin{align*}
&0\leq f_{j,-}\leq \chi_j\leq f_{j,+}\leq 1
\qquad\text{and}\qquad %
S_{\infty,k}(f_{j,\pm})\ll_k\delta^{-k}
\end{align*}
for $j=1,2$ and any real $k\geq0$,
and also
\begin{align*}
\vol\Bigl(\overline{\{A %
\col f_{1,\pm}(A)\neq\chi_1(A)\}}\Bigr)\ll\delta
\qquad\text{and}\qquad
\mu_{\rm H}\Bigl(\overline{\{\Gamma h\col f_{2,\pm}(\Gamma h)\neq\chi_2(\Gamma h)\}}\Bigr)\ll\delta.
\end{align*}
Define the two functions 
$f_{\pm}\in{\rm C}_b^{\infty}({\rm M}_{d\times n}(\mathbb{R}/\mathbb{Z}) \times \Gamma \backslash \Gamma {\rm H})$
through 
$f_{\pm}(T,p)=f_{1,\pm}(T)f_{2,\pm}(p)$.
Then
\begin{align*}
&0\leq f_{-}\leq \chi\leq f_{+}\leq 1;
\hspace{30pt}
S_{\infty,k}(f_{\pm})\ll_k\delta^{-k};
\hspace{30pt}
S_{2,k'}(f_{\pm})\ll_k\delta^{\frac12-k'}
\end{align*}
for any real $k\geq0$ and $k'\geq1$.
Recalling the definition \eqref{SqfDEF},
it follows that the sum in \eqref{smallsolTHMpf1} is bounded from above by
$\#\scrR_q\cdot\scrA_q(f_+)$,
and by \autoref{MAINTHM} this equals
\begin{align}\notag
&\#\scrR_q\biggl(\int_{{\rm M}_{d\times n}(\mathbb{Z}\backslash\mathbb{R})} f_{1,+}\,dT
\int_{\Gamma\backslash \Gamma {\rm H}} f_{2,+}\,d\mu_{\rm H}
+O\Bigl(S_{\infty,\kappa+\ve}(f_+) q^{-\vartheta+\ve}
+S_{2,\kappa'+\ve}(f_+)q^{-\vartheta'+\ve}\Bigr)\biggr)
\\\label{smallsolTHMpf2}
&=\#\scrR_q\biggl(\vol(U)\mu_{\rm H}(\tJ_{\vecb}^{-1}(\tOmega_r))+O(\delta)+
O\Bigl(\delta^{-\kappa-\ve}q^{-\vartheta+\ve}+\delta^{\frac12-\kappa'-\ve}q^{-\vartheta'+\ve}\Bigr)\biggr).
\end{align}
Similarly, 
the sum in \eqref{smallsolTHMpf1} is bounded from below by
$\#\scrR_q\cdot\scrA_q(f_-)$,
which is again estimated by %
the right-hand side of
\eqref{smallsolTHMpf2}.
It follows that also the 
sum in \eqref{smallsolTHMpf1} itself
is estimated by %
the right-hand side of
\eqref{smallsolTHMpf2}.
Note here that by \autoref{JBpushforwardLEM},
$\mu_{\rm H}(\tJ_{\vecb}^{-1}(\tOmega_r))=c(\Omega,\vecb,r)$,
the constant defined in \eqref{smallsolTHMres2}.
We now optimize by choosing 
$\delta=q^{-\alpha}$
with $\alpha=\alpha(d,n):=\min\bigl(\frac{\vartheta}{1+\kappa},\frac{\vartheta'}{\frac12+\kappa'}\bigr)$.
Using also $\vol(U)>0$, it follows that 
\begin{align}\label{smallsolTHMpf3}
\sum_{R\in\scrR_q}\chi\bigl(q^{-1}R,\tn_+(q^{-1}R)D(q)\bigr)
=\#\scrR_q\cdot\vol(U)\bigl(c(\Omega,\vecb,r)+O\bigl(q^{-\alpha+\ve'}\bigr)\bigr),
\end{align}
where $\ve'$ depends on $\ve$,
with $\ve'\to0$ as $\ve\to0$. %
From now on we write $\ve$ in place of $\ve'$.
Let us note that the exponent $\alpha=\alpha(d,n)$
here satisfies the formula
stated in \autoref{smallsolTHM};
this is immediate from the formulas for $\kappa,\vartheta,\kappa',\vartheta'$ in
\autoref{MAINTHM},
where in the special case $d=2$, $n=1$ we make
use of the fact that $\theta\leq\frac7{64}<\frac3{10}$.

It remains to prove that we can replace the factor 
$\#\scrR_q\cdot\vol(U)$ in \eqref{smallsolTHMpf3} by
$\#(\scrR_q\cap qU)$.
To this end, note first that by 
repeating the above argument but with $\chi_2\equiv f_{2,\pm}\equiv 1$,
we obtain
\begin{align*}
\#(\scrR_q\cap qU)=\#\scrR_q\cdot\vol(U)\bigl(1+O\bigl(q^{-\alpha+\ve}\bigr)\bigr),
\end{align*}
that is, there exists a constant $C=C(U,\ve)>0$ %
such that
\begin{align*}
\#\scrR_q\cdot\vol(U)\bigl(1-Cq^{-\alpha+\ve}\bigr)\leq
\#(\scrR_q\cap qU)\leq \#\scrR_q\cdot\vol(U)\bigl(1+Cq^{-\alpha+\ve}\bigr).
\end{align*}
We have $\bigl(1+Cq^{-\alpha+\ve}\bigr)^{-1}\geq 1-O(q^{-\alpha+\ve})$
and so
$\#\scrR_q\cdot\vol(U)\geq\#(\scrR_q\cap qU)\bigl(1-O(q^{-\alpha+\ve})\bigr)$;
and if $Cq^{-\alpha+\ve}\leq\frac12$ then also
$\bigl(1-Cq^{-\alpha+\ve}\bigr)^{-1}\leq 1+O(q^{-\alpha+\ve})$,
allowing us to conclude
$\#\scrR_q\cdot\vol(U)=\#(\scrR_q\cap qU)\cdot\bigl(1+O(q^{-\alpha+\ve})\bigr)$.
Using the last estimate in
\eqref{smallsolTHMpf3} gives
\eqref{smallsolTHMres1},
and even when $Cq^{-\alpha+\ve}>\frac12$
we conclude that \eqref{smallsolTHMres1}
is a valid bound from below
on the quantity in \eqref{smallsolTHMpf1}.
Note also that the statement around \eqref{smallsolTHMres1}
holds trivially if $\#(\scrR_q\cap qU)=0$;
hence from now on we may assume $\#(\scrR_q\cap qU)\geq1$.
Now to complete the proof, note that 
$Cq^{-\alpha+\ve}>\frac12$
implies $q\ll1$,
hence $\#\scrR_q\ll1$,
and so
$\#\scrR_q\cdot\vol(U)\leq 1+O(q^{-\alpha+\ve})
\leq \#(\scrR_q\cap qU)\cdot(1+O(q^{-\alpha+\ve}))$,
provided that we take the implied constant sufficiently large.
Using the last inequality in \eqref{smallsolTHMpf3} gives
the desired upper bound.
\end{proof}

\begin{proof}[Proof of the statement in Remark \ref{smallsolTHMgenREM}]
The proof of \autoref{smallsolTHM}
carries over, with essentially the only
difference being that $\chi_2:\Gamma\bs\Gamma {\rm H}\to\{0,1\}$
is now taken to be the characteristic function
of the intersection $\cap_{j=1}^k \tJ_{\vecb_j}^{-1}(\tOmega_{j,r_j})$,
where $\tOmega_{j,r_j}=
\bigl\{\ASL_d(\Z)g\col g\in\ASL_d(\R),\:\#(\Z^dg\cap\Omega_j)=r_j\bigr\}$.
Now by a property valid in arbitrary metric spaces, %
$\partial\bigl(\cap_{j=1}^k \tJ_{\vecb_j}^{-1}(\tOmega_{j,r_j})\bigr)
\subset\cup_{j=1}^k \partial \bigl(\tJ_{\vecb_j}^{-1}(\tOmega_{j,r_j})\bigr)$,
and hence
$\partial_{\ve}\bigl(\cap_{j=1}^k \tJ_{\vecb_j}^{-1}(\tOmega_{j,r_j})\bigr)
\subset\cup_{j=1}^k \partial_{\ve}\bigl( \tJ_{\vecb_j}^{-1}(\tOmega_{j,r_j})\bigr)$
for every $\ve$.
Hence, using the fact that
each set $\tJ_{\vecb_j}^{-1}(\tOmega_{j,r_j})$
is smooth, it follows that also the intersection
$\cap_{j=1}^k \tJ_{\vecb_j}^{-1}(\tOmega_{j,r_j})$ is smooth.
The rest of the proof is essentially the same as before,
and we obtain \eqref{smallsolTHMgenREMres} with
\begin{align}\label{smallsolTHMgenREMresCONST}
c=\mu_H\biggl(\bigcap_{j=1}^k \tJ_{\vecb_j}^{-1}(\tOmega_{j,r_j})\biggr).
\end{align}
\end{proof}

\begin{proof}[Proof of the statement in Remark \ref{smallsolTHMneqdREM}]
Recall that when $d=n$, the map $\tn_-$ (see \eqref{tnnDEF}) gives an identification
between $\M_n(\R/\Z)$ and $\Gamma\bs\Gamma{\rm H}$.
Let $m_{\vecb}$ be the 'multiplication' map from 
$\Gamma\bs\Gamma{\rm H}$ to $(\R/\Z)^n$ given by
$m_{\vecb}(\tn_-(A))=\vecb A$ for all $A\in\M_n(\R/\Z)$.
Let $\chi_1:\M_n(\R/\Z)\to\{0,1\}$ be the characteristic function of $U$;
let $\chi_2:\Gamma\bs\Gamma{\rm H}\to\{0,1\}$ be the characteristic function
of $m_{\vecb}^{-1}(\Omega)$, and let $\chi:\M_n(\R/\Z)\times\Gamma\bs\Gamma{\rm H}\to\{0,1\}$ 
be the characteristic function of $U\times m_{\vecb}^{-1}(\Omega)$.
Recall that $\tn_+(q^{-1}R)D(q)=\tn_-(q^{-1}R^{-1})$
for every $R\in\scrR_q$, by \autoref{lem:mtx_relation}.
Hence the number of $R\in\scrR_q\cap qU$
such that $\vecx=\vecb R^{-1}$ lies in $q\Omega$
is now again given by the sum in \eqref{smallsolTHMpf1}.
One verifies that $\mu_{\rm H}\circ m_{\vecb}^{-1}=\vol$,
the Lebesgue measure on $(\R/\Z)^n$,
and by an argument as in \autoref{JbinvtOmegarISSMOOTHlem},
$m_{\vecb}^{-1}(\Omega)$ is a smooth subset of $\Gamma\bs\Gamma{\rm H}$.
Now the proof of \autoref{smallsolTHM}
carries over to the present case.
\end{proof}

\subsection{By-product: counting matrices} \label{ss:byprod}
The next theorem gives an optimal bound on the following quantity,
for any given $1\leq r<n<d$, any prime $p$ and integer $1\leq b\leq \frac{p-1}2$:
\begin{equation} \label{e:to_count}
N_{p,b}:=\#\{ X \in {\rm M}_{d\times n}(\mathbb{Z}) : \| X \|_\infty \leq b \text{ and } \rank(X \bmod p) = r \}.
\end{equation} 
The proof of this bound is a by-product of the proof of
our main result, \autoref{MAINTHM};
in particular it uses an interpretation in terms of Hecke operators,
and Rogers' formula (\autoref{lem:Rogers}).

\begin{theorem}\label{byprodTHM}
Let $1\leq r<n<d$. For every prime $p$ and integer $1\leq b\leq \frac{p-1}2$,
\begin{equation}\label{byprodTHMres}
N_{p,b}\asymp_d \max\bigl(b^{dr},b^{dn}p^{-(d-r)(n-r)}\bigr).
\end{equation}
\end{theorem}

\begin{remark}
The same %
counting problem was considered by Ahmadi and Shparlinski in \cite[Theorem 9]{AhShS2007}. They were however interested in obtaining asymptotics, which they did through results ultimately relying on Deligne-type methods for estimating the number of $\F_p$-points on varieties. Their large $p$ asymptotics are non-trivial in the range where $b$ is large, specifically -- with our notation -- whenever $b \geq p^{\gamma_{r, n, d} + \varepsilon}$ for some positive $\varepsilon$, where 
\[
\gamma_{r, n, d} = \max\left( \frac 12 + \frac{(n - r)(d - r)}{2nd}, 1 - \frac{1}{2(n - r)(d -r) + 2} \right).
\]
It should be noted that their result is valid for arbitrary $n,d\geq1$.

Our method yields an upper bound of the correct order of magnitude for arbitrary $p$ and $b$;
however we are not able to handle the case of square matrices ($n=d$).
This stems from the application of Rogers' formula in our approach;
\autoref{lem:Rogers} is only valid under the assumption $n<d$.
\end{remark}

\begin{proof} 
The main work will be spent on proving that \eqref{byprodTHMres} gives a valid
bound from above on $N_{p,b}$. %
To start, let $\pi_p \colon \mathbb{Z}^d \twoheadrightarrow \mathbb{F}_p^d$ be the canonical projection;
denote by $\Gr_{r,d}(\F_p)$ the space of $r$-dimensional linear subspaces of $\F_p^d$, 
and by $X_d$ the space $\SL_d(\Z) \backslash \SL_d(\R)$.
Observe that for a linear subspace $V$ of $\F_p^d$ with dimension $r$, $\pi_p^{-1}(V)$ is a sublattice of $\Z^d$ whose covolume is $p^{d-r}$, hence $p^{\frac {r-d}d} \pi_p^{-1}(V)$ is a unimodular lattice in $\R^d$.
We may thus introduce the map
\begin{align*}
\phi_p \colon \Gr_{r,d}(\mathbb{F}_p) \to X_d,
\qquad\phi_p(V)=p^{\frac rd - 1} \pi_p^{-1}(V).
\end{align*}
It follows that an upper bound for $N_{p,b}$ is given by
\begin{align*}
 &\sum_{V \in \Gr_{r,d}(\mathbb{F}_p)} \#\{ (\vecv_1, \ldots, \vecv_n) \in (\pi_p^{-1}(V) \cap [-b,b]^d)^n : \dim \mathrm{Span}_\mathbb{R}(\{\vecv_1, \ldots, \vecv_n\}) \geq r \}) \\
&= \sum_{V \in \Gr_{r,d}(\mathbb{F}_p)} \#\{ (\vecv_1, \ldots, \vecv_n) \in (\phi_p(V) \cap b\,\tilde{C})^n : \dim \mathrm{Span}_\mathbb{R}(\{\vecv_1, \ldots, \vecv_n\}) \geq r \}),
\end{align*}
where $\tilde{C}$ is the cube $\tilde{C}=p^{\frac rd-1}[-1,1]^d\subset\R^d$.

At this point we recall the connection between the Grassmannian over $\F_p$ and lattices, namely that there is a 
bijection between $\Gr_{r,d}(\F_p)$ and the lattices $p \Z^d \subset L \subset \Z^d$ of index $p^{d-r}$.
Furthermore, the family of such lattices can be used to define a Hecke operator:
Set
\begin{align*}
D_p'=p^{\frac{r}d-1}\matr{I_r}{}{}{pI_{d-r}}\in\SL_d(\R),
\end{align*}
and introduce, as in Section \ref{ss:Heckepts},  
the Hecke operator $T_{D_p'}$,
acting on functions on $X_d$.
It then follows from \cite[Lemma 3.13]{Shi94} that
\begin{align*}
(T_{D_p'}\Phi)(L)=\frac 1{\# \Gr_{r,d}(\F_p)} \sum_{\substack{pL \subset L' \subset L \\ [L : L'] = p^{d-r}}} \Phi\left(p^{\frac rd -1} L' \right),
\end{align*}
for any $\Phi:X_d\to\mathbb{C}$.

Hence:
\[
N_{p,b} \leq \# \Gr_{r,d}(\F_p) \bigl[T_{D_p'}(F_p)\bigr](\Z^d),
\]
where $F_p \colon X_d \to \Z_{\geq 0}$
is defined, for $L \in X_d$, by 
\[
F_p(L) = \# \{ (\vecv_1, \ldots, \vecv_n) \in (L \cap b\,\tilde{C})^n : \dim \Span_\R{(\vecv_1, \ldots, \vecv_n)} \geq r \}.
\]

We now proceed as in the proof of our main theorem, specifically the part after \eqref{E2qtreatment4}, whose role is now played by the above inequality.

As in that proof, let $\Omega$ be an open neighbourhood of the identity matrix $I_d$ in $\SL_d(\R)$
such that \eqref{Omegaass} holds for all $w\in \Omega$ and $\vecv\in\R^d$;
it then follows that 
\[ 
\forall w \in \Omega:\qquad N_{p,b} \leq \#\Gr_{r,d}(\F_p) \bigl[T_{D_p'}(\tilde{F}_p)\bigr](\Z^dw),
\]
where
\begin{align*}
\tilde{F}_p \colon X_d &\to \Z_{\geq 0}
\\ 
L &\mapsto \#\{(\vecv_1, \ldots, \vecv_n) \in (L \cap 2b\, \tilde{C})^n : \dim \Span_\R{\{\vecv_1, \ldots, \vecv_n\}} \geq r \}.
\end{align*}
Hence 
\begin{align}
    N_{p,b} &\ll_d \#\Gr_{r, d}(\F_p) \int_{X_d} [T_{D_p'}(\tilde{F}_p)](L)\, d\mu_0(L)
= \#\Gr_{r, d}(\F_p) \int_{X_d} \tilde{F}_p\, d\mu_0.
\end{align}
(Recall that $\mu_0$ is the $\SL_d(\R)$-invariant probability measure on $X_d$.)
The last upper bound can be rewritten as 
\begin{equation} \label{e:Ncp_bound}
N_{p,b} \ll_d \#\Gr_{r, d}(\F_p) \int_{X_d} \sum_{(\vecv_1, \ldots, \vecv_n) \in L^n} \chi_{p,b}(\vecv_1, \ldots, \vecv_n)\, d \mu_0(L),
\end{equation}
where 
$\chi_{p,b} \colon (\R^d)^n \to \{0, 1 \}$ is the characteristic function of the set 
$\{(\vecv_1, \ldots, \vecv_n) \in (2b\, \tilde{C})^n :  \dim \Span_\R{\{\vecv_1, \ldots, \vecv_n\}} \geq r \}$.
We now rewrite the integrand in such a way that we can apply \autoref{lem:Rogers}:
writing $L = \Z^d g$ for some $g \in \SL_d(\R)$, it is 
\begin{align*}
&\sum_{(m_1, \ldots, m_n) \in (\Z^d)^n} \chi_{p,b}(m_1 g, \ldots, m_n g)
= \sum_{X \in {\rm M}_{n \times d}(\Z)} \chi_{p,b}(Xg)
= \sum_{\substack{X \in {\rm M}_{n \times d}(\Z)\\ X\neq\bn}} \chi_{p,b}(Xg).
\end{align*}
We are now in a position to use \autoref{lem:Rogers} and deduce that 
the integral in \eqref{e:Ncp_bound} is equal to
\[
\sum_{m=1}^n \sum_{B \in A_{n,m}} \int_{{\rm M}_{m \times d}(\R)} \chi_{p,b}(BX)\, dX,
\]
where we recall that the matrix $B$ has rank $m$,
and so $BX$ has rank $\leq m$.
By the definition of $\chi_{p,b}$, it follows that the integrand vanishes whenever $m \leq r-1$, so the sum is equal to
\[
\sum_{m=r}^n \sum_{B \in A_{n,m}} \int_{{\rm M}_{m \times d}(\R)} \chi_{p,b}(BX)\, dX.
\]
If we now define $\chi \colon {\rm M}_{n \times d}(\R) \to \{0, 1\}$ to be the characteristic function of the set of matrices $M \in {\rm M}_{n \times d}(\R)$ such that $\| M \|_\infty \leq 1$ and $\rank{M} \geq r$, we have, 
using
$2b\, \tilde{C} = 2bp^{\frac rd-1} [-1, 1]^d$: 
\begin{align}\notag
\sum_{B \in A_{n,m}} \int_{{\rm M}_{m \times d}(\R)} \chi_{p,b}(BX)\, dX 
&= \sum_{B \in A_{n,m}} \int_{{\rm M}_{m \times d}(\R)} \chi((2b)^{-1}p^{1-\frac rd}BX)\, dX
\\\notag 
&= (2bp^{\frac rd-1})^{md} \sum_{B \in A_{n,m}} \int_{{\rm M}_{m \times d}(\R)} \chi(BX)\, dX 
\ll_d b^{md} p^{-m(d-r)}.
\end{align}
In the last step we used the fact that 
$\sum_{B\in A_{n,m}}\int_{{\rm M}_{m \times d}(\R)} \chi(BX)\, dX\leq\int_{X_d}\sum_{X\in\M_{n\times d}(\Z)}\chi(Xg)\,d\mu_0(g)
<\infty$,
by \autoref{lem:Rogers}
and \cite[Theorem 2]{wS58}.
Plugging the last bound back into \eqref{e:Ncp_bound}, 
we finally obtain
\begin{align*}
N_{p,b} \ll_d \#\Gr_{r,d}(\F_p) \sum_{m = r}^n (b^d p^{-(d-r)})^m \ll_d p^{r(d-r)} 
\max\bigl((b^d p^{-(d-r)})^r,(b^d p^{-(d-r)})^n\bigr)
\\
=\max\bigl(b^{dr},b^{dn}p^{-(d-r)(n-r)}\bigr),
\end{align*}
i.e.\ we have proved that \eqref{byprodTHMres} gives a valid upper bound on
$N_{p,b}$.

\vspace{5pt}

To finish, we prove that the same expression is also a lower bound on $N_{p,b}$; 
it should be noted that 
this proof 
is completely elementary.  
As a first step we note that 
\begin{align}\label{byprodTHMlbpf1}
\#\{Y\in\M_r(\Z)\col\|Y\|_\infty\leq b\text{ and }\det Y\not\equiv0\mod p\}> b^{r^2}.
\end{align}
This is proved by induction:
First, by immediate inspection (using $b\geq1$), we have 
\begin{align}\label{byprodTHMlbpf2}
\#\{y\in\Z\cap[-b,b]\col y\not\equiv a\mod p\}> b
\qquad (\forall a\in\Z).
\end{align}
This fact, applied with $a=0$, means that \eqref{byprodTHMlbpf1} holds for $r=1$.
Next, for $r\geq2$, 
write $Y=(y_{ij})\in\M_{r}(\Z)$,
and let $Y'$ be the top left $(r-1)\times(r-1)$ submatrix of $Y$. 
Then by expanding $\det Y$ along the bottom row,
we have
$\det Y=y_{r,r}\cdot \det Y'+h$,
where $h$ is an integer which is independent of $y_{r,r}$.
Hence for any fixed choice of $Y'$ with $\|Y'\|_\infty\leq b$ and $\det Y'\not\equiv0\mod p$,
and any fixed choice of the entries $y_{r,i}$ and $y_{i,r}$ ($i=1,\ldots,r-1$),
there is some $a\in\Z$ such that 
$\det Y\equiv0\mod p$ holds if and only if $y_{r,r}\equiv a\mod p$;
and so by \eqref{byprodTHMlbpf2} there are more than $b$ choices of $y_{r,r}\in\Z\cap[-b,b]$
which make $\det Y\not\equiv0\mod p$.
Since the number of choices of $Y'$ as above is $\geq b^{(r-1)^2}$ 
(by induction), and each entry $y_{r,i}$ and $y_{i,r}$ ($i=1,\ldots,r-1$)
can be chosen in more than $b$ ways,
it follows that \eqref{byprodTHMlbpf1} holds.

Note that any 
matrix $X\in\M_{d\times n}(\Z)$
with $\|X\|_\infty\leq b$
whose top left $r\times r$ submatrix 
has determinant $\not\equiv0\mod p$
and whose last $n-r$ columns vanish identically,
belongs to the set in \eqref{e:to_count}.
Hence \eqref{byprodTHMlbpf1} immediately implies that
\begin{align}\label{byprodTHMlbpf3}
N_{p,b}\geq b^{r^2}\cdot b^{r(d-r)}=b^{dr}.
\end{align}

Next we will prove that we also have
$N_{p,b}\gg b^{dn}p^{-(d-r)(n-r)}$.
Let 
\begin{align*}
B':=(\Z\cap[-\tfrac12b,\tfrac12b])^d
\quad\text{and}\quad
B:=(\Z\cap[b,b])^d,
\end{align*}
and note that $\#(\Z\cap[-\tfrac12b,\tfrac12b])>\frac12b$ and hence
$\#B'>(\tfrac12 b)^d$.
We claim that for every vector subspace $V\subset\F_p^d$ of dimension $r$,
\begin{align}\label{lowerboundpf1}
\#(B\cap \pi_p^{-1}(V))>(\tfrac12 b)^dp^{r-d}.
\end{align}
To prove this,
set $H:=\max\{\#\alpha^{-1}(\vecw)\col \vecw\in \F_p^d/V\}$,
where $\alpha$ is the projection map from $B'$ to $\F_p^d/V$.
Then 
$(\tfrac12 b)^d<\#B'\leq \#(\F_p^d/V)\cdot H=p^{d-r}H$,
and so $H>(\frac12b)^dp^{r-d}$.
But the definition of $H$ implies that there exist
$H$ distinct vectors $\vecv_1,\ldots,\vecv_H$ in $B'$
satisfying $\pi_p(\vecv_1)\equiv\cdots\equiv\pi_p(\vecv_H)\mod V$.
It follows that $\vecv_1-\vecv_i$ for $i=1,2,\ldots,H$
are $H$ distinct vectors lying in $B\cap \pi_p^{-1}(V)$,
and hence $\#(B\cap \pi_p^{-1}(V))\geq H>(\frac12b)^dp^{r-d}$,
i.e.\ \eqref{lowerboundpf1} is proved.

Now let us construct matrices $X$ belonging to the set in 
\eqref{e:to_count} as follows:
First choose the left $d\times r$ submatrix $X'$ of $X$ to have all entries in $\Z\cap[-b,b]$
and full rank mod $p$.
By the argument giving \eqref{byprodTHMlbpf3},
this choice can be made in $\geq b^{dr}$ ways. 
Let $V\subset\F_p^d$ be the span of the columns of $X'$ reduced mod $p$.
Finally, pick each remaining column of $X$ as an arbitrary vector in
$B\cap \pi_p^{-1}(V)$.
By \eqref{lowerboundpf1}, these columns can be chosen in more than
$\bigl((\tfrac12 b)^dp^{r-d}\bigr)^{n-r}$ ways,
and our construction guarantees that $X$ belongs to the set in 
\eqref{e:to_count}.
Hence
\begin{align}\label{byprodTHMlbpf4}
N_{p,b}>b^{dr}\cdot\bigl((\tfrac12 b)^dp^{r-d}\bigr)^{n-r}
\end{align}
Together, 
\eqref{byprodTHMlbpf3} and \eqref{byprodTHMlbpf4}
imply the desired lower bound,
$N_{p,b} \gg_d \max\bigl(b^{dr},b^{dn}p^{-(d-r)(n-r)}\bigr)$
(with the implied constant being $2^{-d(n-r)}$).
\end{proof}

\thispagestyle{empty}
{\footnotesize
\bibliographystyle{amsalpha}
\bibliography{bibliography}
}
\end{document}